\documentclass[12pt, a4paper, reqno]{amsart}
\usepackage{tikz}
\usetikzlibrary{arrows.meta, positioning, fit}
\usepackage{amssymb,amsmath,amsfonts,amsthm}
\usepackage{yhmath}
\usepackage{url}
\usepackage{preamble}
\usepackage{tikz-cd}
\usepackage{color}
\usepackage{todonotes}

\allowdisplaybreaks
\newcommand{\tf}{{\tilde f}}

\newcommand{\tg}{{\tilde g}}

\newcommand{\tA}{{\mathcal A}}
\newcommand{\tG}{{\mathcal G}}
\newcommand{\tQ}{{\mathcal Q}}
\newcommand{\tB}{{\mathcal B}}

\newcommand{\Diff}{\mathrm{Diff}^1}
\newcommand{\td}{\mathrm{d}}
\newcommand{\gs}{\geqslant}
\newcommand{\ls}{\leqslant}
\newcommand{\e}{\mathrm{e}}
\newcommand{\rk}{\mathrm{k}}
\newcommand{\m}{\mathrm{m}}
\newcommand{\csubset}{\underset{\circ}{\subseteq}}
\newcommand{\csupset}{\underset{\circ}{\supseteq}}
\newcommand{\ceq}{\overset{\circ}{=}}
\newcommand{\cneq}{\overset{\circ}{\neq}}
\newcommand{\orb}{\mathrm{Orb}}
\newcommand{\M}{\mathcal{M}}

\newtheorem{corollary}[theorem]{Corollary}

\DeclareMathOperator*{\essinf}{ess\,inf}
\DeclareMathOperator*{\esssup}{ess\,sup}

\def\supp{\operatorname{supp}}
\def\diam{\operatorname{diam}}

\title
[Continuity properties of partial entropy]
{ Continuity properties of partial entropy }
\begin{document}
\begin{abstract}
 We establish a general criterion on the upper semi-continuity of partial entropy in all directions for $C^{1+\alpha}$ diffeomorphisms: it holds when the respective sums of Lyapunov exponents are continuous. This addresses, in arbitrary dimensions, the converse aspect of the entropic continuity of the Lyapunov exponents established by Buzzi, Crovisier, and Sarig \cite{bcs2022}. Consequently, the entropy (and all the partial entropies) is always upper semi-continuous at generic ergodic measures of every $C^{1+\alpha}$ diffeomorphism, which extends the $C^{\infty}$ result of Newhouse \cite{New89}. 

Numerous applications and examples are provided, including topics related to measures with dominated splittings, SRB measures, average expanding diffeomorphisms, singular flows, standard maps, and symbolic codings for diffeomorphisms.
\end{abstract}

\maketitle

\tableofcontents

\section{Introduction}\label{sec:intro}

In the study of chaotic dynamical systems, the quantification and stability of complexity
invariants remain a central challenge. Such invariants usually link geometric properties of typical orbits (e.g., asymptotic counting and expansion) to probabilistic structures (e.g., invariant measures) of the system. Among them, Kolmogorov-Sinai metric entropy and Lyapunov
exponents stand out as foundational objects: metric entropy captures complexity from probabilistic and information-theoretic perspectives, while Lyapunov exponents characterize
the exponential rates of expansion or contraction of tangent vectors along orbits.

While metric entropy and Lyapunov exponents both serve as quantitative measures of chaos, the precise relations between these two quantities were established through a series of seminal works by Margulis, Ruelle \cite{1978Ruelle-inequality}, Pesin \cite{Pesin-entropyformula_1977}, and Ledrappier and Young \cite{LEDRAPPIER_YOUNG_A, LEDRAPPIER_YOUNG_B}.  Their stability against measure perturbations is demonstrated    in their respective continuity properties (for entropy \cite{Bowenexpansiveness, mis1976, New89, BuzziIsrael, Downarowicz2004, burguet2012} and for the Lyapunov exponent \cite{avilaviana2010, bochi2022, bockeviana2017, Furstenberg1983, ruelle1979, viana2020}).   Nevertheless, how the stabilities of entropy and Lyapunov exponents depend on each other has remained unknown for a long time. Only recently did Buzzi, Crovisier, and Sarig \cite{bcs2022} establish a fundamental form of entropic continuity for the Lyapunov exponent:
$$\textbf{\it Continuity of entropy }\quad \Longrightarrow\quad  \text{\it Continuity of Lyapunov exponent,}$$  on the set of ergodic measures with entropy bounded away from zero, with respect to $C^{\infty}$ surface diffeomorphisms. 

In this paper, we study the converse of the entropic continuity of Lyapunov exponents. It is noteworthy that the entropy is possibly not lower semi-continuous. For example, periodic measures can approximate any ergodic hyperbolic measure of positive entropy (Katok \cite{katok1980}), with continuous approximations of the Lyapunov exponents (Wang and Sun \cite{Sunwang2010}). Therefore, the reasonable problem is as follows:
\[
\textbf{\it Continuity of Lyapunov exponent }
\quad \stackrel{\textbf{\normalsize ?}}{\Longrightarrow} \quad
\textbf{\it Upper semi-continuity of entropy }.
\]

In this work, within the arbitrary dimensional setting and only under the $C^{1+\alpha}$ regularity, we obtain a positive answer to this question for entropy along all partial directions. 
The entropy along all partial directions, namely the partial entropy, constitutes a refinement of the metric entropy introduced by Ledrappier and Young in \cite{LEDRAPPIER_YOUNG_A, LEDRAPPIER_YOUNG_B}. It is a crucial tool for characterizing the distribution of dynamical complexity over Lyapunov subspaces or along strong-unstable directions. The results in this paper provide a new criterion for the upper semi-continuity of partial entropy.

Among the various formulations regarding the upper semi-continuity of the partial entropy given in \Cref{sec:1.1}, we present the following version herein:

\begin{theorem}\label{first}
  Let $f$ be a $C^{1+\alpha}$ diffeomorphism of a compact Riemannian manifold without boundary. Consider a sequence of ergodic measures $\{\mu_n\}_{n\in\mathbb N}$ converging to an ergodic measure $\mu$ in the weak$^{*}$ topology. For any $\rk\in\mathbb N$, if the sum of the top $\rk$ Lyapunov exponents of $\mu_n$ converges to that of $\mu$, then the corresponding partial entropies (whenever defined) satisfy $$\limsup_{n\to\infty} h_{\mu_n}^{\rk}(f) \ls h_{\mu}^{\rk}(f).$$
\end{theorem}
The present mechanism persists under perturbations of the system and along the sequence of invariant measures, as we can see in \Cref{thm bestcor}. In addition, we establish quantitative estimates for the upper semi-continuity defect of partial entropies in the setting where the sum of Lyapunov exponents lacks continuity; see \Cref{thm best}.  See \Cref{sec:1.1} for more detailed discussions.

The upper semi-continuity ensures that small perturbations of a system and its invariant measures cannot induce sudden, unbounded increases in its partial entropy. When combined with the variational principle (\cite{Walter, wuweisheng-u-entropy, HuWuZhu21}), it guarantees the existence of equilibrium states and therefore serves as an essential property for understanding the stochastic structure of chaotic systems and their stability.

Regarding the total entropy, a remarkable result of Newhouse \cite{New89} establishes the upper semi-continuity of the metric entropy for $C^{\infty}$ diffeomorphisms. In addition, the defect of upper semi-continuity is also investigated in the context of expansivity \cite{Bowenexpansiveness, mis1976, yomdin1987, BuzziIsrael} and in low-dimensional settings \cite{burguet2012, burgertliao, dawei-luo-2025USC,tianxueting2026USC}. Nevertheless, this property may fail at certain ergodic measures in lower regularity classes: counterexamples exist for $C^r$ systems with $r<\infty$, including Misiurewicz's four-dimensional examples \cite{mis1973}, the surface diffeomorphisms constructed by  Downarowicz and Newhouse \cite{New05}, and Buzzi \cite{Buzzi2014}.  In contrast to these counterexamples, our main theorem shows that—apart from  exceptional cases—upper semi-continuity holds at every generic ergodic measure in the $C^{1+\alpha}$ setting.

\begin{corollary}\label{sec}
Let $f$ be a $C^{1+\alpha}$ diffeomorphism of a compact Riemannian manifold without boundary, and $\mu$ be a generic ergodic measure of $f$. Consider any sequence $\{\nu_n\}_{n\gs1}$ of invariant measures converging to $\mu$ in the weak$^{*}$ topology.  
Then the metric entropies satisfy
$$\limsup_{n\to \infty} h_{\nu_n}(f) \ls h_{\mu}(f).$$ 
\end{corollary}

\subsection{Main results}\label{sec:1.1}

Let $f$ be a $C^{1+\alpha}$ diffeomorphism of a compact $\td$-dimensional Riemannian manifold $M$ without boundary preserving a Borel probability measure $\mu$.
The Oseledets Multiplicative Ergodic Theorem \cite{Oseledets} guarantees that for $\mu$-almost every $x\in M$, there exist Lyapunov exponents $\hat{\lambda}_1(x, f) > \hat{\lambda}_2(x, f) > \dots > \hat{\lambda}_{r(x)}(x, f)$ and a corresponding decomposition of the tangent space $T_xM = \bigoplus_{i=1}^{r(x)} E_i(x)$ such that for any $0 \neq v \in E_i(x)$,
$$\lim_{n\to\infty} \frac{1}{n} \log \|Df^n(v)\| = \hat{\lambda}_i(x, f).$$
Taking multiplicities into account, we denote all Lyapunov exponents (counted with multiplicity) by $\lambda_1(x, f)\geqslant \lambda_2(x,f)\geqslant \dots \geqslant \lambda_{\td}(x, f)$. When $\mu$ is ergodic, the Lyapunov exponents are constant $\mu$-almost everywhere and are denoted by $\hat{\lambda}_i(\mu, f)$ ($1 \ls i \ls r(\mu)$) and $\lambda_i(\mu, f)$ ($1 \ls i \ls \td$). 

The Margulis–Ruelle inequality \cite{1978Ruelle-inequality} states that the metric entropy of $\mu$ is bounded above by the sum of its positive Lyapunov exponents. Moreover, Pesin \cite{Pesin-entropyformula_1977} proved that the equality holds if $\mu$ is absolutely continuous.

Ledrappier and Young \cite{LEDRAPPIER_YOUNG_A} further demonstrated that
measures absolutely continuous along unstable manifolds—known as Sinai–Ruelle–Bowen measures—are precisely the ones for which the Pesin entropy formula holds.
For general measures, they introduced in \cite{LEDRAPPIER_YOUNG_B} the partial entropy $\hat{h}_\mu^i(x, f)$, which is associated with the subspace $F_i(x) = E_1(x) \oplus \cdots \oplus E_i(x)$ corresponding to Lyapunov exponents at least $\hat{\lambda}_i(x, f)$. It is  the entropy of any measurable partition $\xi^i$ subordinate to the unstable manifold $W^i$ relative to $F_i$  (such partitions always exist), defined by
$$\hat{h}_\mu^i(x, f):=H_{\mu_x}(\xi^i\mid f\xi^i),$$
with  $\mu_x$ the ergodic component of $x$. 
Taking multiplicities into account, set $$\rk = \dim E_1(x)+\dots+\dim E_i(x)$$ and let $$h^{\rk}_{\mu}(x, f)=\hat{h}^i_{\mu}(x, f).$$ When $\mu$ is ergodic, $\hat{h}_\mu^i(x, f)=\hat{h}_\mu^i(f)$ and $h^{\rk}_{\mu}(x, f)=h^{\rk}_{\mu}(f)$ are constants. 
Using this, they established the  general entropy formula that incorporates the fractal transverse dimensions in \cite[Theorem C]{LEDRAPPIER_YOUNG_B}.

In this context, our results on the upper semi-continuity of partial entropies clarify how the continuous variation of Lyapunov exponents governs the stability of partial complexity. Within the Ledrappier-Young framework, this ensures that the complexity remains controlled when the dynamical system or its measures are perturbed.
The precise formulations of these results are presented in the following subsections, organized sequentially from sequences of ergodic measures to sequences of invariant measures and quantitative defect estimates.

\subsubsection{Along sequences of ergodic measures}$\,$\smallskip

Denote by $\mathrm{Diff^{1+\alpha}}(M)$ the set of $C^{1+\alpha}$ diffeomorphisms on a compact Riemannian manifold $M$ without boundary.     For any $f\in \mathrm{Diff^{1+\alpha}}(M)$, $f$-invariant  measure $m$ and  $\rk\in\mathbb N$, denote the $\rk$-dimensional exponent gap function as
$$\mathrm{Hyp}^{\mathrm{k}}(m):=\essinf_{x\sim m}\{\lambda_\rk(x,f)-\max\{\lambda_{\rk+1}(x,f),0\}\},$$
where $\essinf_{x\sim m} f(x)$ denotes the essential infimum of $f$ with respect to the measure \(m\). Denote 
$$\lambda_{\rk}(m,f)=\int_M \lambda_{\rk}(f,x) \, \td m(x).$$
If $\mathrm{Hyp}^{\mathrm{k}}(m)>0$, the $\rk$-dimensional partial entropy $$h^\rk_m(f):=\int_{M} h^\rk_{m_x}(f)\,\td m(x)$$ is well-defined, where $m_x$ is the ergodic component of $m$, corresponding to $x$.

\begin{theorem}\label{thm 3.3}
    Let $f_n\in\mathrm{Diff}^{1+\alpha}(M)$ and $\mu_n$ be an ergodic measure of $f_n$. Assume that 
    \begin{itemize}
    \item $f_n\xrightarrow{C^1} f\in \mathrm{Diff}^{1+\alpha}(M)$
    with $\sup_{n\in \mathbb N}\{\|f_n\|_{C^{1+\alpha}}\}< \infty$, \smallskip
    \item $\mu_n\xrightarrow{*}\mu$.
    \end{itemize}    
If there exists $\mathrm{k} \in\mathbb N$ such that
\begin{itemize}
    \item $\mathrm{Hyp}^{\mathrm{k}}(\mu)>0,$ \smallskip
    \item $\sum_{i=1}^{\mathrm{k} }\lambda_i(\mu_n,f_n)\xrightarrow{n\to\infty}\sum_{i=1}^\rk\lambda_i(\mu,f),$
\end{itemize}
then $h^\rk_{\mu_n}(f_n)$ is well-defined for any $n\in\mathbb N$ large enough and
\[
\limsup_{n\to\infty}h^\rk_{\mu_n}(f_n)\ls h_{\mu}^\rk(f).
\]
\end{theorem}

\begin{remark}
  In Theorem \ref{thm 3.3}, the measure $\mu$ is $f$-invariant, but not necessarily ergodic. When $\mu$ is ergodic, the inequality $\mathrm{Hyp}^{\rk}(\mu)>0$ holds automatically at every jump among positive Lyapunov exponents, i.e., whenever $\lambda_{\rk}(\mu,f) \neq \lambda_{\rk+1}(\mu,f)$ and $\lambda_{\rk}(\mu,f)>0$.  \Cref{first} follows from this theorem when we take $f_1=f_2=\cdots=f$.
\end{remark}


\smallskip

\subsubsection{Along sequences of invariant measures}$\,$\smallskip

To begin, we extend the definition of partial entropy to all invariant measures, which may fail to possess distinct Lyapunov exponents, via ergodic decomposition.

\begin{definition}[Partial Entropy for Invariant Measures]
    For any $f$-invariant measure $m$, possibly with $\mathrm{Hyp}^\rk(m)=0$, the $\rk$-dimensional partial entropy $\tilde h^\rk_{m}(f)$ is defined by its ergodic decomposition:
    $$\tilde h^\rk_{m}(f):=\int_{\mathrm{Hyp}^{\mathrm{k}}(m_x)>0} h^\rk_{m_x}(f)\,\td m(x).$$
\end{definition}

It immediately follows that $\tilde h^\rk_{m}(f)$ is well-defined for any invariant measure and $\tilde h^\rk_{m}(f)= h^\rk_{m}(f)$ whenever $h^\rk_{m}(f)$ is well-defined.

Next, we prove the following theorem generalizing \Cref{thm 3.3}, which relaxes the ergodic assumption for $\mu_n$.

\begin{theorem}\label{thm bestcor}
    Let $f_n\in\mathrm{Diff}^{1+\alpha}(M)$ and $\mu_n$ be an invariant measure of $f_n$. Assume
    \begin{itemize}
    \item $f_n\xrightarrow{C^1} f\in \mathrm{Diff}^{1+\alpha}(M)$
    with $\sup_{n\in \mathbb N}\{\|f_n\|_{C^{1+\alpha}}\}< \infty$, \smallskip
    \item $\mu_n\xrightarrow{*}\mu$.
    \end{itemize}    
If there exists $\rk\in\mathbb N$ such that
\begin{itemize}
    \item $\mathrm{Hyp}^{\mathrm{k}}(\mu)>0,$ \smallskip
    \item $\sum_{i=1}^\rk\lambda_i(\mu_n,f_n)\xrightarrow{n\to\infty}\sum_{i=1}^\rk\lambda_i(\mu,f),$
\end{itemize}
then we have
\[
\limsup_{n\to\infty}\tilde h^\rk_{\mu_n}(f_n)\ls h_{\mu}^\rk(f).
\]
\end{theorem}
\begin{remark}
For partially hyperbolic diffeomorphisms, the upper semi-continuity of the unstable entropy was established by Yang \cite{YANG_expanding_entropy} and Hu, Hua, and Wu  \cite{wuweisheng-u-entropy}. 
More generally, Newhouse \cite{New89} introduced the notion of hyperbolic rate to study the upper semi-continuity of metric entropy for hyperbolic measures. This  was  extended to nonhyperbolic measures in \cite{Liao19}.  For partial entropy,  Hu and Wu \cite{huwu2024} proved that the upper semi-continuity holds on a set of invariant measures with the same expansive rate or the same hyperbolic rate under the compactness assumption.  
See more related results of $u$-entropy in the survey \cite{Tahzibi2021}. 
\end{remark}

Together with the Ledrappier-Young entropy formula \cite[Theorem C]{LEDRAPPIER_YOUNG_B}, see \cite{LEDRAPPIER_YOUNG_A} and \cite{LEDRAPPIER_YOUNG_B} for instance, we obtain the following corollary for the upper semi-continuity of the metric entropy.
\begin{corollary}\label{USC of entropy}
        Let $f_n\in\mathrm{Diff^{1+\alpha}}(M)$ and $\mu_n$ be an invariant measure of $f_n$. Assume
    \begin{itemize}
    \item $f_n\xrightarrow{C^1} f\in \mathrm{Diff^{1+\alpha}}(M)$
    with $\sup_{n\in \mathbb N}\{\|f_n\|_{C^{1+\alpha}}\}< \infty$, \smallskip
    \item $\mu_n\xrightarrow{*}\mu$.
    \end{itemize}    
If there exists $u\in\mathbb N$ such that
\begin{itemize}
    \item $\essinf_{x\sim\mu}\lambda_u(x,f)>0\gs \esssup_{x\sim\mu}\lambda_{u+1}(x,f),$ \smallskip
    \item $\sum_{i=1}^u\lambda_i(\mu_n,f_n)\xrightarrow{n\to\infty}\sum_{i=1}^u\lambda_i(\mu,f),$
\end{itemize}
then we have
\[
\limsup_{n\to\infty} h_{\mu_n}(f_n)\ls h_{\mu}(f).
\]
\end{corollary}
\begin{remark}\label{rem of positive LE sum}
     When the first assumption
     $$\essinf_{x\sim\mu}\lambda_u(x,f)>0\gs \esssup_{x\sim\mu}\lambda_{u+1}(x,f)$$ 
     holds, the second assumption $$\sum_{i=1}^u\lambda_i(\mu_n,f_n)\xrightarrow{n\to\infty}\sum_{i=1}^u\lambda_i(\mu,f)$$ is equivalent to the continuity of the sum of positive Lyapunov exponents, that is
    \[
    \sum_{i=1}^\td \lambda_i^+(\mu_n,f_n)\xrightarrow{n\to\infty}\sum_{i=1}^\td \lambda_i^+(\mu,f).
    \]
    where $\lambda_{i}^+(m,f)=\int_M \lambda_{i}^+(f,x) \, \td m$, and $\lambda_{i}^+(f,x)=\max\{\lambda_{i}(f,x), 0\}$.
\end{remark}

For any ergodic limit measure $\mu$, the first assumption holds automatically for $u=\dim E^+(\mu)$, where $\dim E^+(\mu)$ denotes the number of positive Lyapunov exponents of the ergodic measure $\mu$ counted with multiplicities. Then \Cref{Cor 1.6} directly follows from \Cref{USC of entropy} and \Cref{rem of positive LE sum}.
\begin{corollary}\label{Cor 1.6}
           Let $f_n\in\mathrm{Diff^{1+\alpha}}(M)$ and $\mu_n$ be an invariant measure of $f_n$. Assume
    \begin{itemize}
    \item $f_n\xrightarrow{C^1} f\in \mathrm{Diff^{1+\alpha}}(M)$
    with $\sup_{n\in \mathbb N}\{\|f_n\|_{C^{1+\alpha}}\}< \infty$, \smallskip
    \item $\mu_n\xrightarrow{*}\mu$.
    \end{itemize}    
If $\mu$ is ergodic and  $\sum_{i=1}^\td \lambda_i^+(\mu_n,f_n)\xrightarrow{n\to\infty}\sum_{i=1}^\td \lambda_i^+(\mu,f)$,
then
\[
\limsup_{n\to\infty} h_{\mu_n}(f_n)\ls h_{\mu}(f).
\]
\end{corollary}

\subsubsection{Quantitative estimates}$\,$ \smallskip

In this section, we estimate the defect of the upper semi-continuity of partial entropy in the quantitative sense. We prove that the defect of the upper semi-continuity of partial entropy is explicitly bounded by a Lipschitz function with respect to the defect of the continuity of the sums of Lyapunov exponents. Theorem \ref{thm bestcor} follows directly from Theorem \ref{thm best}.    

\begin{theorem}\label{thm best}
      Let $f_n\in\mathrm{Diff^{1+\alpha}}(M)$ and $\mu_n$ be an invariant measure of $f_n$. Assume 
    \begin{itemize}
    \item $f_n\xrightarrow{C^1} f\in \mathrm{Diff}^{1+\alpha}(M)$
    with $\sup_{n\in \mathbb N}\{\|f_n\|_{C^{1+\alpha}}\}< \infty$, \smallskip
    \item $\mu_n\xrightarrow{*}\mu$,
    \end{itemize}    
and denote $$C_f:={\frac{10^9\cdot\dim M}{\alpha}\cdot(\log\|f\|_{C^1}+1)^4}.$$ If there exists 
$\rk\in\mathbb N$ such that $\mathrm{Hyp}^{\mathrm{k}}(\mu)>0$,
then
\[
\limsup_{n\to\infty}\tilde h^\rk_{\mu_n}(f_n)\ls h^\rk_{\mu}(f)+\frac{C_f}{\mathrm{Hyp}^{\mathrm{k}}(\mu)^4}\cdot\limsup_{n\to\infty}\left(\sum_{i=1}^\rk\lambda_i(\mu,f)-\sum_{i=1}^\rk\lambda_i(\mu_n,f_n)\right).
\] 
\end{theorem}

\begin{remark}[Comparison with Yomdin Theory's Approaches]
Theorem \ref{thm best} reveals the intrinsic relation between the defect in the upper semi-continuity of partial entropy under general measure approximation procedures and the defect of continuity for Lyapunov exponents. The motivation for our defect formula is closely related to the estimates obtained by Yomdin theory \cite{yomdin1987}. Based on Yomdin theory, Newhouse \cite{New89} proved that for $f\in \mathrm{Diff}^r(M)$, such a defect is controlled by
\begin{equation}\label{Yom-bound}\limsup_{n\to\infty} h_{\mu_n}(f)- h_{\mu}(f)\ls\frac{\td }{r}\cdot\limsup_{n\to \infty} \frac{1 }{n}\log \|Df^n\|,\end{equation}
which further implies the upper semi-continuity of the metric entropy for $C^{\infty}$ systems (see also Buzzi \cite{BuzziIsrael}). In the case of $r<\infty$, the upper bound in (\ref{Yom-bound}) does not vanish.
Burguet \cite{Burguet11, burguet2012, burguet2024, burguetinvent} and Buzzi, Crovisier, and Sarig \cite{bcs2022} incorporated  bounded distortion and projective  tangent dynamics into the analysis of entropy growth. For one-dimensional curves, the bounded distortion control of reparameterizations in Yomdin theory has been established. Accordingly, the defect associated with the $1$-dimensional partial entropy can be refined to
\begin{equation}\label{difference}\limsup_{n\to\infty} h^1_{\mu_n}(f)- h_{\mu}(f)\ls \frac{1}{r-1}\cdot \limsup_{n\to \infty} \left(\lambda_1(\mu,f)-\lambda_1(\mu_n, f)\right).\end{equation}
This characterization enables effective control of total entropy when  $\dim M\ls 3$, as either the unstable direction with respect to $f$ or that with respect to $f^{-1}$ is necessarily one-dimensional. This estimate supports the construction of symbolic extensions in \cite{burgertliao} for $3$-dimensional diffeomorphisms.  More generalizations for perturbed systems are presented in \cite{dawei-luo-2025USC}, and related quantitative discussions can be found in \cite{tianxueting2026USC}. Notably, the term $h_{\mu}(f)$ in (\ref{difference}) is not reduced to $h^1_{\mu}(f)$, since the previous methods  compare the $1$-dimensional partial entropies of $\mu_n$ with the metric entropy of $\mu$, and no upper semi-continuity mechanism for partial entropy has been established.  In this sense,  the present work provides a  treatment of this issue.    

The higher dimensional control on bounded distortion in the reparameterzations of Yomdin theory is still unknown. Burguet \cite{Burguet11} {conjectured} the following upper estimate for the defect of total entropy in arbitrary dimension:
\begin{equation}\label{difference-sum}\limsup_{n\to\infty} h_{\mu_n}(f)- h_{\mu}(f) \ls\frac{1}{r-1}\cdot \limsup_{n\to \infty} \left(\sum_{i=1}^\td \lambda_i^+(\mu,f)-\sum_{i=1}^\td \lambda^+_i(\mu_n, f)\right).\end{equation} 

Motivated by the upper bound  formulation in  (\ref{difference-sum}) which centers on  the defect of continuity of Lyapunov exponents, we proceed to directly investigate how the Lyapunov exponent defect constrains the upper semi-continuity defect of (partial) entropy. 
The recent SPR theory for diffeomorphisms developed by Buzzi, Crovisier, and Sarig \cite{SPR2025} indicates that the uniform mass assigned to Pesin blocks by large entropy measures (which leads to the SPR property) arises from the continuity of Lyapunov exponents.
In Theorem \ref{thm best}, we prove that this mechanism, along with its partial counterpart established in this paper, exactly yields an upper bound for the upper semi-continuity defect of the $\rk$-dimensional partial entropy in arbitrary dimensions,
which allows us to establish a partial version for conjecture (\ref{difference-sum}). 

The approach developed here suggests the possibility of extending high-dimensional bounded distortion techniques within Yomdin theory, particularly for problems involving measure approximations. Furthermore, these methods may possibly open new avenues for exploring the statistical properties of dynamical systems, such as symbolic extension, SPR theory, and the distribution of Gibbs states along partial directions.
\end{remark}

\subsubsection{Upper semi-continuity at generic ergodic measures}$\,$ \smallskip

In the following, we prove that for any $C^{1+\alpha}$ diffeomorphism, both the metric entropy map and all the partial entropy maps are upper semi-continuous at generic ergodic measures within the space of invariant measures.
A notable advantage of our result lies in its broad applicability: it holds for every $C^{1+\alpha}$ diffeomorphism without any additional assumptions and extends to all orders of partial entropies.

Let us first explain the concept of generic ergodic measures. Consider the space of invariant measures $\M(M,f)$ and the space of ergodic measures $\M_e(M,f)\subseteq \M(M,f)$. When we say a property holds for generic ergodic measures, it means that it holds for a generic subset of the compact space $\overline{\M_e(M,f)}\subseteq \M(M,f)$ under the weak$^*$ topology. Since it is well-known that the set of ergodic measures $\M_e(M,f)$ is a $G_\delta$ set in $\M(M,f)$, see \cite[Theorem 2.1]{Parth1961}, it follows that $\M_e(M,f)$ itself is a generic subset of $\overline{\M_e(M,f)}$. Consequently, for statements regarding generic ergodic measures, we implicitly assume the ergodicity of the measures involved. Then we can establish the following theorem:

\begin{theorem}\label{USC at generic erg}
    Let $f\in \mathrm{Diff}^{1+\alpha}(M)$, then for generic ergodic measures $\mu$, 
\begin{itemize}
    \item[1.] The entropy map is upper semi-continuous at $\mu$. That is, for any sequence of invariant measures $\mu_n\xrightarrow[n\to\infty]*\mu$,
    $$\limsup_{n\to\infty}h_{\mu_n}(f)\ls h_{\mu}(f).$$
    \item[2.] All partial entropy maps are upper semi-continuous at $\mu$. That is, for any $\rk\in\mathbb N$ such that $h^\rk_\mu(f)$ is well-defined and any sequence of invariant measures $\mu_n\xrightarrow[n\to\infty]{*}\mu$,
    $$\limsup_{n\to\infty}\tilde h^\rk_{\mu_n}(f)\ls  h^\rk_{\mu}(f).$$
\end{itemize}
\end{theorem}
\begin{remark}
    \Cref{sec} follows directly from \Cref{USC at generic erg}.
\end{remark}

\subsection{Applications and examples}\label{sec: 1.2}
In this section, we present various applications and examples of our main theorems, including topics related to measures with dominated splitting, SRB measures, average expanding diffeomorphisms, singular flows, standard maps, and symbolic codings for diffeomorphisms. All the proofs can be found in \Cref{sec example}. 

\subsubsection{Measures with dominated splitting}$\,$\smallskip

Invariant measures that admit  dominated splitting arise naturally in various contexts, most notably within partially hyperbolic systems, and diffeomorphisms away from homoclinic tangencies \cite{LVY13}. 

First, let us recall some basic properties of the dominated splitting, see \cite{ABC} for instance. A $Df$-invariant splitting $T_\Lambda M=E\oplus F$ of the tangent bundle over an $f$-invariant set $\Lambda$ is $L$-dominated if there exists $L>0$ such that given any $x\in \Lambda$ and any unit vectors $u\in E(x), v\in F(x)$, we have $$\| Df_x^L(v)\|\ls \frac{1}{2}\| Df_x^L(u)\|.$$
Note that if an $f$ invariant set $\Lambda$ admits a $L$-dominated splitting $T_\Lambda M=E\oplus F$, then the splitting depends continuously on $x\in \Lambda$ and extends uniquely to a $L$-dominated splitting over the closure of $\Lambda$. More generally, a $Df$-invariant splitting $T_\Lambda M=E_1\oplus\cdots\oplus E_t$ is dominated if given any $\ell\in\{ 1,\cdots,t-1\}$, the splitting $(E_1\oplus\cdots\oplus E_\ell)\oplus(E_{\ell+1}\oplus\cdots\oplus E_t)$ is dominated.

Using results in the previous section, we can establish the following proposition, showing that the partial entropy is upper semi-continuous for measures with dominated splitting. 

\begin{proposition}\label{prop domniated}
     Let $f_n\in\mathrm{Diff^{1+\alpha}}(M)$ and $\mu_n$ be an ergodic measure of $f_n$. Assume
    \begin{itemize}
    \item $f_n\xrightarrow{C^1} f\in \mathrm{Diff^{1+\alpha}(M)}$
    with $\sup_{n\in \mathbb N}\{\|f_n\|_{C^{1+\alpha}}\}< \infty$, \smallskip
    \item $\mu_n\xrightarrow{*}\mu$.
    \end{itemize}    
Let $\rk \in \mathbb{N}$ and $L > 0$ satisfy the following:
\begin{itemize}
    \item There exists an $L$-dominated splitting of $f_n$ over the support of $\mu_n$:
    \[
    T_{\supp \mu_n}M = E^\rk_n \oplus F_n, \quad \dim E_n^\rk = \rk.
    \]
    \item There exists an $L$-dominated splitting of $f$ over the support of $\mu$:
    \[
    T_{\supp \mu}M = E^\rk \oplus F, \quad \dim E^\rk = \rk, \quad \text{and} \quad \essinf_{x \sim \mu} \lambda_\rk(x) > 0.
    \]
\end{itemize}
Then $h^\rk_{\mu_n}(f_n)$ and $h_{\mu}^\rk(f)$ are well-defined when $n$ is large enough, and
\[
\limsup_{n\to\infty} h^\rk_{\mu_n}(f_n)\ls h_{\mu}^\rk(f).
\]
\begin{remark}
  The regularity assumptions on $f_n$ and $f$ in \Cref{prop domniated} can be relaxed to $C^1$, namely $f,f_n\in\Diff(M)$ and $f_n\xrightarrow{C^1} f$. This is because in \cite{Part1} and \cite{Part2}, a parallel partial entropy theory was constructed for $C^1$ diffeomorphisms and measures with dominated splitting, which is analogous to the $C^{1+\alpha}$ theory used in this paper. That framework implies the validity of \Cref{prop domniated} under $C^1$ regularity with analogous proofs.
\end{remark}
\end{proposition}

\subsubsection{SRB measures}$\,$\smallskip

Sinai-Ruelle-Bowen (SRB) measures occupy a central position in the statistical description of chaotic systems.
Recall that by \cite{LEDRAPPIER_YOUNG_A}, for any $f\in\mathrm{Diff}^{1+\alpha}(M)$, an $f$-invariant measure $\mu$ is called an SRB measure if it satisfies the Pesin entropy formula, that is,
\[
h_{\mu}(f)=\int \sum_{i:\lambda_i(x,f)>0}\lambda_i(x,f)\,\td\mu(x).
\]
The following proposition can be established, which shows that the limit of SRB measures is still an SRB measure if the sum of positive Lyapunov exponents is continuous.
\begin{proposition}\label{thm of SRB}
    Let $f_n\in\mathrm{Diff^{1+\alpha}}(M)$ and $\mu_n$ be an SRB measure of $f_n$. Assume
    \begin{itemize}
    \item $f_n\xrightarrow{C^1} f\in \mathrm{Diff^{1+\alpha}}(M)$
    with $\sup_{n\in \mathbb N}\{\|f_n\|_{C^{1+\alpha}}\}< \infty$, \smallskip
    \item $\mu_n\xrightarrow{*}\mu$.
    \end{itemize}    
    If there exists $u\in\mathbb N$ such that 
    \begin{itemize}
        \item $\essinf_{x\sim\mu}\lambda_u(x,f)>0\gs \esssup_{x\sim\mu}\lambda_{u+1}(x,f)$,\smallskip
    \item 
    $\sum_{i=1}^\td \lambda_i^+(\mu_n,f_n)\xrightarrow{n\to\infty}\sum_{i=1}^\td \lambda_i^+(\mu,f)
    $,
    \end{itemize}
    then $\mu$ is an SRB measure of $f$.
\end{proposition}

\subsubsection{Average expanding diffeomorphisms}$\,$\smallskip

In this section, we focus on partially hyperbolic diffeomorphisms with the uniformly expanding on average property (UEAP) in \cite{COSY}. The idea of UEAP arises in the setting of random dynamics in the work by Dolgopyat and Krikorian \cite{Dmitry2007}. For random walks on homogeneous spaces within a similar framework, the concept was first introduced by Eskin and Lindenstrauss in \cite{Eskin}. A number of contributions, such as those by Liu \cite{liu2016}, Brown and Hertz \cite{brown2017}, and Chung \cite{Chung2020}, elucidated the role of this property for the theory. One may also see the recent paper by DeWitt and Dolgopyat \cite{Dmitry2024}.

To begin with, we introduce the partially hyperbolic diffeomorphisms.

\begin{definition}[Partially Hyperbolic Diffeomorphisms]
A $C^r$ diffeomorphism $f$ is partially hyperbolic, denoted by $f\in \mathrm{PH}^r(M)$, if there exists a $Df$-invariant dominated
splitting $$TM= E^u\oplus E^c\oplus E^s$$ such that for every $x\in M$, 
$$\|Df_x\mid_{E^s(x)}\| <1 \text{   and   }\| Df^{-1}_{f(x)} \mid_{E^u(f(x))}\|<1.$$
\end{definition}
Throughout this section, we assume that the center bundle $E^c$ is $2$-dimensional. 
To proceed, we introduce the definition of center bunched as in  \cite{Amie2010,AmieW2013}.

\begin{definition}[Center Bunched]
For $\theta>0$, we say that $f$ is $\theta$-center bunched if for every
$x\in M$,
$$\|Df_x\mid_{E^s(x)}\|^\theta <\frac{\m(Df_x\mid_{E^c(x)})}{\|Df_x\mid_{E^c(x)}\|} \ls \frac{\|Df_x\mid_{E^c(x)}\|}{\m(Df_x\mid_{E^c(x)})}<\m(Df_x\mid_{E^u(x)})^\theta,$$
where for a linear map $A$, $\m(A)=\|A^{-1}\|^{-1}$.
The set of all the center bunched partially hyperbolic diffeomorphisms (for some $\theta>0$) is denoted by $\mathrm{PH}^r_{\mathrm{bun}}(M)$.
\end{definition}

In \cite{SaghinRadu2025}, it was proved that if $f \in \mathrm{PH}^r_{\mathrm{bun}}(M)$ with $r\gs 2$, then for any non-trivial unstable disk $D_p$ centered at $p$, a unit vector $v\in E^c_p$ and a point $q\in D_p$, $f$ admits a \emph{$u$-holonomy} $h^u_{q,p}:E^c_p\to E^c_q$ (see \cite{SaghinRadu2025} for the precise definition).
Then, we can define a H\"older continuous vector field $\chi_{\cdot,p}(v)$ over $D_p$ by the formula
$$q\in D_p \Rightarrow \chi_{q,p}(v) := \frac{h^u_{q,p}(v)}{\|h^u_{q,p}(v)\|}.$$
The uniformly expanding on average property is defined as follows:

\begin{definition}[Uniformly Expanding on Average Property]
A diffeomorphism $f\in \mathrm{PH}^r_{\mathrm{bun}}(M)$ is said to satisfy the \textit{uniformly expanding on average property} (UEAP) if there exists $C_{U}>0$ such that: for every non-trivial unstable disk $D_p$ and every unit vector $v\in E^c_p$, the following holds
$$\liminf_{n\to \infty}\frac{1}{n \mathrm{Leb}^u(D_p)} \int_{D_p} \log \|D_qf^n (\chi_{q,p}(v))\|\,\td \mathrm{Leb}^u(q) >C_{U}.$$
\end{definition}
\begin{remark}\label{Rek.UEAP}
Although not obvious, the UEAP also holds for diffeomorphisms in a $C^2$ neighborhood with the same $C_U$ (cf. \cite[Proposition 2.3]{COSY}).
\end{remark}

The main result in \cite{COSY} shows the existence, finiteness, and full basin property of SRB measures for these diffeomorphisms with UEAP.
\begin{theorem}[Theorem A of \cite{COSY}]\label{Thm. SRB measures}
Let $f\in \mathrm{PH}^2_{\mathrm{bun}}(M)$ with UEAP and satisfy $\mathrm{Def}(f)>0$,
where
$$\mathrm{Def}(f) := C_{U} -\sup_{x\in M}\limsup_{n\to \infty}\frac{1}{n}\log \mid \det (D_x f^n \mid_{E^c(x)})\mid.$$
Then
$$0 < \#\{m: \;\; m \;\; \text{is an ergodic SRB measure of }\;\;f\} <\infty,$$
and every $m$ is a physical measure, i.e., a measure whose basin has positive Lebesgue measure.
\end{theorem}

Combined with our main theorems and several discussions concerning Gibbs $u$-states, we can establish the continuity property of SRB or physical measures for these average expanding diffeomorphisms. The difficulty arises because the center bundle lacks a dominated splitting and contains both positive and negative Lyapunov exponents; in this regime, the UEAP is essential to ensure the continuity of Lyapunov exponents.
\begin{proposition}\label{Pro.Gibbs uppersemi}
Let $f\in \mathrm{PH}^2_{\mathrm{bun}}(M)$ be a diffeomorphism with UEAP and satisfy $\mathrm{Def}(f)>0$. Suppose $f_n\in\mathrm{Diff}^2(M)$ and $f_n\xrightarrow[n\to\infty]{C^2}f$. Let $\mu_n$ be an SRB measure of $f_n$ such that $\mu_n\xrightarrow[n\to\infty]{*} \mu$, then $\mu$ is an SRB measure of $f$. 

Furthermore, if $f$ admits a unique SRB measure $\mu_f$, then there exists a $C^2$ open neighborhood $\mathcal{U}$ of $f$ such that for any $g\in\mathcal{U}$, 
\begin{itemize}
    \item $g$ admits a unique SRB measure $\mu_g$,\smallskip
    \item and $\mu_g$ depends continuously on $g$.
\end{itemize}
\end{proposition}

\subsubsection{Lorenz and Lorenz-derived attractors}$\,$\smallskip

Lorenz and Lorenz-like attractors are central objects in the study of robust chaotic dynamics. The geometric Lorenz attractor \cite{LL56Guckenheimer1976, LL57Guckenheimer1979, afraimovich1977origin} is a canonical example: although not structurally stable, it remains robust under perturbations. To capture this phenomenon, Morales, Pacifico, and Pujals introduced \emph{singular hyperbolicity} \cite{LL92Pujals1999}, which was later generalized to higher dimensions as \emph{sectional hyperbolicity} \cite{LL91Morales08}. For recent studies on singular flows, see, e.g., \cite{CrovisierYang15announce,LL47daLuz2018, LL18BonattidaLuz2021, LL111ShiGanWen2014,SYY,li-liu-singularflowcoding}.

Entropy theory was proved to be an effective tool for classifying the topological structure of flows; see \cite{PYYtrans} and \cite{G-Y-Z-singularflow}. Owing to the difficulties arising from singularities, results on entropy are not easy to obtain.
A groundbreaking work was \cite{APPV09}, in which the authors proved that in dimension three, all singular-hyperbolic attractors are robustly expansive. Consequently, the metric entropy is upper semi-continuous. More recently, it was shown in \cite{PYY-Poincare} that all sectional-hyperbolic chain recurrence classes are robustly entropy-expansive. This result has been further extended to all singular star flows in \cite{pacificoYY2025}.

The first application of our main results is an alternative proof for the upper semi-continuity of metric entropy for sectional hyperbolic sets (in particular, the Lorenz attractor) in \cite{PYY-Poincare}, which avoids the technical and complicated analysis in flow theory.

Let $X$ be a $C^{1+\alpha}$ vector field and $\phi^X_t$ the associated flow on a Riemannian
manifold $M$ without boundary. We denote $f:=\phi^X_1$ as the time-one map of $\phi_t^X$. 
The points $\sigma$ with $X(\sigma)=0$ are called singularities. When the singularities are hyperbolic, they are isolated. Points other than singularities are called nonsingular points. The \emph{regular measures} are defined as invariant measures with vanishing weight on singularities.
The metric entropy (resp. partial entropies) for invariant measures of $\phi^X_t$ is defined as that of its time-one map $f$. Then we introduce the concept of sectional hyperbolic sets.

\begin{definition}[Sectional Hyperbolicity]\label{def sectional hyper}
A compact invariant set $\Lambda$ of a flow $X$ is called sectional hyperbolic, if it admits a dominated splitting $E^{ss}\oplus F^{cu}$, such that $D\phi_t^X\mid_{E^{ss}}$
is uniformly contracting, and $D\phi_t^X \mid_{F^{cu}}$ is sectional-expanding, that is, there exist
constants $C_1>0$, $\lambda>1$ such that for every $x\in \Lambda$ and any subspace $V_x\subset F^{cu}_x$ with $dim(V_x) = 2$, we have
$$\mid \det (D\phi^X_t(x)\mid_{V_x})\mid \gs C_1 \lambda^t, \quad\forall\,t>0.$$ 
\end{definition}

It was proved in \cite{PYY-Poincare} that sectional hyperbolic sets are entropy expansive, and thus the metric entropy function is upper semi-continuous.
Using our main theorems, we give a new proof of this result. Indeed, the metric entropy of invariant measures agrees with the $\dim E^{ss}$-dimensional stable entropy of $\phi^X_t$, and the sum of the negative Lyapunov exponents in $E^{ss}$ is continuous.

\begin{proposition}\label{Thm. Lorenz attractor}
Let $\Lambda$ be a compact invariant set that is sectional hyperbolic for a $C^{1+\alpha}$ flow $\phi^X_t$, with all the singularities in $\Lambda$ hyperbolic. Then the metric entropy 
is upper semi-continuous with respect to all invariant measures supported on $\Lambda$ in the weak$^*$ topology.
\end{proposition}

Our main results can also be applied to study examples derived from the Lorenz attractor in \cite{LYYZ-Derived-from-Lorenze}, where the authors considered a DA-type surgery on the Lorenz attractor in dimension $4$. Such surgeries were first studied by Smale \cite{Smale1967} and Mañé \cite{mane1978} to construct fundamental examples in the theory of partially hyperbolic systems. Their construction yields the first example of a singular chain recurrence class that is Lyapunov stable, lies away from homoclinic tangencies, and contains periodic orbits with distinct indices.
In what follows, we prove that for this open set of flow attractors, the $1$-dimensional partial entropy of regular measures supported on its Lyapunov stable chain recurrence classes varies upper semi-continuously. This application lies beyond the scope of classical results on entropy expansiveness and flow dynamics, establishing new continuity properties for the (partial) entropy of singular flows.
We now proceed to the details.

We say a flow admits a homoclinic tangency if it possesses a hyperbolic periodic orbit whose stable and unstable manifolds intersect non-transversely. Denote by $\mathrm{HT}$ the set of flows admitting homoclinic tangencies, and by $\operatorname{Cl}(\mathrm{HT})$ its closure. A flow $X$ is said to be \emph{away from homoclinic tangencies} if $X\notin \operatorname{Cl}(\mathrm{HT})$

Following Conley's theory, see for instance \cite{Go2010}, we denote by $CR(X)$ the chain recurrence set of $X$. For any $x\in CR(X)$, the chain recurrence class containing $x$ is denoted by $C(x,X)$.
A chain recurrence class $\Lambda$ of a flow $\phi^X_t$ is called Lyapunov stable, if for any neighborhood $U$ of $\Lambda$, there is another neighborhood $V$ of $\Lambda$ such that $\phi^X_t(V) \subset U$ for any $t\gs 0$. The following theorem is included in \cite[Theorem A]{LYYZ-Derived-from-Lorenze}.

\begin{proposition}[Theorem A of \cite{LYYZ-Derived-from-Lorenze}]\label{Pro. LYYZ}
On every $4$-dimensional closed Riemannian manifold $M$ there exists a non-empty
open set of flows $\mathcal{V} \subset X^1(M)\setminus\operatorname{Cl}(\mathrm{HT})$ which are derived from Lorenz attractor, and an open set $U \subset M$ such that for any $X \in \mathcal{V}$, one has $\phi^X_t (U) \subset U$
for any $t > 0$, and there is a unique singularity $\sigma\in U$ which is hyperbolic. Then for this $\sigma$,
\begin{itemize}
\item $C(\sigma,X)$ is the unique Lyapunov stable chain recurrence class in $U$;\smallskip
\item $C(\sigma,X)$ contains hyperbolic periodic orbits with index $1$ and $2$;\smallskip
\item the tangent bundle $T_{C(\sigma,X)}M$ contains a $2$-dimensional subbundle $E$ which is sectional-expanding.
\end{itemize}
\end{proposition}

Next, together with our main results, we can establish the upper semi-continuity of the $1$-dimensional partial entropy as follows:
\begin{proposition}\label{Thm. Derived-Lorenz}
For any $C^{1+\alpha}$ flow $X\in \mathcal{V}$ and any sequence of regular measures $\{\mu_n\}_{n\in\mathbb N}$ supported on $C(\sigma,X)$ converging to a regular measure $\mu$ in the weak$^*$ topology, the $1$-dimensional partial entropy is upper semi-continuous, i.e.
$$\limsup_n h^{1}_{\mu_n}(\phi_t^X)\ls h^1_{\mu}(\phi^X_t),$$ 
where $\phi_t^X$ is the associated flow of $X$.
\end{proposition}

\subsubsection{Standard maps}$\,$\smallskip

A famous example of dynamical systems is the standard map (or Taylor-Chirikov standard map), introduced independently by Taylor and Chirikov (see \cite{CHIRIKOV1979263}). Due to its significance in both physics and dynamical systems \cite{CHIRIKOV1979263, Iz80}, the map has been widely studied. In particular, Gorodetski \cite{Go12} proved that for a generic large parameter, there exists a ``topologically large'' uniformly hyperbolic set for the standard map which is accumulated by elliptic islands.

Taking into account $\mathbb T^2=\mathbb R^2/\mathbb Z^2$ with coordinates $(x,y)\in[0,1)^2$, for each $t\in \mathbb R$, the standard map $f_t:\mathbb T^2\to\mathbb T^2$ is defined as
\[
f_t(x,y)=(2x-y+t\sin(2\pi x),x).
\]

As an easy consequence of \cite[Corollary 2.9]{obatastandardmap}, we can prove that for any $t\in\mathbb R$ large enough, there exists a $C^{1+\alpha}$ neighborhood $\mathcal U$ of $f_t$ in $\mathrm{Diff}^{1+\alpha}(\mathbb T^2)$ such that for any $g\in\mathcal U$, one has $h_{\mathrm{top}}(g)>\frac{\lambda_{\mathrm{min}}(g)}{1+\alpha}$, where
\[
\lambda_{\mathrm{min}}(f):=\min\{\limsup_{n\to\infty}\frac{1}{n}\log\|Df^n\|,\limsup_{n\to\infty}\frac{1}{n}\log\|Df^{-n}\|\}.
\]
Then, combining \cite[Corollary 1]{burguet2024} and \cite[Theorem A]{obatastandardmap}, it follows that for any $t\in\mathbb R$ large enough, there exists a $C^{1+\alpha}$ neighborhood $\mathcal U$ of $f_t$ in $\mathrm{Diff}^{1+\alpha}(\mathbb T^2)$ such that for any $g\in\mathcal U$, $g$ has a unique measure of maximal entropy $\mu_g$, and for any $g_n\xrightarrow{C^{1+\alpha}}g$ in $\mathcal U$, 
\[
\mu_{g_n}\xrightarrow{*}\mu_g\ \text{ and }\ \lambda_1(\mu_{g_n},g_n)\to\lambda_1(\mu_g,g).
\]
Thus, the measures of maximal entropy of diffeomorphisms in the $C^{1+\alpha}$ neighborhood of standard maps are examples that satisfy the continuity of the sum of Lyapunov exponents, which is the main assumption in our results in \Cref{sec:1.1}.

As a deduction of \Cref{USC of entropy}, we obtain the continuity of topological entropy in a $C^{1+\alpha}$ neighborhood $\mathcal U$ of the standard map $f_t$ when $t > 0$ is sufficiently large. 
\begin{proposition}\label{prop standard maps h_top}
    There exists $t_0>0$ such that for any $t\gs t_0$, there exists a $C^{1+\alpha}$ neighborhood $\mathcal U$ of the standard map $f_t$ in $\mathrm{Diff}^{1+\alpha}(\mathbb T^2)$ such that the topological entropy map $h_{\mathrm{top}}:\mathcal U\to \mathbb R$ is continuous.
\end{proposition}
It should be noted that this result also follows directly from the Corollary~2 of Burguet \cite{burguet2024}, together with the preceding discussion.
However, it is presented here to offer an alternative perspective on this problem.

\subsubsection{Symbolic codings of diffeomorphisms}$\,$\smallskip

A powerful strategy for studying a smooth dynamical system is to construct a semi-conjugacy to a symbolic system. This is called the symbolic coding of the diffeomorphism. 
In \cite{Sarig13}, for any $C^{1+\alpha}$ surface diffeomorphism, Sarig constructed a symbolic coding through a countable Markov shift, which captures all ergodic measures with a prescribed level of hyperbolicity. The method was generalized to higher dimensions by Ben Ovadia in \cite{Snir18coding}.

Another refinement of this technique, introduced by Buzzi, Crovisier, and Sarig in \cite{BCS}, involves localizing the coding to specific dynamically relevant subsets, namely homoclinic classes. They constructed irreducible symbolic codings for these classes, which was a key in their proof of the finiteness of measures of maximal entropy for $C^\infty$ surface diffeomorphisms. Recently, the Strong Positive Recurrence (SPR) property for symbolic codings of diffeomorphisms was systematically investigated by Buzzi, Crovisier, and Sarig in \cite{SPR2025}.

An important fact established by Buzzi, Crovisier, and Sarig is that the Oseledets stable and unstable bundles, while generally only measurable on the base manifold, become Hölder continuous when lifted to the Markov shift associated with the symbolic coding. This regularity implies that the lifted unstable geometric potential is also Hölder continuous. Consequently, whenever a sequence of invariant measures converges in the weak$^*$ topology on the symbolic space, the sum of positive Lyapunov exponents of the corresponding projected measures on the base manifold converges accordingly. By combining this continuity with our main result, we deduce the upper semi-continuity of metric entropy for symbolic codings of diffeomorphisms. To formalize this, we first recall the relevant definitions from symbolic dynamics (see for instance \cite{Sarig13,BCS}).
\begin{definition}[Countable Markov Shift]
Let $G=(V, E)$ be a directed graph with a countable set of vertices $V$ and a set of edges $E \subset V \times V$. The \textit{symbolic space} $\Sigma$ associated with $G$ is the set of all bi-infinite paths on the graph with the standard metric:
\[
\Sigma := \left\{ (v_i)_{i \in \mathbb{Z}} \in V^{\mathbb{Z}} : (v_i, v_{i+1}) \in E \text{ for all } i \in \mathbb{Z} \right\}.
\]
The \textit{left shift map} $\sigma: \Sigma \to \Sigma$ is defined by $(\sigma(\mathbf{v}))_i = v_{i+1}$. The dynamical system $(\Sigma, \sigma)$ is called a \textit{countable Markov shift} or a topological Markov shift. 

The \textit{regular set}, denoted by $\Sigma^{\#}$, consists of sequences that are recurrent in both forward and backward time.
It is a standard result that the set $\Sigma^{\#}$ has full measure with respect to any $\sigma$-invariant probability measure on $\Sigma$.
\end{definition}

Let $f \in \mathrm{Diff}^{1+\alpha}(M)$ and $\mu$ be an ergodic measure for $f$. For any given $\chi > 0$, the measure $\mu$ is said to be $\chi$-hyperbolic if all Lyapunov exponents of $\mu$ have absolute values exceeding $\chi$.
The following theorem summarizes the existence and basic properties of symbolic coding for all $\chi$-hyperbolic measures of $C^{1+\alpha}$ diffeomorphisms.  See \cite{Sarig13}, \cite{Snir18coding} and \cite{BCS} for more details.
\begin{theorem}[Symbolic codings for diffeomorphisms]\label{thm: Sarig coding}
Let $f\in \mathrm{Diff}^{1+\alpha}(M)$. For any $\chi>0$, there exists a locally compact countable Markov shift $(\Sigma_f,\sigma)$ and a H\"{o}lder continuous map $\pi:\Sigma_f\to M$ such that $\pi\circ\sigma=f\circ\pi$ and
\begin{itemize}
    \item[1.] $\pi:\Sigma_f^\#\to M$ is finite-to-one.
    \item[2.] $\nu(\pi(\Sigma_f^\#))=1$ for any $\chi$-hyperbolic measure $\nu\in\mathcal{M}_{\mathrm{e}}(M,f)$. Moreover, there exists an ergodic measure $\bar{\nu}$ on $\Sigma_f$ such that $\pi_*(\bar{\nu})=\nu$. Conversely, if $\bar{\nu}$ is a $\sigma$-ergodic measure on $\Sigma_f$, then $\pi_*(\bar{\nu})$ is $f$-ergodic, hyperbolic, and $h_{\nu}(f)=h_{\bar{\nu}}(\sigma)$.
    \item[3.] For any $x\in\pi(\Sigma_f)$, there exists a splitting $T_xM=E^s(x)\oplus E^u(x)$ such that
    \[
    \limsup_{n\to\infty}\frac{1}{n}\log\|Df^n|_{E^s(x)}\| \ls -\frac{\chi}{2},
    \]
    \[
    \limsup_{n\to\infty}\frac{1}{n}\log\|Df^{-n}|_{E^u(x)}\| \ls -\frac{\chi}{2}.
    \]
    Moreover, on each component $\Sigma_f'$ of $\Sigma_f$, $\dim E^{u/s}(\pi(\underline{x}))$ are constants and the maps $\underline{x}\to E^{u/s}(\pi(\underline{x}))$ are H\"{o}lder continuous.
\end{itemize}
\end{theorem} 
As discussed previously, property 3 in \Cref{thm: Sarig coding} implies the continuity of the sum of positive Lyapunov exponents for the corresponding projected measures under the convergence upon the countable symbolic coding Markov shift. Combining with \Cref{USC of entropy}, the following proposition is obtained.
\begin{proposition}\label{prop: USC of coding system}
Let $f\in\mathrm{Diff}^{1+\alpha}(M)$ and $(\Sigma_f,\sigma)$ be its symbolic coding given by \Cref{thm: Sarig coding}. For any sequence of invariant probability measures $\{\bar{\mu}_n\} \subset \mathcal M(\Sigma_f,\sigma)$ such that $\bar{\mu}_n\xrightarrow{*}\bar{\mu}$ for some $\bar{\mu} \in \mathcal M(\Sigma_f,\sigma)$, 
\[ 
\limsup_{n\to\infty} h_{\bar{\mu}_n}(\sigma)\ls h_{\bar{\mu}}(\sigma).
\]
\end{proposition}

It should be noted that this result was established in a more general symbolic setting by Iommi, Todd, and Velozo \cite{Todd22entropyofTMS}, whose proof is entirely within the symbolic framework of countable Markov shifts. In contrast, our approach provides an alternative dynamical proof for this upper semi-continuity result, derived from the smooth dynamics on the base manifold. Furthermore, their work analyzes the upper semi-continuity defect that arises when a sequence of invariant measures loses mass to infinity. Using our quantitative estimate in \Cref{thm best}, we can also provide a dynamical counterpart to this phenomenon: there exists a constant $L$ depending only on $f$ and $\chi$ such that if the sequence of measures loses mass $\rho$ to infinity, the corresponding entropy defect is bounded by $L\rho$.

Since \Cref{USC of entropy} only requires the continuity of the sum of positive Lyapunov exponents on the base manifold, the same proof shows that Proposition \ref{prop: USC of coding system} also holds under a weaker topology: the weak$^*$ topology with respect to admissible potentials. Following \cite{BCS}, a bounded continuous function $\phi$ on $\Sigma_f$ is called an \textit{admissible potential} if $\phi = \varphi \circ \pi$ for some $\varphi: M \to \mathbb{R}$ which is either a Hölder continuous function or the unstable geometric potential $\varphi^u(x) = -\log \mathrm{Jac}(Df|_{E^u(x)})$. Because the space of admissible potentials is a proper subspace of $C^0_{\mathrm{b}}(\Sigma_f, \mathbb{R})$, this convergence imposes a weaker requirement than the standard weak$^*$ topology. Extending the semi-continuity to this weaker topology is based  on our main results and the observation by Buzzi, Crovisier, and Sarig regarding the Hölder continuity of the lifted Oseledets bundles, as discussed previously.

\bigskip

\noindent{\bf Strategy of proof} 

\smallskip
\smallskip

In the setting of $C^{1+\alpha}$ diffeomorphisms, partial entropy (tied to Lyapunov subspaces) lacks direct continuity tools; uniform subordinate partitions on unstable manifolds are difficult to construct for approximating sequences of measures, and the geometric structure of the unstable lamination is complicated.

We establish the upper semi-continuity of partial entropy in this paper. To this end, we present a systematic analysis and apply key results from Pesin theory, along with the SPR theory for diffeomorphisms developed by Buzzi, Crovisier, and Sarig. Furthermore, we propose a novel method to decompose partial entropy into a main term and an error term, and we derive rigorous estimates for both components.
\smallskip 

\begin{itemize}
\item[(1)] Weighted Pesin Blocks:\\
\begin{itemize}
\item Extend classical Pesin theory to define weighted Pesin blocks (\Cref{definition of GPB}). This relaxes the uniform separation condition on Lyapunov exponents and focuses only on the exponent gap required for the partial entropy to be well-defined. \\

\item Adapt the SPR theory for diffeomorphisms developed by Buzzi, Crovisier, and Sarig to weighted Pesin blocks:  Link the continuity of the sum of  Lyapunov exponents to the existence of the uniformly large-measure weighted Pesin blocks (Proposition \ref{uniform size prop}), ensuring a uniform size of weighted Pesin unstable manifolds, which provides a geometric framework to control entropy variation.\\
\end{itemize}

\item[(2)]  Decompose the partial entropy into a main term (subordinate partition entropy on uniform local weighted Pesin unstable lamination) and an error term (entropy outside uniform large-measure weighted Pesin blocks).  \\

\begin{itemize}
\item 
Develop a measure-theoretic approach for (possibly non-invariant) Borel measures to establish the upper semi-continuity of the main term of partial entropy, tailored to weak$^*$ convergence of measures and the $C^1$ continuity of local weighted Pesin unstable manifolds.\\

\item  Use recurrence time estimates (Lemma \ref{lem 3.11}) and small-boundary partitions (Definition \ref{small boundary}) to bound the error term, ensuring it remains uniformly bounded as the approximating measures converge to the limit.
\\

\item Use the joinings and the passage lemma from \cite{DMsymbolicextension} to bridge the ergodic decompositions of the approximating measures and the limit measure, extending the main results from ergodic to general invariant measures.\\

\end{itemize}

\end{itemize}

\bigskip

\noindent{\bf Organization of the paper}  
\smallskip
\smallskip

In \Cref{sec:pes}, we develop the framework of weighted Pesin blocks. We introduce weighted Pesin blocks in \Cref{sec 2.1} and establish their subexponential angle variation. Then, we construct weighted Pesin charts and their corresponding weighted Pesin unstable manifolds in \Cref{sec new 2.3} and \Cref{sec 2.3}. In \Cref{sec 2.4}, a crucial quantitative estimate (\Cref{quant uniform size prop}) is established, which links the continuity of the sum of Lyapunov exponents to the existence of uniformly sized weighted Pesin blocks.

\Cref{sec 3} is dedicated to the upper semi-continuity of the partial entropy.  Subsequent to the entropy analysis of the Borel measures presented in \Cref{sec 3.1}, we construct subordinate partitions in \Cref{sec 3.2} and decompose the partial entropy into a main term and an error term. The upper semi-continuity of the main term is established in \Cref{sec 3.3}, while the error term is controlled in \Cref{sec 3.4}. This is the most delicate part of the proof. In \Cref{sec 3.5}, we extend the argument from ergodic to general invariant measures via the theory of joinings from \cite{DMsymbolicextension}, completing the proofs of our main results (\Cref{thm best,thm bestcor}). \Cref{sec 3.6} proves the upper semi-continuity for generic ergodic measures. The proofs for the applications in \Cref{sec: 1.2} can be found in \Cref{sec example}.

\smallskip

\section{Pesin theory and weight function}\label{sec:pes}
In this section, we adapt Pesin theory to a weighted setting, designed for applications to the partial entropies of invariant measures in \Cref{thm 3.3} and \Cref{thm best}. This weighted perspective extends the classical Pesin framework and provides a complementary approach to the study of partial entropy.

The motivation for introducing this weighted setting is as follows. We aim to investigate the dynamical properties along the $\rk$-dimensional partial direction for invariant measures satisfying $\mathrm{Hyp}^\rk(m)>0$ via Pesin theory. The classical construction of Pesin blocks relies on a $\chi$-hyperbolicity condition, which provides a uniform separating threshold, namely the level \(0\), between the
positive and negative Lyapunov directions. For invariant measures, however, the condition $\mathrm{Hyp}^\rk(m)>0$ only provides a gap between the $\rk$-th and the $(\rk+1)$-th Lyapunov exponents componentwise. The absolute positions of these exponents may vary substantially among different ergodic components; in particular, the $(\rk+1)$-th exponent on one component may lie above the $\rk$-th exponent on another. Hence, even though the desired gap exists along each component, there may be no global separating constant valid for the whole measure.

To overcome this difficulty, we introduce a weight function $\varphi$ that is continuous and non-negative, and we use it to define weighted Pesin blocks together with weighted Pesin charts. In these weighted Pesin charts, the dynamics becomes ``uniformly partially hyperbolic''. This allows us to define weighted Pesin unstable manifolds for every point in a weighted Pesin block. Moreover, these manifolds coincide almost everywhere with the classical $\rk$-dimensional Pesin unstable manifolds and vary continuously.

This refinement provides more explicit information and plays a crucial role in the present study. While the continuity of the sums of the top $\rk$ Lyapunov exponents alone is insufficient to guarantee uniformly sized classical Pesin blocks (see Appendix \ref{appendix a5}), it is precisely enough to ensure uniformly sized weighted Pesin blocks (\Cref{uniform size prop} and \Cref{quant uniform size prop}). This, in turn, establishes the necessary rigidity result for the convergence of measures in the present study.

Let $M$ be a compact Riemannian manifold  without boundary with a Riemannian norm $\|\cdot\|$ that induces a distance $d$. For any two distinct points $x,y\in M$, let $\gamma_{xy}$ denote the geodesic connecting $x$ and $y$, and let
$$P_{\gamma_{xy}}: T_yM\to T_xM $$
be the parallel transport operator along $\gamma_{xy}$, which is an isometric linear isomorphism preserving the Riemannian inner product.

For a $C^{1+\alpha}$ diffeomorphism $f:M\to M$, the $\alpha$-H\"older semi-norm of the differential operator $Df$ is defined as
$$[Df]_\alpha := \sup_{\substack{x,y\in M \\ x\neq y}} \frac{\left\| P_{\gamma_{f(x)f(y)}} \circ D_yf - D_xf \circ P_{\gamma_{xy}} \right\|}{d(x,y)^\alpha}.$$
From the preceding definition, we observe that if we restrict our analysis to a uniformly small local chart, the parallel transport $P_{\gamma_{xy}}$ can be approximated arbitrarily well by the identity operator. To streamline the notation, we suppress this isomorphism in what follows.
The $C^1$ and $C^{1+\alpha}$ norms of the diffeomorphism  $f$ are given by   
 \begin{eqnarray*}\|f\|_{C^1}&:=&\max\{\|Df\|, \|Df^{-1}\|\}, \\[2mm]
 \|f\|_{C^{1+\alpha}}&:=&\max\{\|f\|_{C^1},\,[Df]_\alpha, [Df^{-1}]_\alpha\}.
 \end{eqnarray*}

\subsection{Weighted Pesin block}\label{sec 2.1}

Let $f\in\mathrm{Diff}^{1+\alpha}(M)$ be a diffeomorphism. 
First, we recall the definition of the classical Pesin block (see, e.g., \cite{pesin1976, Pesinbook} or \cite[Definition 1.2]{SPR2025}):

\begin{definition}[Pesin block]
For any parameters $K\gs1$, $\chi>0$ and $\epsilon>0$ satisfying $\epsilon\ll \chi$,
a non-empty set $\Lambda(K,\chi,\epsilon)\subseteq M$ is called a Pesin block if there exist direct sum decompositions
$$T_yM=E^{s}(y)\oplus E^u(y),\quad \forall\ y\in \bigcup_{n=-\infty}^{+\infty} f^n\Lambda(K,\chi,\epsilon)$$
such that for any $x\in\Lambda(K,\chi,\epsilon)$, $n\in \mathbb{Z}$, $m\in \mathbb{N}$,
\begin{align*}
    \|Df^m|_{E^{s}(f^n(x))}\| &\ls K \cdot \e^{-m\chi  + |n|\epsilon},\\[2mm]
    \|Df^{-m}|_{E^u(f^n(x))}\| &\ls K \cdot \e^{-m\chi  + |n|\epsilon}. 
\end{align*}
\end{definition}

In this paper, we extend the classical Pesin block to the following notion of the weighted Pesin block, which can be seen as a partial version of the classical one.

\begin{definition}[Weighted Pesin Block]\label{definition of GPB}
      For parameters $\rk\in\mathbb N$, $K\gs1$, $\chi>0$ and $\epsilon>0$ that satisfy $\epsilon\ll \chi$, more specifically 
\begin{align}\label{var}
       \epsilon\ls\frac{\alpha\cdot\chi^2}{10^3\cdot(\log\|f\|_{C^1}+1)},
\end{align}
      and a continuous weight function $\varphi \in C^0(M)$ with $\varphi(x)\gs 0$,
      a non-empty set $\Lambda^\rk(K,\chi,\epsilon,\varphi)\subseteq M$ is called a weighted Pesin block if there exist direct sum decompositions
      $$T_yM=E^{cs}(y)\oplus E^u(y),\quad \forall\ y\in \bigcup_{n=-\infty}^{+\infty} f^n\Lambda^\rk(K,\chi,\epsilon,\varphi),$$
      such that $\dim E^u(y)=\rk$ and for any $x\in\Lambda^\rk(K,\chi,\epsilon,\varphi)$, $n\in \mathbb{Z}$, $m\in \mathbb{N}$,
\begin{align}\label{GPB def 1}
    \|Df^m|_{E^{cs}(f^n(x))}\| &\ls K \cdot \e^{-m\chi  + |n|\epsilon}\cdot \e^{\Phi_m(f^n(x))},\\[2mm]
    \label{GPB def 2}
    \|Df^{-m}|_{E^u(f^n(x))}\| &\ls K \cdot \e^{-m\chi  + |n|\epsilon}\cdot \e^{\Phi_{-m}(f^n(x))},
\end{align}
    where $\Phi_{m}(x):= \sum_{i=0}^{m-1}\varphi(f^i(x))$ and $\Phi_{-m}(x):= -\sum_{i=1}^{m}\varphi(f^{-i}(x))$. 
\end{definition}

\begin{remark}

The definition of a weighted Pesin block coincides with that of the classical Pesin block when $\varphi \equiv 0$ and $\rk$ is the dimension of the unstable Oseledets bundles. In this sense, the weighted Pesin block can be viewed as a partial version of the classical Pesin block.  Notably, the weight function $\varphi$ constitutes the fundamental distinction between the weighted Pesin blocks and their classical counterparts (for which $\varphi \equiv \text{constant}$). By introducing $\varphi$, we achieve a refined characterization of the spatial distribution of the boundary between the $\rk$-th and the $(\rk+1)$-th Lyapunov exponent on $M$, applicable to general invariant measures satisfying $\mathrm{Hyp}^\rk(\mu) > 0$.
 
Furthermore, the continuity of $\varphi$ ensures the compactness of weighted Pesin blocks (\Cref{pro 2.1}) and their persistence under system perturbations (\Cref{uniform size prop} and \Cref{quant uniform size prop}), and $\varphi \gs 0$ enforces the necessary expansion in the $E^u$ direction, thereby guaranteeing the existence of weighted Pesin unstable manifolds (\Cref{sec new 2.3} and \Cref{sec 2.3}).
\end{remark}

In the following proposition, we present some basic properties of weighted Pesin blocks and the corresponding decompositions, including compactness, uniqueness, and continuity. See Appendix \ref{appendix a1} for details.
\begin{proposition}\label{pro 2.1}
    Fix $K,\chi,\epsilon$ and $\varphi$ as in \Cref{definition of GPB}. Then for each weighted Pesin block $\Lambda^\rk(K,\chi,\epsilon,\varphi)$ we have:
    \begin{itemize}
    \item[$(a)$] The closure of $\Lambda^\rk(K,\chi,\epsilon,\varphi)$ is again a weighted Pesin block with the same parameters; therefore we may always assume without loss of generality that $\Lambda^\rk(K,\chi,\epsilon,\varphi)$ is compact.\smallskip
    \item[$(b)$] The decomposition $T_xM=E^{cs}(x) \oplus E^u(x)$ is unique for every $x\in \Lambda^\rk(K,\chi,\epsilon,\varphi)$.\smallskip
    \item[$(c)$] The maps $x\mapsto E^{cs}(x)$ and $x\mapsto E^u(x)$ are continuous on $\Lambda^\rk(K,\chi,\epsilon,\varphi)$.
    \end{itemize}
\end{proposition}


To estimate the size of unstable manifolds, as in the classical Pesin theory \cite{pesin1976}, a subexponential variation of angles comes into play, which holds as follows.

To facilitate a thorough analysis, we formally state the rigorous definition of the angle between two subspaces. Let $T^1_xM:= \{v\in T_xM : \|v\|=1\}$ and let $\angle(v_1,v_2)$ denote the classical Euclidean angle between two vectors $v_1, v_2$ in a Euclidean space $F$. 
The angle between two linear subspaces $F_1$ and $F_2$ of $F$ is given by:
$$\angle(F_1,F_2):=\min\{\angle(v_1,v_2): v_1\in F_1,\; v_2\in F_2\}.$$
  We denote
$$L:= 4\cdot(\log \|f\|_{C^1}+1),\qquad \tau:= 2L\cdot\chi^{-1},$$
throughout this paper.

\begin{proposition}\label{pro 2.2}
    Given a weighted Pesin block $\Lambda^\rk(K,\chi,\epsilon,\varphi)$, for any $n\in \mathbb{Z}$, 
    $$\angle(E^{cs}(f^n(x)),E^u(f^n(x)))\gs \frac{\epsilon}{\|f\|_{C^1}^2}\cdot K^{-\tau} \cdot \e^{-|n|\tau\epsilon}.$$  
\end{proposition}
\begin{proof}

    We prove this proposition by contradiction. 
    Suppose  there exists some $n\in \mathbb Z$ such that
    $$\angle(E^{cs}(f^n(x) ),E^u(f^n(x)))< \frac{\epsilon}{\|f\|_{C^1}^2}\cdot K^{-\tau} \cdot \e^{-|n|\tau\epsilon}.$$ 
    It  implies the existence of $v_{1},v_{2}\in T_{f^n(x)}^1M$ with $v_{1}\in E^{cs}(f^nx)$, $v_{2}\in E^u(f^nx)$ and  
    $$\angle(v_1,v_2)<\frac{\epsilon}{\|f\|_{C^1}^2}\cdot K^{-\tau} \cdot \e^{-|n|\tau\epsilon}. $$
    Thus, for any $m\in \mathbb{N}$,
    $$\angle(D_{f^n(x)}f^mv_1,D_{f^n(x)}f^mv_2)< \frac{\epsilon}{\|f\|_{C^1}^2}\cdot K^{-\tau} \cdot \e^{-|n|\tau\epsilon}\cdot \e^{mL} $$
    (see \Cref{lem 2.4} for a detailed computation).  
    Hence, for all $0\ls m \ls H-1$, where $H:=\bigl[\frac{|n|\tau\epsilon}{L}+\frac{\log K\cdot \tau}{L}\bigr]+1$, we have
    $$\angle(D_{f^n(x)}f^mv_1,D_{f^n(x)}f^mv_2)< \frac{\epsilon}{\|f\|_{C^1}^2}. $$
    Then,
    $$ \frac{\|D_{f^{n+m}(x)}f (D_{f^n(x)}f^{m}(v_{2}))\|_{f^{n+m+1}(x)}}{\|D_{f^n(x)}f^{m}(v_{2})\|_{f^{n+m}(x)}}\cdot \Big(\frac{\|D_{f^{n+m}(x)}f (D_{f^n(x)}f^{m}(v_{1}))\|_{f^{n+m+1}(x)}}{\|D_{f^n(x)}f^{m}(v_{1})\|_{f^{n+m}(x)}}\Big)^{-1}<\e^\epsilon,$$
    (see a proof in \Cref{lem 2.3}). 
    Define
    $$v^{m}_{i}:=\frac{D_{f^n(x)}f^{m}(v_{i})}{\|D_{f^n(x)}f^{m}(v_{i})\|_{f^{n+m}(x)}}$$
    for $i=1,2$. Note that $\|v^{m}_{i}\|=1$, and thus
\begin{align*}
    \frac{\|D_{f^n(x)}f^H(v_{2})\|_{f^{n+H}(x)}}{\|D_{f^n(x)}f^H(v_{1})\|_{f^{n+H}(x)}} 
    =\frac{\|D_{f^n(x)}f(v^{0}_{2})\|_{f^{n+1}(x)}}{\|D_{f^n(x)}f(v^{0}_{1})\|_{f^{n+1}(x)}}\cdots 
    \frac{\|D_{f^{n+H-1}(x)}f(v^{H-1}_{2})\|_{f^{n+H}(x)}}{\|D_{f^{n+H-1}(x)}f(v^{H-1}_{1})\|_{f^{n+H}(x)}}
    <\e^{H\epsilon}.
\end{align*}
   On the other hand, by \Cref{GPB def 1} and \Cref{GPB def 2},
\begin{align*}
    \|D_{f^n(x)}f^H(v_2)\|_{f^{n+H}(x)} &\gs K^{-1} \cdot \e^{H(\chi-\epsilon)  -|n|\epsilon}\cdot \e^{\Phi_H(f^n(x))},\\[2mm]
    \|D_{f^n(x)}f^{H}(v_1)\|_{f^{n+H}(x)} &\ls K \cdot \e^{-H\chi  + |n|\epsilon}\cdot \e^{\Phi_{H}(f^n(x))}.
\end{align*}
    Consequently,
    $$\frac{\|D_{f^n(x)}f^H(v_{2})\|_{f^{n+H}(x)}}{\|D_{f^n(x)}f^H(v_{1})\|_{f^{n+H}(x)}} 
    \gs K^{-2}\cdot \e^{2H(\chi-\epsilon)-2|n|\epsilon}.$$
    Taking logarithms, and using $\epsilon\ls100^{-1}\chi$ from \Cref{var},
\begin{align*}
    \log \frac{\|D_{f^n(x)}f^H(v_{2})\|_{f^{n+H}(x)}}{\|D_{f^n(x)}f^H(v_{1})\|_{f^{n+H}(x)}}-H\epsilon
    &\gs 2H(\chi-\epsilon)-2|n|\epsilon-2\log K-H\epsilon\\[2mm]
    &\gs H(2\chi-3\epsilon)-2(|n|\epsilon+\log K)\\[2mm]
    &\gs \bigl(\bigl[\tfrac{2|n|\epsilon+2\log K}{\chi}\bigr]+1\bigr)\cdot\chi-2(|n|\epsilon+\log K)\\[2mm]
    &\gs 0.
\end{align*}
   This implies
   $$ \frac{\|D_{f^n(x)}f^H(v_{2})\|_{f^{n+H}(x)}}{\|D_{f^n(x)}f^H(v_{1})\|_{f^{n+H}(x)}} \gs \e^{H\epsilon},$$
   which contradicts the earlier inequality. Hence the proposition is proved.
\end{proof}

\subsection{Weighted Pesin charts}\label{sec new 2.3}
    In this section, we construct the weighted Pesin charts, which are parallel to  the Pesin charts in classical Pesin theory.
    
    We begin by specifying the constants used in \Cref{sec new 2.3}. Let $C(\epsilon) := (\epsilon - \epsilon e^{-2\epsilon})^{-1}$, and let $C_i$ ($i \in \mathbb{N}$) denote suitably large constants that are universal, independent of $f$ and non-decreasing with respect to the index $i$. Their explicit values are given in Appendix \ref{appendix a2}. See there for more details.
    
    Now define $$\Lambda^*:=\bigcup_{i\in\mathbb{Z}}f^i\Lambda^\rk(K,\chi,\epsilon,\varphi).$$ Similar to the classical Pesin theory, we will introduce a new norm $\|\cdot\|'_x$ on this invariant set $\Lambda^*$. Given any $x\in \Lambda^*$, for any $v,w\in T_xM$ we write $v=v_u+v_{cs}$ and $w=w_u+w_{cs}$ with $v_u,w_u\in E^{u}(x)$ and $v_{cs},w_{cs}\in E^{cs}(x)$. A new inner product $\langle \cdot,\cdot\rangle'$ on $T_xM$ is defined as follows:
\begin{align*}
    \langle v_{u}, w_u\rangle'_{x}
    &:=\sum_{n=0}^{+\infty}\e^{2n(\chi-2\epsilon)} \e^{-2\Phi_{-n}(x)}\langle D_{x}f^{-n}v_{u},D_{x}f^{-n}w_{u}\rangle_{f^{-n}(x)},\\[2mm]
    \langle v_{cs}, w_{cs}\rangle'_{x}  
    &:=\sum_{n=0}^{+\infty}\e^{2n(\chi-2\epsilon)} \e^{-2\Phi_{n}(x)}\langle D_{x}f^{n}v_{cs},D_{x}f^{n}w_{cs}\rangle_{f^{n}(x)},\\[2mm]
    \langle v, w\rangle '_{x}
    &:=\langle v_{u}, w_u\rangle'_{x}+\langle v_{cs}, w_{cs}\rangle'_{x}.
  \end{align*}
   The norm $\|\cdot\|'_x$ induced by this inner product $\langle \cdot, \cdot \rangle '_{x}$ is called the weighted Lyapunov norm.
   
   The following lemma characterizes the relationship between the weighted Lyapunov norm $\|\cdot\|'_x$ and the original Riemannian norm $\|\cdot\|_x$. Note that the weighted Lyapunov norm enables one-step length contraction or expansion properties along the iteration of $f$. See Appendix \ref{appendix a2} for a detailed proof.

\begin{lemma}\label{lem 2.5}
Let $x\in\bigcup_{|i|\ls m}f^i\Lambda^\rk(K,\chi,\epsilon,\varphi)$. For any $v\in T_xM$, $v_{cs}\in E^{cs}(x)$ and $v_u\in E^u(x)$, the following properties hold:
\begin{align*}
    \frac{1}{2}\|v\|_x&\ls\|v\|'_{x}\ls  \frac{C_1}{2}\cdot C(\epsilon)\cdot\|f\|^{2}_{C^1}\cdot (K\e^{m\epsilon})^{\tau+1}\|v\|_x,\\[2mm]
    (\|f\|_{C^1}+\e^{\varphi(f^{-1}(x))+\chi-2\epsilon})^{-1}&\ls\frac{\|D_{x}f^{-1}(v_{u})\|'_{f^{-1}(x)}}{\|v_{u}\|'_{x}}\ls  \e^{-(\varphi(f^{-1}(x))+\chi-2\epsilon)},  \\[2mm]
    \frac{1}{2}\|f\|_{C^1}^{-1}&\ls\frac{\|D_{x}f(v_{cs})\|'_{f(x)}}{\|v_{cs}\|'_{x}}\ls \e^{\varphi(x)-(\chi-2\epsilon)}.
\end{align*}
\end{lemma}

    To proceed, the return time function $\kappa(x):\Lambda^{*}\to \mathbb{N}$ is defined as
    $$\kappa(x):= \min\{m:x\in\bigcup_{|i|\ls m}f^i\Lambda^\rk(K,\chi,\epsilon,\varphi)\},$$ 
    and the hyperbolicity function is defined as $l(x):\Lambda^{*}\to [1,+\infty)$ as
    $$l(x):= K\cdot \e^{\kappa(x)\epsilon}.$$
    Let $L_x:T_xM\to\mathbb{R}^{\mathrm{d}}$ be the linear map sending $E^u(x)$ to $\mathbb{R}^{\rk}\times\{0\}$ and $E^{cs}(x)$ to $\{0\}\times\mathbb{R}^{cs}$ such that for any $v,u\in T_xM$,
    $$\langle L_xv, L_xu\rangle=\langle v, u\rangle'_{x}.$$
    For any $w=w_u +w_{cs}$ with $w\in\mathbb{R}^{\mathrm{d}}$, $w_u\in\mathbb{R}^{\rk}\times\{0\}$ and $w_{cs}\in\{0\}\times\mathbb{R}^{cs}$, we define the box norm $|\cdot|$ on $\mathbb{R}^{\mathrm{d}}$ by
    $$|w|=\max\{|w_u|,\, |w_{cs}|\},$$
    and for any $x\in\Lambda^*$ we set
    $$R_x(r):=\{v:|v|\ls r,\,v\in\mathbb{R}^{\mathrm{d}}\},$$
    and $R_x:=R_x(\infty)$. We denote the box norm on $R_x$ by $|\cdot|_x$.
    
If $x\in\Lambda^*$, then we have the following estimate which controls the norm and the conorm of the operator $L_x$. See Appendix \ref{appendix a2} for details.
\begin{align}\label{lx norm}
   \frac{1}{4}\ls \m(L_x)\ls\|L_x\|\ls  C_1\cdot C(\epsilon)\cdot\|f\|^{2}_{C^1} \cdot l(x)^{\tau+1}.
\end{align}

   Denote $\Psi_x:\mathbb{R}^{\mathrm{d}}\to M$ by 
   $$\Psi_x:= \exp_x\circ L_x^{-1}$$
   whenever it is well-defined.
   The lift of $f$ by $\Psi_\cdot$ along the orbit is denoted by $\tilde{f}^{\pm1}_x$ whenever it is well-defined:
   $$\tilde{f}_x:= \Psi^{-1}_{f(x)}\circ f\circ\Psi_x
    =L_{f(x)}\circ\exp_{f(x)}^{-1}\circ f \circ \exp_x\circ L_x^{-1},$$
    $$\tilde{f}_x^{-1}
    := \Psi^{-1}_{f^{-1}(x)}\circ f^{-1}\circ\Psi_x
    =L_{f^{-1}(x)}\circ\exp_{f^{-1}(x)}^{-1}\circ f^{-1} \circ \exp_x\circ L_x^{-1}.$$
    Note that there exists a uniform positive number $\delta$ such that for any $x\in M$, $$\exp_x: \{v:\|v\|_x<\delta,\ v\in T_xM\}\to M$$ is a ${C}^{\infty}$ diffeomorphism onto its image with $\|{D}\exp_x\|$ and $\|{D}\exp_x^{-1}\|\ls2$. This means that $\tilde{f}_x$ and $\tilde{f}_{x}^{-1}$ are well-defined on $R_{x}(r_0)\subset \mathbb{R}^{\mathrm{d}}$, where $r_0$ is defined by 
    $$r_0:= \frac{\delta}{C_2\cdot\|f\|_{C^{1}}},$$
    and $C_2$ is chosen large enough. Then by \Cref{lx norm}, the $C^1$-norms of $\Psi_x$ on $R_{x}(r_0)$ and of $\Psi_x^{-1}$ on $\Psi_x(R_{x}(r_0))$ can be estimated as:
\begin{align}\label{psi 1}
    \|\Psi_x|_{R_{x}(r_0)}\|_{C^1}
    \ls  C_3
\end{align}
and
\begin{align}\label{psi -1}
    \|\Psi_x^{-1}|_{\Psi_x(R_{x}(r_0))}\|_{C^1}
    \ls  C_3\cdot C(\epsilon)\cdot\|f\|^{2}_{C^1} \cdot l(x)^{\tau+1}.
\end{align}

    Now we extend $\tilde{f}_x$ and $\tilde{f}_{x}^{-1}$ from the local domain $R_{x}(r(x))$ to the whole space $\mathbb{R}^{\mathrm{d}}$, where $r(x):\Lambda^{*}\to(0,1)$ is a function defined as
    \begin{equation}\label{charts size}
       r(x):= \min\Big\{\frac{r_0}{100}, \Big(\frac{\epsilon}{4C_6C(\epsilon)^3\e^{2\tau\epsilon}\|f\|^{10}_{C^{1+\alpha}}l(x)^{4\tau}}\Big)^{\frac{1}{\alpha}}\Big\}.
    \end{equation}
    Choose a $C^\infty$ bump function $\rho:\mathbb R\rightarrow\mathbb R$ such that:
\begin{equation*}
	\rho(z)=\left\{
	\begin{aligned}
		1 & , &|z|&<1\\
		0 & , & |z|&>2
	\end{aligned}
	\right.
\end{equation*}
and $\rho(z) \in [0,1]$ for all $z\in \mathbb R$.

We define $\tg_x:\mathbb{R}^{\mathrm{d}}\to \mathbb{R}^{\mathrm{d}}$ and $\tg^{-1}_{x}:\mathbb{R}^{\mathrm{d}}\to \mathbb{R}^{\mathrm{d}}$ as:
\begin{align*}
  \tg_x(v)
  :=\rho(\frac{\|v\|_x}{r(x)})\cdot\tilde{f}_x(v)&+(1-\rho(\frac{\|v\|_x}{r(x)}))\cdot D\tilde{f}_x(0)v,\\[2mm]
  \tg^{-1}_{x}(v):= &(\tg_{f^{-1}(x)})^{-1}(v).
\end{align*}
Then the following lemma estimates the Lipschitz constants of the maps $\tg_x-D\tilde{f}_x(0)$ and $\tg^{-1}_{x}-D\tilde{f}_{x}^{-1}(0)$ with respect to $v$. See Appendix \ref{appendix a2} for details.
\begin{lemma}\label{lem 2.6}
    $$\max\{\mathrm{Lip}(\tg_x-D\tilde{f}_x(0))\ ,\ \mathrm{Lip}(\tg^{-1}_{x}-D\tilde{f}_{x}^{-1}(0))\}\ls\frac{\epsilon}{4\|f\|_{C^1}\cdot l(x)}.$$
\end{lemma}

Now we define the cone spaces on each weighted Pesin chart. For each weighted Pesin chart $(R_x,\Psi_x)$ with the standard splitting $\mathbb{R}^{\rk}\oplus\mathbb R^{cs}$ as before, we define the center-stable cone and the unstable cone at any point $v\in R_x$ and for any $a>0$ as follows:
$$\mathrm{Q}_x^u(v,a):= \{w\in T_vR_x:|w_2|\ls a|w_1|,\ w=w_1+w_2, \ w_1\in \mathbb{R}^{\rk}, \ w_2\in \mathbb R^{cs}\},$$
$$\mathrm{Q}_x^{cs}(v,a):= \{w\in T_vR_x:|w_2|\ls a|w_1|,\ w=w_1+w_2, \ w_1\in \mathbb R^{cs}, \ w_2\in \mathbb{R}^{\rk}\}.$$
Note that on the tangent space we still use the same metric $|\cdot|_x$ as on $R_x$. The next lemma states that if $\mathrm{Lip}(\tg_x-D\tilde{f}_x(0))$ and $\mathrm{Lip}(\tg^{-1}_{x}-D\tilde{f}_{x}^{-1}(0))$ are sufficiently small, then $\tg_x$ and $\tg^{-1}_{x}$ preserve the cones. See Appendix \ref{appendix a2} for a detailed proof.
\begin{lemma}\label{cone pre}
Using the notation above, suppose that $$\max\{\mathrm{Lip}(\tg_x-D\tilde{f}_x(0)),\ \mathrm{Lip}(\tg^{-1}_{x}-D\tilde{f}_{x}^{-1}(0))\}<\frac{\epsilon}{4\|f\|_{C^1}\cdot l(x)}.$$ Then it follows that
\begin{itemize}
\item[(i)]\emph{(Cone-preserving)} For any $w\in R_x$:
\begin{align*}
    D\tg_x(w)( \mathrm{Q}^u_{x}(w,\frac{1}{3}))
    &\subseteq \mathrm{Q}^u_{{f(x)}}(\tg_xw,\frac{1}{3}),\\[2mm]
    D\tg_{x}^{-1}(w)(\mathrm{Q}^{cs}_{{x}}(w,\frac{1}{3}))
    &\subseteq \mathrm{Q}^{cs}_{{f^{-1}(x)}}(\tg_x^{-1}w,\frac{1}{3}).
\end{align*}
\item [(ii)]\emph{(Hyperbolicity)} For any $w\in R_x$, any $v_1\in \mathrm{Q}^u_{x}(w,\frac{1}{3})$ and any $v_2\in \mathrm{Q}^{cs}_{x}(w,\frac{1}{3})$:
\begin{align*}
    |D\tg_x(w) (v_1)|_{f(x)}&\gs \e^{\varphi(x)+\chi-3\epsilon}\cdot|v_1|_x,\\[2mm]
    |D\tg^{-1}_x(w) (v_2)|_{f^{-1}(x)}&\gs \e^{-\varphi(f^{-1}(x))+\chi-3\epsilon}\cdot|v_2|_x.
\end{align*}
\end{itemize}

\end{lemma}

In summary, for any point $x$ inside a given weighted Pesin block, we prove the existence of a local neighborhood $R_x(r(x))$ whose size $r(x)$ is governed by the hyperbolicity function $l(x)$ such that after a suitable normalization $\Psi_x$, the map $\tilde{f}_x$ can be viewed as a small perturbation of $D\tilde{f}_x(0)$. This is achieved by employing a smooth bump function $\rho$ to construct a global diffeomorphism $\tg_x$ on $R_{f^m(x)}\simeq \mathbb{R}^{\mathsf{d}}$ that remains close to $D\tilde{f}_x(0)$. By the graph transformation induced by the sequence of diffeomorphisms $\{\tg_{f^m(x)}\}_{m\in\mathbb{Z}}$ over the sequence of base spaces $\{R_{f^m(x)}\}_{m\in\mathbb{Z}}$:
$$\tilde{g}_{f^m(x)}:\,R_{f^m(x)}\to R_{f^{m+1}(x)},$$
we can obtain invariant fake foliation families $\{\widetilde{\mathcal{F}}_{f^m(x)}\}_{m\in\mathbb{Z}}$, which is the content of the following proposition. We also denote $\mathbb R^{u}:=\mathbb R^\rk$. See Appendix \ref{appendix a2} for details.
\begin{proposition}\label{pro 2.8.2}
For any $x\in\Lambda^*$, there exist unique global fake foliations $\widetilde{\mathcal{F}}_x^{i}$ on $R_x=\mathbb{R}^{\mathrm{d}}$ with $C^{1+\mathrm{H\"older}}$ leaves for $i\in\{u,cs\}$, such that for any $y\in \mathbb{R}^{\mathrm{d}}$,
\begin{enumerate}
    \item \emph{Almost tangency}: The unique leaf $\widetilde{\mathcal{F}}_x^{i}(y)$ containing $y$ is the graph of a $C^{1+\mathrm{H\"older}}$ map $\phi:\mathbb{R}^{i}\to \mathbb{R}^{\hat i}$ with $\|D\phi\| \ls \frac {1}{3}$, where $\hat i\in\{cs,u\}$ and $\mathbb{R}^i\oplus\mathbb{R}^{\hat{i}}= R_x.$\smallskip
    \item \emph{Invariance}: The foliations are invariant under $\tg_{(\cdot)}$, in the sense that
    $$\tg_x(\widetilde{\mathcal{F}}_x^{i}(y))=\widetilde{\mathcal{F}}_{f(x)}^{i}(\tg_x(y)).$$
    \item \emph{Hyperbolicity}: For any leaf $\widetilde{\mathcal{F}}_x^{u}(y)$ and any $z_1,z_2\in\widetilde{\mathcal{F}}_x^{u}(y)$
    \begin{align*}
        |\tg^{-1}_x z_1-\tg^{-1}_xz_2|_{f^{-1}(x)}\ls \e^{-\varphi(f^{-1}(x))-\chi+3\epsilon}\cdot|z_1-z_2|_{x}.
    \end{align*}
   For any leaf $\widetilde{\mathcal{F}}_x^{cs}(y)$ and any $z_1,z_2\in\widetilde{\mathcal{F}}_x^{cs}(y)$
    \begin{align*}
        |\tg_xz_1-\tg_xz_2|_{f(x)}\ls \e^{\varphi(x)-\chi+3\epsilon}\cdot|z_1-z_2|_{x}.
    \end{align*}
\end{enumerate}
\end{proposition}

\subsection{Weighted Pesin unstable manifolds}\label{sec 2.3}
    Recall that $\Gamma$ denotes the set of regular points.
    For any $x\in\Gamma$ such that $\lambda_\rk(x,f)>\max\{\lambda_{\rk+1}(x,f),0\}$, the classical Pesin unstable manifold of dimension $\rk$ is defined by
    $$W^{\rk}(x)=\{ y\in M: \limsup_{n\rightarrow \infty}\frac{1}{n}\log \td(f^{-n}(x),f^{-n}(y))\ls-\lambda_{\rk}(x,f)+\epsilon,\ \forall \epsilon>0\},$$
    where $\td$ is the distance induced by the Riemannian metric on $M$. By classical Pesin theory, $W^\rk(x)$ is a $\rk$-dimensional injectively immersed $C^1$ manifold.

    In this section, we study the invariant manifolds obtained from the weighted Pesin charts; these are called the weighted Pesin unstable manifolds. We first define the local weighted Pesin unstable manifolds $W^{\mathrm{GPu}}_{\mathrm{loc}}(x)\subseteq M$ for any $x\in\Lambda^*$:
    \begin{equation}\label{defintion of local GPu}
    W^{\mathrm{GPu}}_{\mathrm{loc}}(x):=\Psi_x\left(R_{x}(r(x))\cap\widetilde{\mathcal{F}}_x^{u}(0)\right).
    \end{equation}
    The following lemma shows that \Cref{var} implies that the local weighted Pesin unstable manifolds stay in the weighted Pesin charts $(R_x(r(x)),\Psi_x)$ under backward iteration of $\tg_{(\cdot)}$.  See Appendix \ref{appendix a3} for a detailed proof.  
\begin{lemma}\label{lem 2.9}
    For any $x\in\Lambda^*$, $$\tg^{-1}_x\left(R_{x}(r(x))\cap\widetilde{\mathcal{F}}_x^{u}(0)\right)\subset R_{f^{-1}(x)}\left(r(f^{-1}(x))\right)\cap\widetilde{\mathcal{F}}_{f^{-1}(x)}^{u}(0).$$
\end{lemma}

    Since $\tg_{(\cdot)}=\tf_{(\cdot)}$ whenever the orbit stays inside the weighted Pesin charts $(R_x(r(x)),\Psi_x)$, we immediately obtain the following lemma by checking the definition.
\begin{lemma}
    For any $x\in\Lambda^*$, $f^{-1}(W^{\mathrm{GPu}}_{\mathrm{loc}}(x))  \subseteq W^{\mathrm{GPu}}_{\mathrm{loc}}(f^{-1}x).$
\end{lemma}
    We now turn to the definition of the global weighted Pesin unstable manifolds. For any $x\in\Lambda^*$, the global weighted Pesin unstable manifold is defined by
    $$W^{\mathrm{GPu}}(x):= \Big\{y\in M:\limsup_{n\to +\infty} \frac{1}{n}\log(\frac{ d(f^{-n}(x),f^{-n}(y))}{\e^{\Phi_{-n}(x)}})\ls-\frac{1}{2}\chi\Big\}.$$
    Since $\varphi\in C^0(M)$ and $\varphi\gs0$, it follows that $d(f^{-n}x,f^{-n}y)\to 0$ exponentially for any $y\in W^{\mathrm{GPu}}(x)$ as $n\to\infty$. Consequently,
    $$\frac{1}{n}\sum_{i=1}^{n}|\varphi(f^{-i}(x))-\varphi(f^{-i}(y))|\to0.$$
    Thus, for any $x,y\in \Lambda^*$, $y\in W^{\mathrm{GPu}}(x)$ if and only if $x\in W^{\mathrm{GPu}}(y)$. Recall that in classical Pesin theory there is a standard relation between global and local unstable manifolds:
    $$W^u(x)=\bigcup_{n=0}^{+\infty}f^n(W^u_{loc}(f^{-n}(x))).$$
    The following proposition shows that this relation also holds for the weighted Pesin unstable manifolds. See Appendix \ref{appendix a3} for a detailed proof.
\begin{proposition}\label{gpu pro}
    For any $x\in\Lambda^*$,
     $$W^{\mathrm{GPu}}(x)=\bigcup_{n=0}^{+\infty}f^n(W^{\mathrm{GPu}}_{\mathrm{loc}}(f^{-n}(x))).$$
\end{proposition}

   Recall that $\Gamma$ is the set of regular points. We further define 
   $$\Gamma_{\rk}:=\Gamma\cap\{x:\lambda_{\rk}(x,f)>\max\{\lambda_{\rk+1}(x,f),0\}\}.$$
   Thus, $\Gamma_\rk$ has full measure for any invariant measure $\mu$ of $f$ with $\mathrm{Hyp^{\mathrm{k}}(\mu)>0}$.
   Next, we define the $\varphi$-Birkhoff regular set $B_\varphi$ by
   $$B_\varphi:=\{x\in M:\lim_{n\to +\infty} \frac{1}{n}\sum_{i=1}^{n}\varphi(f^{-i}(x))\ \text{exists}\}.$$
   By the Birkhoff Ergodic Theorem, $B_\varphi$ has full measure for any $f$-invariant measure. 
   
    The following lemma shows that on the set $\Lambda^*\cap\Gamma_\rk\cap B_\varphi$, which has full measure for any invariant measure $\mu$ of $f$ with $\mathrm{Hyp^{\mathrm{k}}(\mu)>0}$, the weighted Pesin unstable manifold $W^{\mathrm{GPu}}(x)$ coincides with the classical Pesin unstable manifold $W^\rk(x)$. 
\begin{lemma}\label{lem 2.12}
    For any $x\in\Gamma_{\rk}\cap B_{\varphi}\cap\Lambda^*$,
    $$W^{\mathrm{GPu}}(x)=W^\rk(x).$$
\end{lemma}
\begin{proof}
    By the definitions of $B_\varphi$ and $W^{\mathrm{GPu}}(x)$, we obtain 
    $$W^{\mathrm{GPu}}(x)= \Big\{y:\limsup_{n\to +\infty} \frac{1}{n}\log \mathrm{d}(f^{-n}(x),f^{-n}(y))\ls-(C(x)+\frac{1}{2}\chi),\ y\in M\Big\},$$
    where $x\in\Gamma_{\rk}\cap B_{\varphi}\cap\Lambda^*$ and
    $$C(x)=\lim_{n\to +\infty} \frac{1}{n}\sum_{i=1}^{n}\varphi(f^{-i}(x)).$$
    Since $\dim W^{\mathrm{GPu}}(x)=\rk=\dim W^\rk(x)$, the definition of Pesin unstable manifolds yields $W^{\mathrm{GPu}}(x)=W^\rk(x).$
\end{proof}
\begin{remark}
An important fact is that the weighted Pesin
unstable manifolds are defined at every point of a weighted Pesin block, not only
on the Oseledets regular set. This pointwise construction is needed for the
closedness and upper-limit arguments for the unstable laminations in \Cref{sec 3}, while \Cref{lem 2.12} ensures
that it agrees with the classical Pesin unstable manifold $W^\rk$ on the full-measure regular set.
\end{remark}

\subsection{Quantitative uniform size estimate}\label{sec 2.4}
In this section, we prove that the continuity of the sum of Lyapunov exponents implies the uniform large size of weighted Pesin blocks with uniform parameters. Furthermore, a quantitative formulation of this estimate is provided. Recall that for $\rk\in\mathbb N$, $$\mathrm{Hyp^{\mathrm{k}}(\mu)}:=\essinf_{x\sim\mu}\{\lambda_\rk(x,f)-\max\{\lambda_{\rk+1}(x,f),0\}\}.$$

Now we introduce the concept of the upper limit of a sequence of sets, which is slightly different from the classical ones.
Let $S_n\subseteq M,\ n\in\mathbb N$ be a sequence of sets. We define the upper limit of $S_n$ by
\[
\limsup_{n\to\infty}S_n:=\{x\in M:\exists\ x_n\in S_{k_n}\ \text{ s.t.}\ k_n\gs n \text{ and } \lim _{n\to\infty}x_n=x \}.
\]
By definition, the set $\limsup_{n\to\infty}S_n$ can be shown to be closed, and is therefore compact.

\begin{proposition}\label{uniform size prop}
    Let $f_n\in\mathrm{Diff^{1+\alpha}}(M)$ and $\mu_n$ an ergodic measure of $f_n$. Assume
    \begin{itemize}
    \item  $f_n\xrightarrow{C^1} f\in \mathrm{Diff^{1+\alpha}(M)}$
    \item and $\mu_n\xrightarrow{*}\mu$. 
    \end{itemize}
    If there exists $\rk\in\mathbb N$ such that
\begin{itemize}
    \item $\mathrm{Hyp}^{\mathrm{k}}(\mu)>0,$ 
    \item $\sum_{i=1}^\rk\lambda_i(\mu_n,f_n)\xrightarrow{n\to\infty}\sum_{i=1}^\rk\lambda_i(\mu,f),$
\end{itemize}
then for any $\rho>0$ there exist $N\in\mathbb N$, $K\gs1$, $\chi>0$ and $\epsilon>0$ satisfying \Cref{var} and a positive function $\varphi\in C^0(M)$ such that there exist $$\Lambda^\rk_n=\Lambda^\rk_n(K^N,N\chi,N\epsilon,N\varphi)$$ as a weighted Pesin block for $f_n^N$ for any $n$ large enough and $$\Lambda^\rk=\Lambda^\rk(K^N,N\chi,N\epsilon,N\varphi)$$ as a weighted Pesin block for $f^N$ satisfying
\begin{itemize}
    \item $\limsup_{n\to\infty}\Lambda_n^\rk =\Lambda^\rk$,
    \item $\mu_n(\Lambda_n^\rk )>1-\rho \ \text{for any } n\in\mathbb N$ large enough,
    \item $\mu(\Lambda^\rk)>1-\rho$.
\end{itemize}
\begin{remark}
The converse of \Cref{uniform size prop} also holds, thus establishing an equivalence between its condition and its conclusion. Indeed, observe that the continuity of the unstable bundle $E^u(y)$ on weighted Pesin blocks directly implies the continuity of the integrand $y \mapsto \log \|\wedge^\rk(D_y f)|_{E^u(y)}\|$ on such sets. Therefore, the existence of a family of such blocks whose measures can arbitrarily approach $1$ is sufficient to force the convergence of the corresponding integrals under the weak$^*$ convergence of measures, thereby establishing the continuity of the sum of the top $\rk$ Lyapunov exponents. We leave the details to the reader. 
\end{remark}
\begin{proof}
    We note that in the proof of this proposition, the notation $E^{u}(x)$ and $E^{cs}(x)$ denote the subbundles of the Oseledets splitting corresponding to exponents $\lambda_1(x,f)\gs\cdots\gs\lambda_\rk(x,f)$ and $\lambda_{\rk+1}(x,f)\gs\cdots\gs\lambda_\td(x,f)$ respectively, instead of the decomposition given by the weighted Pesin blocks. 

    Denote  $\Upsilon_1=\sup_n\{\|f_n\|_{C^1}\}$, and we prove this proposition with parameters \[K=\Upsilon_1^3,\ \chi=\frac{\mathrm{Hyp^{\mathrm{k}}(\mu)}}{5} \text{ and } \epsilon= \frac{\alpha\cdot\chi^2}{10^3\cdot(\log\Upsilon_1+1)}.\] 
    
    Choose $\rho_0>0$ and $\epsilon'>0$ small enough that will be determined later. Since $\mathrm{Hyp}^{\mathrm{k}}(\mu)>0$, there exists a measurable function defined as $$\lambda_0(x,f)=\frac{\lambda_\rk(x,f)+\max\{\lambda_{\rk+1}(x,f),0\}}{2}$$  such that for $\mu-a.e.\ x\in M$ we have
    $$\lambda_\rk(x,f)>\lambda_0(x,f)+2\chi>\lambda_0(x,f)-2\chi>\max\{\lambda_{\rk+1}(x,f),0\},$$
    and in particular we have $0\ls\lambda_0(x,f)\ls \log\Upsilon_1$.

    By Lusin's Theorem we can choose a compact set $\Gamma_0\subseteq M$ such that $\mu(\Gamma_0)>1-\rho_0$, and $E^{cs}(x), E^{u}(x)$ and $\lambda_i(x,f),\ 0\ls i\ls \td$ are both continuous for $x\in\Gamma_0$. Now consider the splitting $T_{\Gamma_0}M=E^{cs}\oplus E^u$, by the continuity of $E^{cs}$ and $E^u$ there exists a positive number $\sigma>0$ such that $\angle (E^{cs}(x),E^{u}(x))>100\sigma$ for any $x\in\Gamma_0$.


   For any $x \in M$, let $\measuredangle(\cdot, \cdot)$ denote the standard distance in the Grassmannian manifold of subspaces of $T_xM$. To compare subspaces at distinct points, for any $y \in M$ within the injectivity radius of $x$, $T_yM$ can be identified with $T_xM$ via the linear map $\iota_{y,x} := D_y \exp_x^{-1}: T_yM \to T_xM$, which geometrically represents the parallel transport along the unique geodesic connecting $y$ to $x$. Given a subspace $E_y \subseteq T_yM$, we denote its image under this identification by $\hat{E}_y := \iota_{y,x}(E_y) \subseteq T_xM$. The distance between $E_x \subseteq T_xM$ and $E_y \subseteq T_yM$ is then defined by $\measuredangle(E_x, E_y) := \measuredangle(E_x, \hat{E}_y)$.
   
   By the Oseledets Multiplicative Ergodic Theorem, one can choose $\epsilon'>0$ small enough such that there exist a integer $N>0$ and a compact set $\Gamma_1\subseteq\Gamma_0$ such that $\mu(\Gamma_1)>1-2\rho_0$, and for any $x\in\Gamma_1$, $n\gs N$, the following properties hold:
    \begin{itemize}
        \item[1 (cs).] 
        \begin{align*}
            \|Df^n|_{E^{cs}(x)}\|&\ls \e^{(\lambda_0(x,f)-2\chi-\epsilon')\cdot n},\\[2mm]
             \|\wedge^{\td-\rk} (Df^n)|_{E^{cs}(x)}\|&\ls \ \e^{(\sum_{i=\rk+1}^d\lambda_i(x,f)+\epsilon')\cdot n}.
        \end{align*}
        \item[1  (u).]  
        \begin{align*}
       \|Df^{-n}|_{E^{u}(x)}\|&\ls \e^{(-\lambda_0(x,f)-2\chi-\epsilon')\cdot n}, \\[2mm]
        \|\wedge^\rk (Df^{-n})|_{E^u(x)}\|&\ls \e^{(-\sum_{i=1}^\rk \lambda_i(x,f)+\epsilon')\cdot n}.
        \end{align*}
        \item[2 (cs).]
        For any $\td-\rk$ dimensional linear subspace $V\subseteq T_xM$,
        \[
        \|\wedge^{\td-\rk}(Df^n)|_V\|\gs \e^{(\sum_{i=\rk+1}^\td\lambda_i(x,f)-\epsilon')\cdot n}.
        \]
        If $\measuredangle(V,E^{cs}(x))>\sigma$, then\[
        \|\wedge^{\td-\rk}(Df^n)|_{V}\|\gs\e^{(\sum_{i=\rk+1}^\td\lambda_i(x,f)+4\chi-\epsilon')\cdot n}.
        \]
        If $\measuredangle(V,E^{cs}(x))\ls\sigma$, then\[
        \m(Df^{-n}|_{V})\gs\e^{(-\lambda_0(x,f)+2\chi+\epsilon')\cdot n}.
        \]
        \item[2   (u).]
        For any $\rk$-dimensional linear subspace $V\subseteq T_xM$,
        \[
        \|\wedge^{\rk}(Df^{-n})|_V\|\gs \e^{(-\sum_{i=1}^\rk\lambda_i(x,f)-\epsilon')\cdot n}.
        \]
        If $\measuredangle(V,E^{u}(x))>\sigma$, then\[
        \|\wedge^{\rk}(Df^{-n})|_{V}\|\gs\e^{(-\sum_{i=1}^\rk\lambda_i(x,f)+4\chi-\epsilon')\cdot n}.
        \]
        If $\measuredangle(V,E^{u}(x))\ls\sigma$, then\[
        \m(Df^{n}|_{V})\gs\e^{(\lambda_0(x,f)+2\chi+\epsilon')\cdot n}.
        \]
    \end{itemize}
Then we fixed an $N\in\mathbb N$ large enough. By Tietze Extension Theorem we can choose a continuous extension $\varphi\in C^0(M)$ of $\lambda_0\in C^0(\Gamma_1)$ such that $0\ls\varphi\ls\log\Upsilon_1$. In the same manner, for any $1\ls i\ls \td$, one can choose $\bar\lambda_i\in C^0(M)$ as the continuous extension of $\lambda_i\in C^0(\Gamma_1)$ such that $-\log\Upsilon_1\ls\bar\lambda_i\ls\log\Upsilon_1$.
Then by the continuity of these functions, for any $x\in\Gamma_1$ there exists an open neighborhood $U_x$ of $x$ such that for any $y\in U_x$ and $n$ large enough, the following properties hold:
\begin{itemize}
    \item[A (cs).]
    For any $\td-\rk$ dimensional linear subspace $V\subseteq T_yM$, \[
    \|\wedge^{\td-\rk}(Df^N_n)|_V\|\gs\e^{(\sum_{i=\rk+1}^\td\bar \lambda_i(y)-2\epsilon')\cdot N},
    \]
    if $\measuredangle(V,E^{cs}(x))>\sigma$,\[
    \|\wedge^{\td-\rk}(Df^N_n)|_V\|\gs\e^{(\sum_{i=\rk+1}^\td\bar \lambda_i(y)+4\chi-2\epsilon')\cdot N}.
    \]
    \item[B (cs).]
    For any $\td-\rk$ dimensional linear subspace $V\subseteq T_yM$, if $\measuredangle(V,E^{cs}(x))\ls\sigma$,\[
    \m(Df^{-N}_n|_V)\gs\e^{(-\varphi(y)+2\chi)\cdot N}.
    \]
    \item[A  (u).]
    For any $\rk$-dimensional linear subspace $V\subseteq T_yM$,\[
    \|\wedge^{\rk}(Df_n^{-N})|_V\|\gs\e^{(-\sum_{i=1}^\rk\bar \lambda_i(y)-2\epsilon')\cdot N},
    \]
    if $\measuredangle(V,E^{u}(x))>\sigma$,\[
    \|\wedge^{\rk}(Df_n^{-N})|_V\|\gs\e^{(-\sum_{i=1}^\rk\bar \lambda_i(y)+4\chi-2\epsilon')\cdot N}.
    \]
    \item[B  (u).]
    For any $\rk$-dimensional linear subspace $V\subseteq T_yM$, 
    if $\measuredangle(V,E^{u}(x))\ls\sigma$,\[
    \m(Df^{N}_n|_V)\gs\e^{(\varphi(y)+2\chi)\cdot N}.
    \]
\end{itemize}

By the compactness of $\Gamma_1$, we can choose finitely many $x_j\in\Gamma_1$ such that $\{U_{x_j}\}_{j=1}^q$ forms a finite open cover of $\Gamma_1$. We denote $U:=\bigcup_{j=1}^qU_{x_j}$, which is an open neighborhood of the compact set $\Gamma_1$. Thus, when $n$ is large enough, $\mu_n(U)\gs 1-3\rho_0$.

By the ergodicity of $\mu_n$ and the continuity properties of Lyapunov exponents, since $\mathrm{Hyp}^{\rk}(\mu)>0$, it follows that $\mathrm{Hyp}^{\rk}(\mu_n)>0$ for all sufficiently large $n$. Under this condition, the Oseledets splitting $T_x M = E^{cs}(x) \oplus E^u(x)$ with $\dim E^u(x) = \rk$ and $\dim E^{cs}(x) = \td-\rk$ is $\mu_n$-almost everywhere well-defined. Then the \textit{center-stable lift} $\tau^{cs}_{\mu_n} \in \mathcal{P}(\mathcal{G}_{\td-\rk}(TM))$ and the \textit{unstable lift} $\tau^u_{\mu_n} \in \mathcal{P}(\mathcal{G}_{\rk}(TM))$ of $\mu_n$ are defined as probability measures on the respective Grassmannian bundles, characterized by the maps $\iota^{cs}: x \mapsto (x, E^{cs}(x))$ and $\iota^{u}: x \mapsto (x, E^{u}(x))$. Formally, $\tau^{cs}_{\mu_n}$ and $\tau^{u}_{\mu_n}$ are defined by 
\[
\tau^{cs}_{\mu_n}=\iota^{cs}_*(\mu_n),\quad \tau^{u}_{\mu_n}=\iota^{u}_*(\mu_n).
\] By previous discussions, these lifts are well-defined Borel probability measures. 

We further denote \[
B_{cs}:=\{(y,V):y\in U_{x_j}, V\subseteq T_yM \text{ and } \measuredangle(V,E^{cs}(x_j))>\sigma \text{ for some }1\ls j\ls q\},
\]
and\[
\bar B_{cs}:=\{y\in \Gamma_\rk:(y,E^{cs}(y))\in B_{cs}\}.
\]
Then
\begin{align*}
    N\cdot(\sum_{i=\rk+1}^\td \lambda_i(\mu_n,f_n))=&\int\log \|\wedge^{\td-\rk}(Df_n^N|_{(y,V)})\|\,\td\tau_{\mu_n}^{cs}(y,V)\\[2mm]
    =&\int_{(\wedge^{\td-\rk}T_UM)\setminus B_{cs}} \log\|\wedge^{\td-\rk}(Df_n^N|_{(y,V)})\|\,\td\tau_{\mu_n}^{cs}(y,V)\\[2mm]
    &+\int_{B_{cs}} \log\|\wedge^{\td-\rk}(Df_n^N|_{(y,V)})\|\,\td\tau_{\mu_n}^{cs}(y,V)\\[2mm]
    &+\int_{\wedge^{\td-\rk}T_{M\setminus U}M} \log \|\wedge^{\td-\rk}(Df_n^N|_{(y,V)})\|\,\td\tau_{\mu_n}^{cs}(y,V).
\end{align*}
Together with property A (cs), it follows that
\begin{align*}
    \sum_{i=\rk+1}^\td \lambda_i(\mu_n,f_n)\gs& \int_{(\wedge^{\td-\rk}T_UM)\setminus B_{cs}}\left(\sum_{i=\rk+1}^\td\bar \lambda_i(y)-2\epsilon'\right) \,\td\tau_{\mu_n}^{cs}(y,V)\\[2mm]
    &+\int_{B_{cs}} \left(\sum_{i=\rk+1}^\td\bar \lambda_i(y)+4\chi-2\epsilon'\right)\,\td\tau_{\mu_n}^{cs}(y,V)\\[2mm]
    &-3\rho_0(\td-\rk)\log\Upsilon_1\\[2mm]
    =&\int_{\wedge^{\td-\rk}T_UM}\left(\sum_{i=\rk+1}^\td\bar \lambda_i(y)-2\epsilon'\right) \,\td\tau_{\mu_n}^{cs}(y,V)\\[2mm]
    &+\tau^{cs}_{\mu_n}(B_{cs})\cdot 4\chi-3\rho_0(\td-\rk)\log\Upsilon_1\\[2mm]
    \gs&\int\left(\sum_{i=\rk+1}^\td\bar \lambda_i(y)-2\epsilon'\right) \,\td\tau_{\mu_n}^{cs}(y,V)\\[2mm]
    &+\tau^{cs}_{\mu_n}(B_{cs})\cdot 4\chi-6\rho_0(\td-\rk)\log\Upsilon_1.\\[2mm]
    =&\int\sum_{i=\rk+1}^\td\bar \lambda_i(y) \,\td\mu_n(y)-2\epsilon'+\tau^{cs}_{\mu_n}(B_{cs})\cdot 4\chi-6\rho_0(\td-\rk)\log\Upsilon_1.
\end{align*}
Since $\mu_n\xrightarrow{*}\mu$, when $n$ is large enough, we have 
\begin{align*}
    \int\sum_{i=\rk+1}^\td\bar \lambda_i(y) \,\td\mu_n(y)&\gs\int\sum_{i=\rk+1}^\td\bar \lambda_i(y) \,\td\mu(y)-\epsilon'\\[2mm]
    &\gs\int_{\Gamma_1}\sum_{i=\rk+1}^\td\lambda_i(y,f) \,\td\mu(y)-2\rho_0(\td-\rk)\log\Upsilon_1-\epsilon'\\[2mm]
    &\gs \sum_{i=\rk+1}^\td\lambda_i(\mu,f)-4\rho_0(\td-\rk)\log\Upsilon_1-\epsilon'.
\end{align*}
Together we obtain 
\begin{align*}
    \sum_{i=\rk+1}^\td \lambda_i(\mu_n,f_n)\gs\sum_{i=\rk+1}^\td\lambda_i(\mu,f)+\tau^{cs}_{\mu_n}(B_{cs})\cdot 4\chi-10\rho_0\td\log\Upsilon_1-3\epsilon'.
\end{align*}
Moreover, since $\tau^{cs}_{\mu_n}(B_{cs})=\mu_n(\bar B_{cs})$, we have
\[
\mu_n(\bar B_{cs})\ls\frac{\sum_{i=\rk+1}^\td \lambda_i(\mu_n,f_n)-\sum_{i=\rk+1}^\td\lambda_i(\mu,f)+10\rho_0\td\log\Upsilon_1+3\epsilon'}{4\chi}.
\]
In the same manner, if we denote
\[
B_{u}:=\{(y,V):y\in U_{x_j}, V\subseteq T_yM \text{ and } \measuredangle(V,E^{u}(x_j))>\sigma \text{ for some }1\ls j\ls q\},
\]
and\[
\bar B_{u}:=\{y\in \Gamma_\rk:(y,E^{u}(y))\in B_{u}\},
\]
then similarly we can obtain
\[
\mu_n(\bar B_{u})\ls\frac{\sum_{i=1}^\rk \lambda_i(\mu,f)-\sum_{i=1}^\rk\lambda_i(\mu_n,f_n)+10\rho_0\td\log\Upsilon_1+3\epsilon'}{4\chi}.
\]
Note that when $n$ is large enough, we have
$$|\sum_{i=1}^\td\lambda_i(\mu,f)-\sum_{i=1}^\td\lambda_i(\mu_n,f_n)|\ls\epsilon',$$
thus
\[
\mu_n(U\setminus(\bar B_u\cup\bar B_{cs}))\gs 1-3\rho_0-\frac{\sum_{i=1}^\rk \lambda_i(\mu,f)-\sum_{i=1}^\rk\lambda_i(\mu_n,f_n)+10\rho_0\td\log\Upsilon_1+4\epsilon'}{2\chi}.
\]

Now, we denote \[
\epsilon''=\epsilon''(\rho_0,\epsilon'):=3\rho_0+\frac{10\rho_0\td\log\Upsilon_1+4\epsilon'}{2\chi}.
\] 
which can be arbitrarily  small as long as $\rho_0$ and $\epsilon'$ are small enough.
Note that for any $y\in U\setminus(\bar B_u\cup\bar B_{cs})$, there exists an $x_j$ such that $y\in U_{x_j}$ and \[
\measuredangle(E^{cs}(y),E^{cs}(x_j))\ls\sigma,
\]
\[
\measuredangle(E^{u}(y),E^{u}(x_j))\ls\sigma.
\]
Thus, together with properties B (cs) and B (u), it follows that
\[
\m(Df^{-N}_n|_{E^{cs}(y)})\gs\e^{(-\varphi(y)+2\chi)\cdot N},
\]
\[
\m(Df^{N}_n|_{E^u(y)})\gs\e^{(\varphi(y)+2\chi)\cdot N}.
\]
Thus, for any $y\in f^{-N}_n(U\setminus(\bar B_u\cup\bar B_{cs})),$
\[
\|Df_n^{N}|_{E^{cs}(y)}\|\ls\e^{(\varphi(f_n^{N}y)-2\chi)\cdot N},
\]
and for any $y\in f^{N}_n(U\setminus(\bar B_u\cup\bar B_{cs})),$
\[
\|Df_n^{-N}|_{E^{u}(y)}\|\ls\e^{(-\varphi(f_n^{-N}y)-2\chi)\cdot N}.
\]

Now we start to calculate the size of the Pliss times and weighted Pesin blocks using arguments derived from \cite[Section 2]{SPR2025}. To begin with, let us consider the center-stable direction. For $n$ large enough, we define $$\varphi^{cs}_1(y):=-\log \frac{\|Df_n^{N}|_{E^{cs}(y)}\|}{\e^{N\varphi(f_n^Ny)}},\ \varphi^{cs}_j(y):=\varphi^{cs}_1(y)+\varphi^{cs}_1(f^N_ny)+\cdots+\varphi^{cs}_1(f^{N(j-1)}_ny).$$
Then $$\varphi^{cs}_1(y)\gs -2N\log\Upsilon_1,$$ and inside a set with $\mu_n$ measure $$1-\frac{\sum_{i=1}^\rk \lambda_i(\mu,f)-\sum_{i=1}^\rk\lambda_i(\mu_n,f_n)}{2\chi}-\epsilon'',$$ we have $$\varphi^{cs}_1(y)\gs2\chi N. $$

We define $P_N^{cs}:=\{x:\forall j\gs0,\varphi^{cs}_j(x)\gs\chi N\cdot j\}$. The following lemma, which provides a quantitative estimate of the size of Pliss points, is established in \cite[Lemma 2.18 and Remark 2.19]{SPR2025}.
\begin{lemma}\label{pliss point}
Let $T$ be an invertible measure preserving map on a probability space $(\Omega,\nu)$. Suppose $\varphi\in L^\infty(\nu)$ such that $\varphi\gs B$ and in a set with $\nu$ measure $\kappa$ one has $\varphi(y)\gs A$, denote $\varphi_j:=\sum_{i=0}^{j-1}\varphi\circ T^{i}$, then for every $\beta<A$,
 \[
\nu\{x:\forall j\gs0, \varphi_j(x)\gs\beta j\}\gs\frac{\kappa A+(1-\kappa)B-\beta}{A-\beta}.
\]
\end{lemma}
Thus, by this lemma,
\begin{align*}    
\mu_n(P^{cs}_N)&\gs \frac{ (1-\frac{\sum_{i=1}^\rk \lambda_i(\mu,f)-\sum_{i=1}^\rk\lambda_i(\mu_n,f_n)}{2\chi}-\epsilon'')\cdot 2\chi N}{2\chi N-\chi N}\\[2mm]
 &\quad-\frac{ (\frac{\sum_{i=1}^\rk \lambda_i(\mu,f)-\sum_{i=1}^\rk\lambda_i(\mu_n,f_n)}{2\chi}+\epsilon'')\cdot2N\log\Upsilon_1  +\chi N}{2\chi N-\chi N}\\[2mm]
&=1-\frac{(\frac{\sum_{i=1}^\rk \lambda_i(\mu,f)-\sum_{i=1}^\rk\lambda_i(\mu_n,f_n)}{2\chi}+\epsilon'')\cdot(2\chi+2\log\Upsilon_1) }{\chi}.
\end{align*}
In the same manner, we can define
$$\varphi^{u}_1(y):=-\log \frac{\|Df_n^{-N}|_{E^{u}(y)}\|}{\e^{-N\varphi(f_n^{-N}y)}},\ \varphi^{u}_j(y):=\varphi^{u}_1(y)+\varphi^{u}_1(f^{-N}_ny)+\cdots+\varphi^{u}_1(f^{-N(j-1)}_ny),$$
and \[
P^u_N:=\{x:\forall j\gs0,\varphi^{u}_j(x)\gs\chi N\cdot j\}.
\]
Similarly, it follows that
\begin{align*}    
\mu_n(P^{u}_N)\gs 1-\frac{(\frac{\sum_{i=1}^\rk \lambda_i(\mu,f)-\sum_{i=1}^\rk\lambda_i(\mu_n,f_n)}{2\chi}+\epsilon'')\cdot(2\chi+2\log\Upsilon_1) }{\chi}.
\end{align*}
Thus, we denote $P_N:=P_N^u\cap P^{cs}_N$ and deduce that  \[
\mu_n(P_N)\gs 1-\frac{(\frac{\sum_{i=1}^\rk \lambda_i(\mu,f)-\sum_{i=1}^\rk\lambda_i(\mu_n,f_n)}{2\chi}+\epsilon'')\cdot(4\chi+4\log\Upsilon_1) }{\chi}.
\]

 To construct the weighted Pesin blocks, we now consider \cite[Proposition 2.21]{SPR2025} with the following choices:
\begin{itemize}
    \item \( f := f^N_n \)\smallskip
    \item \( n_0 := 1 \)\smallskip
    \item \( P_{n_0} := P_N \)\smallskip
    \item \( \chi := N\chi \)\smallskip
    \item \( \epsilon := N\epsilon \)
\end{itemize}
Furthermore, to adapt this result to our setting, the cocycles in the proposition are replaced with
\[
    \frac{\|Df_n^{N}|_{E^{cs}(y)}\|}{e^{N\varphi(f_n^Ny)}}
    \quad \text{and} \quad
    \frac{\|Df_n^{-N}|_{E^{u}(y)}\|}{e^{-N\varphi(f_n^{-N}y)}},
\]
whose norms are bounded by $\Upsilon_1^{2N}$. Then, the following proposition can be obtained by a proof similar to that of \cite[Proposition 2.21]{SPR2025}. See Appendix \ref{appendix a4} for details.
\begin{proposition}\label{prop 2.15}
  For $n$ sufficiently large, there exists a compact set $\Lambda^\rk_n\subseteq M$ such that
  \begin{equation}
 \label{size inequality}
\mu_n(M\setminus\Lambda^\rk_n)\ls 4\cdot\frac{\chi+\log\Upsilon_1}{\epsilon}\cdot \frac{(\frac{\sum_{i=1}^\rk \lambda_i
(\mu,f)-\sum_{i=1}^\rk\lambda_i(\mu_n,f_n)}{2\chi}+\epsilon'')\cdot(4\chi+4\log\Upsilon_1) }{\chi}
\end{equation}
and $\Lambda_n^\rk$ satisfies the following properties: for any $x\in\Lambda_n^\rk $, $h\in \mathbb{Z}$ and $m\in \mathbb{N}$,
\begin{align*}
    \|Df^{Nm}_n|_{E^{cs}(f^{Nh}_n(x))}\| &\ls \Upsilon_1^{2N} \cdot \e^{-mN\chi  + |h|N\epsilon}\cdot \e^{\sum_{j=1}^mN\varphi(f^{Nj}_n(f^{Nh}_nx)) },\\[2mm]
    \|Df^{-Nm}_n|_{E^{u}(f^{Nh}_n(x))}\| &\ls \Upsilon_1^{2N} \cdot \e^{-mN\chi  + |h|N\epsilon}\cdot \e^{\sum_{j=1}^{m}-N\varphi(f^{-Nj}_n(f^{Nh}_nx)) }.
\end{align*}
\end{proposition}

Recall that $0\ls\varphi\ls\log\Upsilon_1$, which implies
$$ \e^{\sum_{j=1}^mN\varphi(f^{Nj}_n(f^{Nh}_nx))}\ls \e^{\sum_{j=0}^{m-1}N\varphi(f^{Nj}_n(f^{Nh}_nx))}\cdot \Upsilon_1^N.$$
Together with \Cref{prop 2.15}, it follows that\[
\|Df^{Nm}|_{E^{cs}(f^{Nh}_n(x))}\| \ls \Upsilon_1^{3N} \cdot \e^{-mN\chi  + |h|N\epsilon}\cdot \e^{\sum_{j=0}^{m-1}N\varphi(f^{Nj}_n(f^{Nh}_nx))}.
\]
By comparing the definition \Cref{GPB def 1} and \Cref{GPB def 2}, we have shown that $\Lambda^\rk_n$ is a weighted Pesin Block for $f^N_n$ with parameters $\Lambda^\rk_n(K^N,N\chi,N\epsilon,N\varphi).$

When $n$ is large enough, $|\sum_{i=1}^\rk \lambda_i(\mu,f)-\sum_{i=1}^\rk\lambda_i(\mu_n,f_n)|$ can be arbitrarily small due to previous assumptions. Since $\epsilon''=\epsilon''(\rho_0,\epsilon')$ can be arbitrarily small when $\rho_0$ and $\epsilon'$ are small enough, it follows that

$$\frac{\sum_{i=1}^\rk \lambda_i(\mu,f)-\sum_{i=1}^\rk\lambda_i(\mu_n,f_n)}{2\chi}+\epsilon''$$
can be arbitrarily small when $n$ is large enough and $\rho_0,\ \epsilon''$ are small enough. Thus, the right hand side of \Cref{size inequality} can be arbitrarily small, i.e. for any $\rho>0$ as in this proposition, we can let $\mu_n(\Lambda^\rk_n)>1-\rho$ when $n$ is large enough.

We now construct the weighted Pesin block $\Lambda^\rk$ (which might not be the largest one) by
$$\Lambda^\rk=\limsup_{n\to\infty}\Lambda^\rk_n.$$
For the same reason as in the proof of \Cref{pro 2.1} (a), we can prove that $\Lambda^\rk$ is a weighted Pesin block for $f^N$ with parameters
$\Lambda^\rk(K^N,N\chi,N\epsilon,N\varphi)$. Moreover, together with the previous discussion that $\mu_n(\Lambda^\rk_n)>1-\rho$ and $\mu_n\xrightarrow{*}\mu$, it follows that $\mu(\Lambda^\rk)>1-\rho$. To this end, the proof of the proposition is completed.
\end{proof}
\end{proposition}

\begin{remark}
Within the framework of the classical Pesin theory ($\varphi \equiv \text{constant}$), \Cref{uniform size prop} is generally false (cf. Appendix \ref{appendix a5}).
\end{remark}

By checking the proof of \Cref{uniform size prop}, more specifically by calculating \Cref{size inequality} carefully, we directly obtain the following quantitative corollary of the previous lemma when the sum of Lyapunov exponents does not converge. One can directly check that this quantitative version implies the previous one.

\begin{proposition}\label{quant uniform size prop}
      Let $f_n\in\mathrm{Diff^{1+\alpha}}(M)$ and $\mu_n$ be an ergodic measure of $f_n$. Assume
      \begin{itemize}
    \item  $f_n\xrightarrow{C^1} f\in \mathrm{Diff^{1+\alpha}(M)}$
    \item and $\mu_n\xrightarrow{*}\mu$. 
    \end{itemize}
    Denote $\tilde C_f:={10^8\cdot\alpha^{-1}\cdot(\log\|f\|_{C^1}+1)^3}.$

Assume $\rk\in\mathbb N$ such that $\mathrm{Hyp}^{\mathrm{k}}(\mu)>0$, 
then for any $\rho>0$ there exist $N\in\mathbb N$, $K\gs1$, $\chi>0$ and $\epsilon>0$ satisfying \Cref{var} and a positive function $\varphi\in C^0(M)$ such that there exists $$\Lambda^\rk_n=\Lambda^\rk_n(K^N,N\chi,N\epsilon,N\varphi)$$ as a weighted Pesin block for $f_n^N$ for any $n$ large enough and $$\Lambda^\rk=\Lambda^\rk(K^N,N\chi,N\epsilon,N\varphi)$$ as a weighted Pesin block for $f^N$ satisfying
\begin{itemize}
    \item $\limsup_{n\to\infty}\Lambda_n^\rk =\Lambda^\rk$,
    \item for any  $n\in\mathbb N$ large enough,
    $$\mu_n(\Lambda_n^\rk )>1-\frac{\tilde C_f}{\mathrm{Hyp^{\mathrm{k}}(\mu)^4}}\cdot\limsup_{n\to\infty} (\sum_{i=1}^\rk\lambda_i(\mu,f)-\sum_{i=1}^\rk\lambda_i(\mu_n,f_n))-\rho,$$ 
    \item and $$\mu(\Lambda^\rk)>1-\frac{\tilde C_f}{\mathrm{Hyp^{\mathrm{k}}(\mu)^4}}\cdot\limsup_{n\to\infty} (\sum_{i=1}^\rk\lambda_i(\mu,f)-\sum_{i=1}^\rk\lambda_i(\mu_n,f_n))-\rho.$$ 
\end{itemize}
\end{proposition}

\begin{remark}\label{remark unifrom size}
This quantitative estimate indicates that a large measure is attained on the weighted Pesin blocks when the sums of Lyapunov exponents are sufficiently close to one another. In \cite{SPR2025} (Entropy hyperbolicity in Section 3.2), this property is examined for hyperbolic measures and specifically for ergodic measures with large metric entropy for $C^{\infty}$ surface diffeomorphisms. Notably, for $C^\infty$ surface diffeomorphisms, a large metric entropy of invariant measures implies that the Lyapunov exponents closely approximate those of the measures of maximal entropy \cite{bcs2022}, and this in turn yields the strong positive recurrent (SPR) property as in \cite{SPR2025} by \Cref{quant uniform size prop}. In this respect, \Cref{quant uniform size prop} can be regarded as a quantitative generalization of the SPR estimate to all partial directions, and it applies to the setting of (non-hyperbolic) invariant limit measures.
\end{remark}

An important fact is that uniform weighted Pesin blocks imply uniform size of the weighted Pesin unstable manifolds. In \Cref{sec 2.3}, we show that for a point $x$ belonging to the weighted Pesin block $\Lambda^\rk(K,\chi,\epsilon,\varphi)$ under given parameters, there exists a local weighted Pesin unstable manifold $W^{\mathrm{GPu}}_{\mathrm{loc}}(x)$ whose domain size is given by
$$r(x)=\min\Big\{\frac{r_0}{100}, \Big(\frac{\epsilon}{4C_6C(\epsilon)^3\e^{2\tau\epsilon}\|f\|^{10}_{C^{1+\alpha}}
K^{4\tau}}\Big)^{\frac{1}{\alpha}}\Big\}.$$
Now consider a sequence $f_n\xrightarrow{C^{1}} f$ with $f_n,f\in \mathrm{Diff}^{1+\alpha}(M)$ and 
$$\Upsilon_{1+\alpha}:=\sup_{n\in\mathbb N}\{\|f_n\|_{C^{1+\alpha}}\}<\infty.$$  
Then for any point inside the weighted Pesin block $\Lambda^\rk_n(K^N,N\chi,N\epsilon,N\varphi)$ as given in \Cref{uniform size prop} or \Cref{quant uniform size prop}, there exists a weighted Pesin unstable manifold of uniform domain size, precisely given by
$$r(\cdot)=\min\Big\{\frac{r_0}{100}, \Big(\frac{N\epsilon}{4C_6C(N\epsilon)^3\cdot\e^{N\epsilon}\cdot\Upsilon_{1+\alpha}^{10N}\cdot
K^{4N\tau}}\Big)^{\frac{1}{\alpha}}\Big\}.$$

We now present a proposition on the continuity of the local weighted Pesin unstable manifolds. By \Cref{pro 2.8.2}, each local weighted Pesin unstable manifold is characterized as the graph of a $C^1$ map $\phi$. Furthermore, previous discussions imply that these functions $\phi$ share a domain of uniform size. The continuity here specifically means that $\phi$ are continuous in the $C^1$ topology. This proposition is necessary for the discussions in \Cref{sec 3}.

\begin{proposition}\label{aprro}
   Under the previous assumptions, let $\{x_n\}_{n\in\mathbb{N}^*}$ be a sequence of points such that $x_n\to x$, $x_n\in \Lambda^\rk_n$ and $x\in \Lambda^\rk$, where $\Lambda^\rk_n$ and $\Lambda^\rk$ are given in \Cref{uniform size prop} or \Cref{quant uniform size prop}. Let $\phi_n$ and $\phi$ denote the maps whose graphs are the local weighted Pesin unstable manifolds of $x_n$ and $x$ as in \Cref{defintion of local GPu}. Let $r>0$ be the uniform size of their domain. Then we have 
   $$d_{C^1}(\phi_n|_{B(r)},\phi|_{B(r)})\xrightarrow{n\to\infty} 0$$
   in the sense of the $C^1$ topology of maps. 
\end{proposition}
\begin{proof}
    Note that $\tg_{n,x_n}\xrightarrow{C^{1}} \tg_{x}$. Then, by \Cref{pro 2.8.2}, the unique global fake foliation $\widetilde{\mathcal{F}}_{x_n}^{i}(0)\xrightarrow{C^{1}}\widetilde{\mathcal{F}}_{x}^{i}(0)$ for $i=u, cs$. This is because the invariant manifolds are obtained by the graph transformation, which is based on the contraction mapping principle. Since the fixed point of a contraction depends continuously on the map, the manifolds inherit this continuity. We refer the reader to \cite[Proposition 6.2.8]{Katokbook},
    \cite[Proposition 3.1]{Amie2010} or \cite[Theorem 4.1 F]{HPS} for details.
\end{proof}

Finally, we present some properties of the weighted Pesin unstable manifolds. These properties come from \cite[Section 3.3]{Ledrappier_Strelcyn_1982} and will be used in the next section for the construction of subordinated partitions. For convenience, we only provide brief explanations here and leave the details to the reader. 

Denote $\Lambda_m:=\bigcup_{|i|\ls m}f^i\Lambda^\rk(K,\chi,\epsilon,\varphi)$. By definition, $\Lambda_m\subseteq\Lambda_{m+1}$ and $\Lambda^*=\bigcup_{m\gs0}\Lambda_m$ for any $m\in \mathbb{N}$.
\begin{enumerate}
   \item \textit{For each $x\in \Lambda_{m}$, there exist $0< \epsilon (m) < 1$ and $r_m>0$ such that for any $y\in\Lambda_{m}\cap B(x,\epsilon(m)r)$ and $0<r\ls r_m$, $W^{\mathrm{GPu}}_{\mathrm{loc}}(y)\cap B(x,r)$ is connected.}
   
    This follows from the definition and the local graph property of $W^{\mathrm{GPu}}_{\mathrm{loc}}(x)$ as shown in \Cref{pro 2.8.2}.
    \smallskip
   \item \textit{The map $y\mapsto W^{\mathrm{GPu}}_{\mathrm{loc}}(y)\cap B(x,R_m)$ is continuous from $B(x,\epsilon(m)R_m)\cap\Lambda_{m}$ to the space of subsets of $B(x,R_m)$ endowed with the Hausdorff topology.}
   
   This follows from the compactness of $\Lambda_m$ and the continuity of fake foliations as shown in \Cref{pro 2.8.2}. See also the previous discussion.
   \smallskip
   \item \textit{There exists $a_m>0$ such that for each $y\in\Lambda_m$, $W^{\mathrm{GPu}}_{\mathrm{loc}}(y)$ contains a closed ball of center $y$ and radius $a_m$ in $W^{\mathrm{GPu}}(y)$.}
   
   This follows from \Cref{pro 2.8.2} and the definition of $W^{\mathrm{GPu}}_{\mathrm{loc}}(y)$ from (\ref{defintion of local GPu}), together with previous discussions on $r(x)$.
   \smallskip
   
   \item \textit{There exist $b_m>0$ and $c_m>0$ such that if $z\in W^{\mathrm{GPu}}_{\mathrm{loc}}(y)$, then for any $n\in\mathbb N$,
   $$d_{W^{\mathrm{GPu}}}(f^{-n}y,f^{-n}z)\ls b_{m}\e^{-nc_{m}}d_{W^{\mathrm{GPu}}}(y,z).$$}
   
   This follows from the proof of \Cref{gpu pro}, which shows the exponential contraction inside $W^{\mathrm{GPu}}_{\mathrm{loc}}$ under iterations by $f^{-1}$.
   \smallskip
   \item \textit{For any $0<r\ls r_m$, if two points $z_1$ and $z_2$ belong to 
    $$S(x,r)=\bigcup_{y\in\Lambda_{m}\cap B(x,\epsilon(m)r)} W^{\mathrm{GPu}}_{\mathrm{loc}}(y)\cap B(x,r),$$
    and they are not in the same leaf of $ W^{\mathrm{GPu}}_{\mathrm{loc}}(y)\cap B(x,r)$ for some $y\in \Lambda_{m}\cap B(x,\epsilon(m)r)$, then $d_{W^{\mathrm{GPu}}}( z_{1}, z_{2})\gs 2r$.}
    
    This follows from the definition and the local graph property of $W^{\mathrm{GPu}}_{\mathrm{loc}}(x)$ as shown in \Cref{pro 2.8.2}.
\end{enumerate}

\section{Upper semi-continuity of partial entropy}\label{sec 3}
\subsection{Entropy for Borel measures}\label{sec 3.1}
We first recall the standard notions of entropy for measurable partitions and
measure–preserving transformations, see \cite{Rokhlin_1967} for instance.

We use the standard concept of measurable partition. For a measurable set $B$ we denote the partition of two-elements $\{B, B^c\}$ by boldface $\mathbf{B}$. Moreover, for a Borel measure $\mu$ and a Borel set $B$ with positive measure, let $\mu_B$ denote the conditional probability measure of $\mu$ on $B$ and $\mu|_{B}$ denote the restricted Borel measure of $\mu$ on $B$.
Let $(M,\mu)$ be a probability space and $\tA=\{A_i\}_{i\in I}$ a measurable partition.
We use  $0\log 0:=0$ throughout the convention.

\smallskip
\noindent{\rm (1) Entropy of a partition.}
If $\tA=\{A_i\}_{i\in I}$ is a measurable partition with at most countably many $i$ such that $\mu(A_i)>0$, and
$$-\sum_{i\in I}\mu(A_i)\,\log\mu(A_i)<\infty,$$ 
then its entropy is
\[
 H_\mu(\tA)
 := -\sum_{i\in I} \mu(A_i)\log\mu(A_i)\in[0,+\infty).
\]
Otherwise, we define $H_\mu(\tA)=\infty$.

\smallskip
\noindent{\rm (2) Conditional entropy with respect to a measurable partition.}
Let $\tA$ and $\tB$ be measurable partitions of $M$.  
By Rokhlin's disintegration theorem, there exists
a family of probability measures $\{\mu_{x,\tB}\}_{x\in M}$ as the conditional measures of $\mu$ with respect to $\tB$.
The conditional entropy of $\tA$ with respect to $\tB$ is then defined by
\[
 H_\mu(\tA\mid\tB)
 :=
 \int_{x\in M}H_{\mu_{x,\tB}}(\tA)\,\td\mu(x). 
\]
\smallskip
Next, we proceed to the entropy of dynamical systems.
Let $f\colon M\to M$ be a measurable transformation and $\mu$ an $f$–invariant
probability measure.

\smallskip

\noindent{\rm (3)} For a finite measurable partition $\tA$, we denote
\[
 \tA_0^{n-1}
 :=
 \bigvee_{k=0}^{n-1} f^{-k}\tA,
 \quad n\in\mathbb{N}.
\]
The entropy of $f$ with respect to $\tA$ is
\[
 h_\mu(f,\tA)
 :=
 \lim_{n\to\infty}\frac{1}{n}\,H_\mu(\tA_0^{n-1}),
\]
and the limit exists by the subadditivity of $H_\mu(\tA_0^{n-1})$ with respect to $n$.

\smallskip

\noindent{\rm (4)} The metric entropy of $f$ with respect to $\mu$ is
\[
 h_\mu(f)
 \;:=\;
 \sup_{\tA} h_\mu(f,\tA),
\]
where the supremum is taken over all finite measurable partitions $\tA$ of $M$.

\smallskip

Now we sightly extend the definitions of $H_\mu(\tA)$ and $H_\mu(\tA\mid\tB)$
from probability measures to general Borel measures by a simple normalization.
\begin{definition}[Entropy for Borel Measures]\label{def borel entropy}
    For any Borel measure $\mu$ on $M$ and measurable partitions $\tA$ and $\tB$, we define
    $$H_\mu(\tA)=\mu(M)\cdot H_{\bar{\mu}}(\tA),$$
    $$H_\mu(\tA\mid\tB)=\mu(M)\cdot H_{\bar{\mu}}(\tA\mid\tB),$$
    where $\bar{\mu}$ denotes the normalization of $\mu$, i.e. $\bar{\mu}(\cdot)=\frac{\mu(\cdot)}{\mu(M)}$.
\end{definition}

Then the following lemma shows that the conditional entropy is super-additive with respect to the Borel measures. Recall that for measurable partitions, we write $\tB_n\nearrow\tB$ if $\{\tB_n\}$ is an increasing sequence of measurable partitions that generates $\tB$.

\begin{lemma}
    Let $\mu_1, \mu_2$ be two Borel measures, $\tA$ be a finite partition, and $\tB$ be a measurable partition. Then we have
    $$H_{\mu_1+\mu_2}(\tA\mid\tB)\gs H_{\mu_1}(\tA\mid\tB)+H_{\mu_2}(\tA\mid\tB).$$
\begin{proof}
    Since $\tB$ is a measurable partition, there exists a sequence of increasing finite measurable partitions $\tB_n$ such that $\tB_n\nearrow\tB$. Then we have  $$H_{m}(\tA\mid\tB_n)\searrow H_m(\tA\mid\tB),\quad m\in\{\mu_1,\mu_2,\mu_1+\mu_2\}.$$
    Thus without loss of generality we can assume that $\tB$ is a finite partition. Denoting $\beta_1=\mu_1(M),\ \beta_2=\mu_2(M)$, we have
    \begin{align*}
        &H_{\mu_1+\mu_2}(\tA\mid\tB)=H_{\beta_1\bar\mu_1+\beta_2\bar\mu_2}(\tA\mid\tB)\\[2mm]
        =&\sum_{B_n\in\tB}(\beta_1\bar\mu_1(B_n)+\beta_2\bar\mu_2(B_n))\cdot H_{(\beta_1\bar\mu_1+\beta_2\bar\mu_2)_{B_n}}(\tA)
    \end{align*}
    where $(\beta_1\bar\mu_1+\beta_2\bar\mu_2)_{B_n}$ denotes the conditional measure of $\beta_1\bar\mu_1+\beta_2\bar\mu_2$ on $B_n\in\tB$.
    By definition, we have
    $$(\beta_1\bar\mu_1+\beta_2\bar\mu_2)_{B_n}(\cdot)=\frac{\beta_1\bar\mu_1(\cdot\cap B_n)+\beta_2\bar\mu_2(\cdot\cap B_n)}{\beta_1\bar\mu_1(B_n)+\beta_2\bar\mu_2(B_n)}.$$
    Together with
    $$\bar\mu_{i,B_n}(\cdot)=\frac{\bar\mu_i(\cdot\cap B_n)}{\bar\mu_i(B_n)},\quad i=1,2,$$
    it follows that
    $$(\beta_1\bar\mu_1+\beta_2\bar\mu_2)_{B_n}=\sum_{i=1}^2\frac{\beta_i\bar\mu_i(B_n)}{\beta_1\bar\mu_1(B_n)+\beta_2\bar\mu_2(B_n)}\cdot\bar\mu_{i,B_n}.$$
    Then we complete the proof by
    \begin{align*}
        H_{\mu_1+\mu_2}(\tA\mid\tB)=&\sum_{B_n\in\tB}(\beta_1\bar\mu_1(B_n)+\beta_2\bar\mu_2(B_n))\cdot H_{(\beta_1\bar\mu_1+\beta_2\bar\mu_2)_{B_n}}(\tA)\\[2mm]
        \gs&\sum_{B_n\in\tB}[(\beta_1\bar\mu_1(B_n)+\beta_2\bar\mu_2(B_n))\cdot\sum_{i=1}^2 \frac{\beta_i\bar\mu_i(B_n)}{\beta_1\bar\mu_1(B_n)+\beta_2\bar\mu_2(B_n)}\cdot H_{\bar\mu_{i,B_n}}(\tA)]\\[2mm]
        =&\sum_{B_n\in\tB}\beta_1\bar\mu_1(B_n)H_{\bar\mu_{1,B_n}}(\tA)+\beta_2\bar\mu_2(B_n)H_{\bar\mu_{2,B_n}}(\tA)\\[2mm]
        =&\beta_1 H_{\bar\mu_1}(\tA\mid\tB)+\beta_2 H_{\bar\mu_2}(\tA\mid\tB)\\[2mm]
        =&H_{\mu_1}(\tA\mid\tB)+H_{\mu_2}(\tA\mid\tB).
    \end{align*}
\end{proof}
\end{lemma}

We immediately obtain the following corollary that the conditional entropy is increasing with respect to the Borel measures.

\begin{lemma}\label{lem 3.2}
     Let $\mu_1, \mu_2$ be two Borel measures such that $\mu_1\gs\mu_2$, i.e. there exists a non-negative Borel measure $m$ such that $\mu_1=\mu_2+m$. Let $\tA$ be a finite partition and $\tB$ be a measurable partition, then  
    $$H_{\mu_1}(\tA\mid\tB)\gs H_{\mu_2}(\tA\mid\tB).$$
\end{lemma}

\subsection{Subordinate partitions and partial entropy}\label{sec 3.2}
From \Cref{sec 3.2} to \Cref{sec 3.4}, we complete the proof of \Cref{thm 3.3}. In this section, we establish the framework required to compute and compare partial entropy via subordinate partitions. 

Using the uniform weighted Pesin blocks from Section \ref{sec 2.4}, we construct a sequence of partitions $\xi^u_n$ subordinated to the weighted Pesin unstable manifolds $W^{\mathrm{GPu}}$.
The primary goal in this section is to achieve a sharp decomposition of the $\rk$-dimensional partial entropy into two parts: a principal component, the ``main term,'' which captures the entropy generated within the large-measure lamination $\Xi^u_n$, and an ``error term'' representing the contribution from the complementary region outside the lamination. This structural decomposition provides the necessary groundwork for the estimates in \Cref{sec 3.3} and \Cref{sec 3.4}.

As for \Cref{thm 3.3}, we first prove that $h_{\mu_n}^\rk(f_n)$ is well-defined for sufficiently large $n$. The reason is since the sum of top $\rk$ Lyapunov exponents is continuous by the assumptions, and the sum of top $\rk-1$ Lyapunov exponents is upper semi-continuous by the property of Lyapunov exponents, it follows that 
\[
\liminf_{n\to\infty}\lambda_\rk(\mu_n,f_n)\gs \lambda_\rk(\mu,f).
\]
In the same manner, one has
\[
\limsup_{n\to\infty}\lambda_{\rk+1}(\mu_n,f_n)\ls \lambda_{\rk+1}(\mu,f).
\]
Thus, it follows that
\[
\liminf_{n\to\infty}\mathrm{Hyp}^{\rk}(\mu_n)\gs \mathrm{Hyp}^{\rk}(\mu)>0,
\]
and $h_{\mu_n}^\rk(f_n)$ is well-defined for sufficiently large $n$.

In proving the upper semi-continuity of partial entropy, our objective is to bound $\limsup_{n\to\infty} h^\rk_{\mu_n}(f_n)$. To this end, let $\{n_j\}_{j \in \mathbb{N}}$ be a subsequence realizing this limit superior, that is,
\[
\lim_{j\to\infty} h^\rk_{\mu_{n_j}}(f_{n_j}) = \limsup_{n\to\infty} h^\rk_{\mu_n}(f_n).
\]
By restricting to this subsequence and relabeling, we may assume without loss of generality that the sequence of partial entropies $\{h^\rk_{\mu_n}(f_n)\}_{n\in\mathbb{N}}$ converges. An advantage of this reduction is that any further subsequence extracted from $\{\mu_n\}_{n\in\mathbb{N}}$ will preserve this exact limit. Consequently, throughout the paper, we safely pass to further subsequences whenever necessary to fulfill specific technical conditions, without changing the notation for convenience.

Recall that for any $\mathrm{k}\in\mathbb N$, 
\[
\mathrm{Hyp}^{\mathrm{k}}(\mu):=\essinf_{x\sim\mu}\{\lambda_\rk(x,f)-\max\{\lambda_{\rk+1}(x,f),0\}\}.
\]
For any $f$-invariant measure $\mu$, the $\rk$-dimensional partial entropy $h^\rk_\mu(f)$ is well-defined if $\mathrm{Hyp}^{\mathrm{k}}(\mu)>0$. 

Next, we recall the concept of subordinate partitions and $\rk$-dimensional partial entropy for invariant measures such that $\mathrm{Hyp^{\mathrm{k}}(\mu)>0}$. The subordinate partition, first appearing in \cite{Ledrappier_Strelcyn_1982}, is defined as follows
\begin{definition}[Subordinate Partitions]
Let $\xi$ be a measurable partition of $M$, $f\in\mathrm{Diff}^{1+\alpha}(M)$, $m$ be an $f$-invariant measure, $1\ls \rk\ls \td$ and $\mathrm{Hyp^{\mathrm{k}}}(m)>0$. Then $\xi$ is subordinate to $W^\rk$ with respect to $(f,m)$, where $W^\rk$ denotes the $\rk$-dimensional Pesin unstable manifolds, if for $m-$a.e. $x\in M$ one has
\begin{itemize}
\item[$\mathrm{(a)}$] $\xi(x)\subseteq  W^\rk(x) $ and $\xi(x)$ contains an open neighborhood of $x$ inside $W^\rk(x)$.\smallskip
\item[$\mathrm{(b)}$] $\xi$ is an increasing partition, which means that $f\xi \prec \xi$.\smallskip
\item[$\mathrm{(c)}$] $\bigvee_{n=0}^\infty f^{-n}\xi$ is the partition into points.
\end{itemize}
\end{definition}
\begin{remark}\label{remark: suborGPu is u}
    In the same manner, we can define measurable partitions that are subordinate to $W^{\mathrm{GPu}}$ where $W^{\mathrm{GPu}}$ denotes the $\rk$-dimensional weighted Pesin unstable manifolds defined in \Cref{sec 2.3}. It should be emphasized that these two concepts are equivalent by \Cref{lem 2.12}, since $W^{\mathrm{GPu}}$ coincides with $W^\rk$ for $m-a.e.\ x\in M$. 
\end{remark}

Now following the ideas in \cite{Ledrappier_Strelcyn_1982} and \cite{LEDRAPPIER_YOUNG_A}, we construct subordinate partitions for $(f^N_n,\mu_n)$ and $(f^N,\mu)$ using the weighted Pesin blocks of uniform size constructed in \Cref{uniform size prop} or \Cref{quant uniform size prop}.

Recall the definition of a small boundary condition, which is important in the construction of the subordinate partition in \cite{Ledrappier_Strelcyn_1982}. See also \cite{YANG_expanding_entropy}. In our setting, since we are considering a sequence of measures, we define the small boundary condition for countable measures simultaneously.
\begin{definition}[Small Boundary Condition]\label{small boundary}
    A measurable partition $\tA$ satisfies the small boundary condition, if for any $ \lambda>0$,  $$\sum_{l=1}^{\infty}\mu_n(B(\partial\tA,\e^{-l\lambda}))<\infty$$ for any $n\in\mathbb N$ and $$\sum_{l=1}^{\infty}\mu(B(\partial\tA,\e^{-l\lambda}))<\infty,$$ where $\partial\tA$ means the boundary of partition $\tA$ and 
    $$B(\partial\tA,r):=\{x\in M:d(x,\partial\tA)\ls r\}.$$ A Borel set $A$ is said to satisfy the small boundary condition, if the partition $\{A,M-A\}$ satisfies the small boundary condition.
\end{definition}
In particular, if a measurable partition $\tA$ satisfies the small boundary condition, then $\mu_n(\partial\tA)=\mu(\partial\tA)=0$ for any $n\in\mathbb N$.
By \cite[Proposition 3.2]{Ledrappier_Strelcyn_1982}, for any $x\in M$, $B(x,r)$ satisfies the small boundary condition for $\mathrm{Leb}-a.e.$ $r>0$ small enough. This property helps us to construct the subordinate partitions.


Now we choose $m\in\mathbb N$ large enough and satisfying $N\mid m$, which will go to $\infty$ in the end of the proof. 
For simplicity, we denote $g_n:=f_n^m$ and $g:=f^m$.

Before proceeding, let us formalize the geometric notion for the size of local weighted Pesin unstable manifolds. We say that a local weighted Pesin unstable manifold $W^u_{\mathrm{loc}}(x)$ at $x\in M$ has size $\epsilon_0>0$ if its preimage under the exponential map, $\exp_x^{-1}(W^u_{\mathrm{loc}}(x))$, can be represented as the graph of a $C^1$ map 
\[
\phi_x: \mathrm{Domain}(\phi_x)\subset E^u(x) \to E^u(x)^\perp
\]
satisfying the conditions
\[
E^u(x)(\epsilon_0) \subset \mathrm{Domain}(\phi_x) \quad \text{and} \quad \sup_{v \in E^u(x)(\epsilon_0)} \|D\phi_x(v)\| \ls 10^{-6},
\]
where $E^u(x)(\epsilon_0)$ denotes the ball of radius $\epsilon_0$ centered at the origin in $E^u(x)$. 

Recall that for any given $\rho>0$ small enough, by \Cref{uniform size prop}, it follows that there exist $\epsilon_0>0$ and weighted Pesin blocks $\Lambda^\rk$ and $\Lambda^\rk_n$ such that for $n$ sufficiently large,
\[
\mu_n(\Lambda^\rk_n)\gs 1-\rho,\quad\mu(\Lambda^\rk)\gs 1-\rho,
\]
and they admit local weighted Pesin unstable manifolds $W^{\mathrm{GPu}}_{\mathrm{loc}}(x,f^N)$ and $W^{\mathrm{GPu}}_{\mathrm{loc}}(x,f_n^N)$ of uniform size $\epsilon_0>0$ respectively.

Now we construct a finite measurable partition with the small boundary condition as follows:
choose a finite cover of $M$ by small boundary balls $\{B_1,\cdots,B_\flat\}$ with diameters $$\approx \frac{\epsilon_0}{10^6\cdot \sup_n\{\|g_n\|_{C^1}\}}:=\frac{\epsilon_1}{1000},$$
 and the finite partition is defined as $$\tA:=\{B_1,B_2\setminus B_1,\cdots,B_\flat\setminus(\bigcup_{i=1}^{\flat-1}B_i)\}.$$
The notation $\approx \frac{\epsilon_1}{1000}$ here means $\in [0.99\cdot\frac{\epsilon_1}{1000}, \, 1.01\cdot\frac{\epsilon_1}{1000}]$. 
It is easy to check that $\tA$ is a small boundary partition.

Now we start to construct subordinate partitions using $\tA$, $\Lambda^\rk_n$ and $\Lambda^\rk$ following the methods in \cite{Ledrappier_Strelcyn_1982}. More precisely, for any $x\in\Lambda^\rk_n$ and $x\in\overline{\tA(z)}$, we define\[
W^{u}_{\mathrm{loc},n}(x,\tA(z)):= W^{\mathrm{GPu}}_{\mathrm{loc}}(x,f^N_n)\cap\overline{\tA(z)}.
\]
Similarly for any $x\in\Lambda^\rk$ and $x\in\overline{\tA(z)}$ we define 
\[
W^{u}_{\mathrm{loc},0}(x,\tA(z)):= W^{\mathrm{GPu}}_{\mathrm{loc}}(x,f^N)\cap\overline{\tA(z)}.
\]
All of them are closed by definition. Note that when $x\notin \partial\tA$, there is a unique $\tA(z)$ such that $x\in\overline{\tA(z)}$, which is the element $\tA(x)$.

We further define 
$$\Xi^u_n:=\bigcup_{x,z:\,  x\in\Lambda^\rk_n\cap\overline{\tA(z)}}W^u_{\mathrm{loc},n}(x,\tA(z)),$$
and
$$\Xi^u_0:=\bigcup_{x,z:\,  x\in\Lambda^\rk\cap\overline{\tA(z)} }W^u_{\mathrm{loc},0}(x,\tA(z)).$$
For any $x\in\Xi^u_n$ and $x\notin \partial\tA$, we also denote the unique local unstable leaf containing $x$ as $W^u_{\mathrm{loc},n}(x,\tA(x))$.

\begin{lemma}
    $\Xi^u_n$ are closed for any $\ n\in\mathbb N\cup\{0\}$, and $$    \limsup_{n\to\infty}\Xi^u_n\subseteq \Xi^u_0.$$
\begin{proof}
    Let $x_m\xrightarrow{m\to\infty}x_0$ and $x_m\in\Xi^u_n$. Thus, there exist $y_m\in\Lambda^\rk_n$ and $z_m\in M$ such that $x_m\in W^u_{\mathrm{loc},n}(y_m,{\tA(z_m)})$. Since $\tA$ is a finite partition, by passing to a subsequence we can assume that $\tA(z_m)\equiv\tA(z_0)$ and $y_m\xrightarrow{m\to\infty}y_0$ without loss of generality. Since both $\Lambda^\rk_n$ and $\overline{\tA(z_0)}$ are closed, $y_0\in\Lambda^\rk_n\cap\overline{\tA(z_0)}$.
    
    Since $x_m\in\overline{\tA(z_0)}$, it follows that $x_0\in\overline{\tA(z_0)}$. Moreover, by the continuity of the local weighted Pesin unstable manifolds, $x_m\in W^{\mathrm{GPu}}_{\mathrm{loc}}(y_m,f^N_n)$ implies that $x_0\in W^{\mathrm{GPu}}_{\mathrm{loc}}(y_0,f^N_n)$. Together we have proved that $$x_0\in W^u_{\mathrm{loc},n}(y_0,\tA(z_0))\subseteq \Xi^u_n,$$ thus $\Xi^u_n$ is closed.

    With the fact that $\limsup_{n\to\infty}\Lambda^\rk_n=\Lambda^\rk,$ the second statement follows similarly.   
\end{proof}
\end{lemma}

Note that the previous assumptions on $\tA$ are sufficient to ensure the following two conditions:
\begin{itemize}\label{lem: property of au}
    \item[1.] For any $x\in\Lambda^\rk_n$ there exists a differential map $\sigma_x:[0,1]^\rk\to W^{\mathrm{GPu}}(x)$ such that $\|\td\sigma_x\|<\frac{\epsilon_1}{100}$ and $W^u_{\mathrm{loc},n}(x,\tA(z))\subseteq \mathrm{Im}(\sigma_x)$, where $\mathrm{Im}(\sigma_x)$ denotes the image of $\sigma_x$.
    \smallskip
    \item[2.] For any $a,b\in \tA(z)\cap\Xi^u_n$ in different $W^u_{\mathrm{loc},n}$, which means that $a\in W^u_{\mathrm{loc},n}(p,\tA(z))$ and $b\in W^u_{\mathrm{loc},n}(q,\tA(z))$ but $W^u_{\mathrm{loc},n}(p,\tA(z))\neq W^u_{\mathrm{loc},n}(q,\tA(z))$, one has $$d^u(a,b)>\frac{\epsilon_0}{10}=100\cdot\sup_n\{\|g_n\|_{C^1}\}\cdot\epsilon_1,$$ where $d^u$ denotes the submanifold distance induced by the $\rk$-dimensional weighted Pesin unstable manifolds.
\end{itemize}

\begin{figure}[!htbp]
    \centering
    \includegraphics[width=0.7\linewidth]{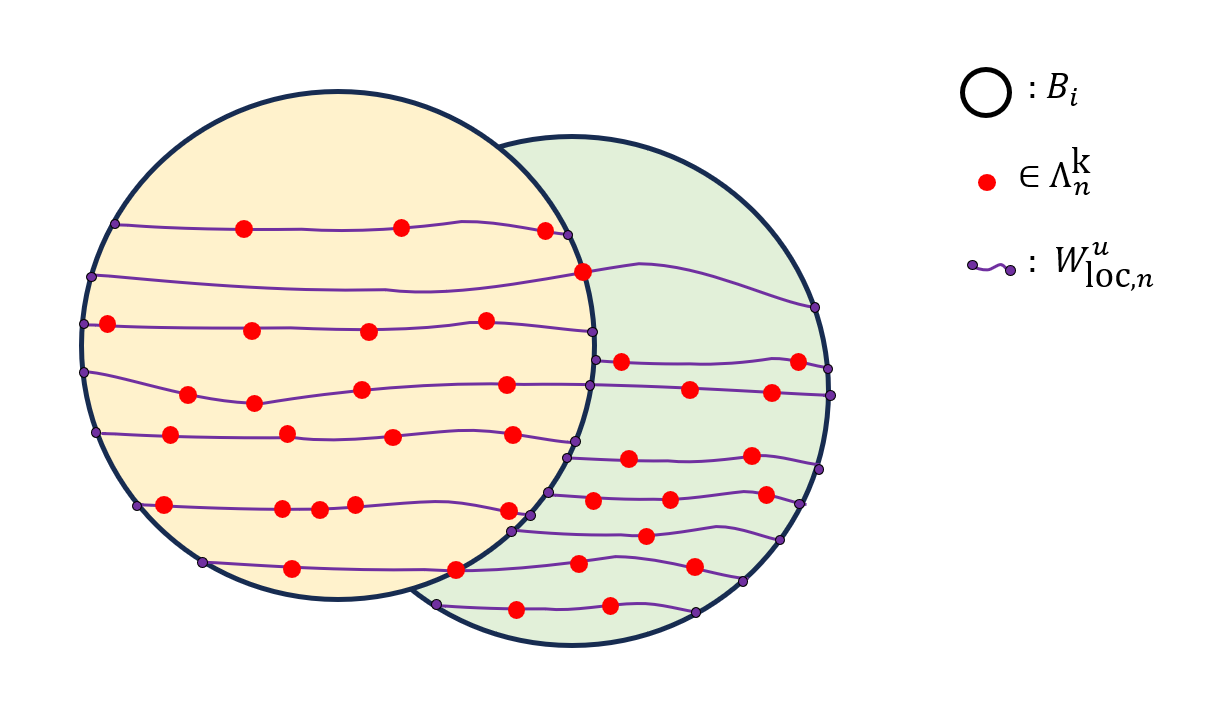}
    \caption{Illustration of the geometric structure}
    \label{fig:placeholder}
\end{figure}

Now we use the local unstable lamination given by weighted Pesin blocks to refine the small boundary partition $\tA$. We define $\tA^u_n$ by:
\begin{enumerate}
    \item If there exists $z\in\Lambda^\rk_n$ such that $x\in W^u_{\mathrm{loc},n}(z,\tA(x))$, then $$\tA^u_n(x):=W^u_{\mathrm{loc},n}(z,\tA(x))\cap\tA(x).$$
    \item If $x\notin\Xi^u_n$, then $\tA^u_n(x):=\tA(x)\setminus\Xi^u_n$.
    \item For other $x$, $\tA^u_n(x):=\{x\}$.
\end{enumerate}
\begin{remark}
    The third case occurs if and only if there exist $y\in\Lambda^\rk_n$ and $z\notin\tA(x)$ such that $x\in W^u_{\mathrm{loc},n}(y,\tA(z))$ and $x\notin W^u_{\mathrm{loc},n}(z',\tA(x))$ for any $z'\in\Lambda^\rk_n$. Consequently, the third case could happen only when $x\in\partial\tA$. Thus, when dealing with the measure-theoretic properties of $\tA^u_n$ in this paper, the third case could be ignored. We also have $\mathbf{\Xi^u_n}\prec\tA^u_n$ by definition (recall that $\mathbf{\Xi^u_n}$ denotes the two-element partition $\{\Xi^u_n,(\Xi^u_n)^c\}$).
\end{remark}

We further denote that
\[
\Xi^u_{n,\infty}:=\bigcup_{i=-\infty}^\infty g_n^i\Xi^u_n,\quad \Xi^u_{0,\infty}:=\bigcup_{i=-\infty}^\infty g^i\Xi^u_0.
\] 
At the end of \Cref{sec 2.3} we show that the conditions required in \cite[Section 3]{Ledrappier_Strelcyn_1982} hold in our setting. According to the classical methods in \cite{Ledrappier_Strelcyn_1982} for constructing the subordinate partitions, the partition $$\xi^u_n:=\bigvee_{i=0}^\infty g^i_n\tA^u_n$$ is a measurable partition subordinate to $W^{\mathrm{GPu}}$ with respect to $(g_n,\mu_n|_{\Xi^u_{n,\infty}})$. By \Cref{remark: suborGPu is u}, it is also subordinate to $W^\rk$.
In the same manner, we can define $\tA^u_0$ and the subordinate partition $\xi^u_0$ with respect to $(g,\mu|_{\Xi^u_{0,\infty}})$ similarly. We remark that both $\mu_n\big|_{\Xi^u_{n,\infty}}$ and $\mu\big|_{\Xi^u_{0,\infty}}$ are invariant Borel measures with respect to $g_n$ and $g$, each forming a component of $\mu_n$ and $\mu$ respectively.

We remark that, by construction, $\Xi^u_{n,\infty}$ is defined as the union of both forward and backward iterates of $\Xi^u_n$, whereas the partition $\xi^u_n$ is generated using only forward iterates of $\tA^u_n$. Since the difference between the forward-only and the forward–backward unions is a set of $\mu_n$-measure zero, it follows that $\mathbf{\Xi^u_{n,\infty}}\prec\xi^u_n$ with respect to $\mu_n$. Similarly, $\mathbf{\Xi^u_{0,\infty}}\prec\xi^u_0$ with respect to $\mu$.

Now we can use the subordinate partitions to calculate the partial entropy $h^\rk_{\mu_n}(g_n)$ and decompose it into the main term and the error term.
More precisely, by the definition of $\xi^u_n$ and the partial entropy, we have
\begin{align*}
h^\rk_{\mu_n}(g_n)&\ls  H_{\mu_n|_{\Xi^u_{n,\infty}}}(g_n^{-1}\xi^u_n\mid \xi^u_n)+\mu_n((\Xi^u_{n,\infty})^c)\cdot m\td\log\Upsilon_1,\\[2mm]
&\ls  H_{\mu_n|_{\Xi^u_{n,\infty}}}(g_n^{-1}\xi^u_n\mid \xi^u_n)+\mu_n((\Xi^u_{n})^c)\cdot m\td\log\Upsilon_1.
\end{align*}
Moreover, by the property of the conditional entropy,
\begin{equation}\label{equ:firstandsecond}
\begin{split}
H_{\mu_n|_{\Xi^u_{n,\infty}}}(g_n^{-1}\xi^u_n\mid \xi^u_n)&=H_{\mu_n|_{\Xi^u_{n,\infty}}}(g_n^{-1}\tA^u_n\mid \xi^u_n),\\[2mm]
&= H_{\mu_n|_{\Xi^u_{n,\infty}}}(g_n^{-1}\tA\mid\xi^u_n) + H_{\mu_n|_{\Xi^u_{n,\infty}}}(g_n^{-1}\tA^u_n\mid g_n^{-1}\tA\vee \xi^u_n).
\end{split}
\end{equation}

For the first term of \Cref{equ:firstandsecond}, \Cref{def borel entropy} and the fact $\mathbf{\Xi^u_n}\prec\xi^u_n$ yield the decomposition $$H_{\mu_n|_{\Xi^u_{n,\infty}}}(g_n^{-1}\tA\mid\xi^u_n) = H_{\mu_n|_{\Xi^u_n}}(g_n^{-1}\tA\mid\xi^u_n) + H_{\mu_n|_{({\Xi^u_{n,\infty}}\setminus\Xi^u_n)}}(g_n^{-1}\tA\mid\xi^u_n).$$ Furthermore, the latter term can be bounded by
\begin{align*}
H_{\mu_n|_{({\Xi^u_{n,\infty}}\setminus\Xi^u_n)}}(g_n^{-1}\tA\mid\xi^u_n)&\ls \mu_n({\Xi^u_{n,\infty}}\setminus\Xi^u_n)\cdot\log\#\tA,\\[2mm]
&\ls \mu_n((\Xi^u_n)^c)\cdot\log\#\tA.
\end{align*}

To estimate the second term of \Cref{equ:firstandsecond}, we refine the conditioning partition by $g_n^{-1}\mathbf{\Xi^u_n}$. Then the property of conditional entropy gives
\begin{align*}
H_{\mu_n|_{\Xi^u_{n,\infty}}}(g_n^{-1}\tA^u_n\mid g_n^{-1}\tA\vee \xi^u_n) &\ls H_{\mu_n|_{\Xi^u_{n,\infty}}}(g_n^{-1}\mathbf{\Xi^u_n}) + H_{\mu_n|_{\Xi^u_{n,\infty}}}(g_n^{-1}\tA^u_n\mid g_n^{-1}\tA\vee\xi^u_n\vee g_n^{-1}\mathbf{\Xi^u_n}),\\[2mm]
&\ls \log 2 + H_{\mu_n|_{\Xi^u_{n,\infty}}}(g_n^{-1}\tA^u_n\mid g_n^{-1}\tA\vee\xi^u_n\vee g_n^{-1}\mathbf{\Xi^u_n}).
\end{align*}

Combining all these estimates, we arrive at the inequality
\begin{equation}\label{entropy split}
\begin{split}
    h^\rk_{\mu_n}(g_n)\ls&\,\,H_{\mu_n|_{\Xi^u_n}}(g_n^{-1}\tA\mid\xi^u_n)
     +\mu_n\left((\Xi^u_n)^c\right)\cdot(\log\#\tA+m\td\log\Upsilon_1)\\[2mm] 
    & + \log 2 + H_{\mu_n|_{\Xi^u_{n,\infty}}}(g_n^{-1}\tA^u_n\mid g_n^{-1}\tA\vee\xi^u_n\vee g_n^{-1}\mathbf{\Xi^u_n}).
\end{split}
\end{equation}

\begin{itemize}
    \item \textbf{The main term:}  The first term in (\ref{entropy split}) \[
H_{\mu_n|_{\Xi^u_n}}(g_n^{-1}\tA\mid\xi^u_n)
\]
is called the main term.
The main term represents the entropy generated by one-step iteration of the finite partition $\tA$ on the unstable lamination $\Xi^u_n$ with large measure, which admits a uniformly contracting property along the backward iteration. 
It will be estimated to be ``upper semi-continuous'' in \Cref{sec 3.3}. \smallskip
\item \textbf{The error term:} The other terms in (\ref{entropy split}) 
\begin{align*}
\mu_n\left((\Xi^u_n)^c\right)\cdot(\log\#\tA+m\td\log\Upsilon_1) + \log 2\\[2mm] + H_{\mu_n|_{\Xi^u_{n,\infty}}}(g_n^{-1}\tA^u_n\mid g_n^{-1}\tA\vee\xi^u_n\vee g_n^{-1}\mathbf{\Xi^u_n})
\end{align*}
is called the error term. It consists of two parts. The first part of the entropy comes from the number of elements of the finite partition $\tA$, which can be controlled by the dimension of $M$. The other part originates from the conditional partition $\tA^u_n$ restricted to each $[\tA\vee g_n\xi^u_n\vee\mathbf{\Xi^u_n}](x)$, which is, more intuitively, from the number of the winding components of $[g_n\xi^u_n](x)$ within the set $\tA(x)$. To bound this part, we need to bound the unstable diameter of $[g_n\xi^u_n](x)$, which leads to the discussion of recurrence time. 
It will be estimated to be uniformly bounded by the loss of continuity of the Lyapunov exponents in \Cref{sec 3.4}.
\end{itemize}

\subsection{The upper semi-continuity estimate of the main term}\label{sec 3.3}
We now proceed to establish the upper semi-continuity of the main term by analyzing the entropy generated along the uniform $\Xi^u_n$ lamination. The core of this analysis lies in the $C^1$-continuity of the local weighted Pesin unstable manifolds (Proposition \ref{aprro}) and their uniform exponential contraction under backward iteration by $g_n$ or $g$. This uniform contraction, inherited from the uniform parameters of the weighted Pesin blocks, ensures that the conditional geometric structures within the lamination evolve stably when the partitions are iterated (\Cref{lem 3.4}). By exploiting this geometric structure, we demonstrate that the complexity within the uniform blocks is well-behaved under the weak$^*$ topology.

First, note that we can choose $m$ large enough such that for any $p\in\mathbb N$ and $x\in\Lambda^\rk_n$, $g_n^{-p}=f_n^{-mp}$ uniformly contracts the unstable distance $d^u$ for any $W^u_{\mathrm{loc},n}(x,\tA(z))$.
Once again, by passing to a subsequence, we may assume that  $\lim_{n\to\infty}(\mu_n|_{\Xi^u_n})$ exists without loss of generality. We denote $\nu:=\lim_{n\to\infty}(\mu_n|_{\Xi^u_n})$.
Since $\limsup_{n\to\infty} \Xi^u_n\subseteq\Xi^u_0$, it follows that
$$\nu(M\setminus\Xi^u_0)=0\quad
\text{and}\quad \nu\ls \mu|_{\Xi^u_0}.$$

Recall that $\xi^u_0$ defined in the previous section is a subordinate partition with respect to $(g,\mu|_{\Xi^u_{0,\infty}})$.
Thus, together with \Cref{lem 3.2}, one has 
\[
h^\rk_{\mu}(g)\gs  H_{\mu|_{\Xi^u_{0,\infty}}}(g^{-1} \xi^u_0\mid \xi^u_0)\gs H_{\mu|_{\Xi^u_{0,\infty}}}(g^{-1}\tA\mid \xi^u_0)\gs H_{\mu|_{\Xi^u_0}}(g^{-1}\tA\mid\xi^u_0).
\]

\begin{remark}
When considering the measure-theoretic properties of a measurable partition with respect to a measure, we may—by definition—disregard null sets of that measure. Accordingly, when analyzing the measure-theoretic properties of measurable partitions with respect to $\mu_n$, we disregard the set $\bigcup_{i=-\infty}^{\infty} g_n^i(\partial\tA)$ for any $n \in \mathbb{N}$. We apply the same procedure to the set $\bigcup_{i=-\infty}^{\infty} g^i(\partial\tA)$ when working with respect to $\mu$.

The notation $\csubset, \csupset, \ceq,$ and $\cneq$ used in the following paper are defined within this framework. Specifically, these symbols imply that we disregard the iterations of $\partial\tA$ when evaluating inclusions, focusing exclusively on points $x$ whose orbits never intersect the boundary of $\tA$ under iteration by $g_n$ or $g$, respectively. However, one should treat the boundary with care when considering topological properties.
\end{remark}

Now we prove the following lemma, which replaces the forward iteration of $\tA^u_n$ with that of $\tA$ to simplify the subsequent analysis. Thus, when considering the coarsening process of $\xi^u_n$ in the main term $H_{\mu_n|_{\Xi^u_n}}(g_n^{-1}\tA\mid\xi^u_n)$, \Cref{lem 3.4} enables us to coarsen $\tA^u_n$ instead of $\xi^u_n=\bigvee_{i=0}^\infty g_n^i\tA^u_n$ (whose geometric structure is complicated).  
\begin{lemma}\label{lem 3.4}
    For $\mu_n-a.e.\ x\in\Xi^u_n$ and any $w\in\mathbb N$, $$[\bigvee_{i=0}^wg_n^i\tA^u_n](x)\ceq [\bigvee_{i=0}^wg^i_n\tA](x)\cap\tA^u_n(x).$$
\begin{proof}
    Since $[\bigvee_{i=0}^w g_n^i\tA^u_n](x)\subseteq [\bigvee_{i=0}^w g^i_n\tA](x)\cap\tA^u_n(x)$ follows directly from the definitions, we only need to prove that $$[\bigvee_{i=0}^w g_n^i\tA^u_n](x)\csupset [\bigvee_{i=0}^w g^i_n\tA](x)\cap\tA^u_n(x).$$
    We are free to assume that the orbit of $x$ never intersects $\partial\tA$ under iteration by $g_n$. Consequently, we have $x\in W^u_{\mathrm{loc},n}(x,\tA(x))$ and for any $j\gs0$,
    \begin{equation}\label{eq in lem3.4}
    \mathrm{diam}^u(g^{-j}_nW^u_{\mathrm{loc},n}(x,\tA(x)))\ls\epsilon_1<\frac{\epsilon_0}{10}.
    \end{equation}
    Note that if $y\in[\bigvee_{i=0}^w g^i_n\tA](x)\cap\tA^u_n(x)$ then for any $0\ls j\ls w$, 
    \[
    g_n^{-j}(y)\in \tA(g^{-j}_nx)\cap g^{-j}_nW^u_{\mathrm{loc},n}(x,\tA(x)).
    \]
    The proof proceeds in two cases:
    \begin{enumerate}
    \smallskip
        \item[(a):] If $$g^{-j}_nW^u_{\mathrm{loc},n}(x,\tA(x))\cap\tA(g^{-j}_nx)\cap\Xi^u_n\cneq \emptyset,$$ then by \Cref{eq in lem3.4} there exists $z\in\Lambda^\rk_n\cap\overline{\tA(g_n^{-j}x)}$ such that
        \[
        g^{-j}_nW^u_{\mathrm{loc},n}(x,\tA(x))\cap\tA(g^{-j}_nx)\subseteq W^{\mathrm{GPu}}_{\mathrm{loc}}(z,f_n^N)\cap\tA(g^{-j}_nx).
        \]
        Thus by definition, we have
        $$g^{-j}_nW^u_{\mathrm{loc},n}(x,\tA(x))\cap\tA(g^{-j}_nx)\subseteq\Xi^u_n$$
        and
        $$g^{-j}_nW^u_{\mathrm{loc},n}(x,\tA(x))\cap\tA(g^{-j}_nx)\subseteq \tA^u_n(g^{-j}_nx).$$

        \smallskip
        \item[(b):] If $$g^{-j}_nW^u_{\mathrm{loc},n}(x,\tA(x))\cap\tA(g^{-j}_nx)\cap\Xi^u_n\ceq \emptyset,$$
        which means that
        $$g^{-j}_nW^u_{\mathrm{loc},n}(x,\tA(x))\cap\tA(g^{-j}_nx)\csubset(\Xi^u_n)^c,$$
        then        
        $$g^{-j}_nW^u_{\mathrm{loc},n}(x,\tA(x))\cap\tA(g^{-j}_nx)\csubset(\Xi^u_n)^c\cap\tA(g^{-j}_nx).$$
        Since $x$ never intersects the boundary, we have $g_n^{-j}x\notin\Xi^u_n$  thus $$\tA^u_n(g_n^{-j}x)=(\Xi^u_n)^c\cap\tA(g^{-j}_nx)$$ and\[
        g^{-j}_nW^u_{\mathrm{loc},n}(x,\tA(x))\cap\tA(g^{-j}_nx)\csubset\tA^u_n(g^{-j}_nx).
        \]
    \end{enumerate}
    Combining two cases, it follows that for $0\ls j\ls w$, \[
        g^{-j}_nW^u_{\mathrm{loc},n}(x,\tA(x))\cap\tA(g^{-j}_nx)\csubset\tA^u_n(g^{-j}_nx).
        \]
        Thus, the proof is complete.
\end{proof}
\end{lemma}

We now define a sequence of finite coarsenings of the measurable partitions $\tA^u_n$ and $\tA^u_0$. Unlike the strategy in \cite{YANG_expanding_entropy}, our setting only admits a weighted Pesin unstable lamination instead of a foliation. Consequently, a coarsening with a zero-measure boundary is not expected to be obtained. 

Therefore, we consider the boundary within the compact subspaces $\Xi^u_n$ and $\Xi^u_0$ instead of $M$, and use the subspace topology. Since the measures we considered in this section are supported in these compact subspaces, it is enough to estimate the entropy. This is formalized in the following lemma:

\begin{lemma}\label{lem 3.5}
There exist measurable partitions $\tA_{n,w}^u\succ\mathbf{\Xi^u_n} $ and $\tA_{0,w}^u\succ \mathbf{\Xi^u_0}$ for any $n,w\in\mathbb N$ such that $\tA^u_{n,w}\nearrow \tA^u_n$  and $\tA^u_{0,w}\nearrow \tA^u_{0}$ when $w\to\infty$, and for any $n\in\mathbb N\cup\{0\}$, inside the compact subspaces $\Xi^u_n$ we have $\mu_n(\partial_{\Xi^u_n}(\tA^u_{n,w}|_{\Xi^u_n}))=0$. Moreover, $\tA^u_{n,w}$ and $\tA^u_{0,w}$ are finite partitions outside $\partial\tA$.
\begin{proof}
For each $1\ls i\ls \flat$, we can choose a transversal of dimensions $\td-\rk$ such that for any $x\in B_i\cap\Lambda_n^\rk $, $W^\mathrm{GPu}_{\mathrm{loc}}(x,f^N_n)$ intersects $T_i$ exactly at one point. Denote $T_{\tA(x)}=T_i$ when $\tA(x)\subseteq B_i$, and $$\tA_{n}^1(x):=\{y\in\tA(x): \exists\ z\in\Lambda^\rk_n \text{ s.t. } y\in W^u_{\mathrm{loc},n}(z,\tA(x))\}.$$  
Thus for $n\in\mathbb N\cup\{0\}$, inside $\tA(x)$ we can define the unstable holonomy map $\pi^u_{n,\tA(x)}$ as
\[
\begin{aligned}
\pi^u_{n,\tA(x)} \colon \tA^1_n(x)\cap\Xi^u_n &\longrightarrow T_{\tA(x)} \\[1mm]
z\quad &\longmapsto W^{\mathrm{GPu}}_{\mathrm{loc}}(z,f^N_n)\cap T_{\tA(x)}.
\end{aligned}
\]
There is a Euclidean structure on each $T_i$, which can be seen as $T_i\subseteq \mathbb R^{\td-\rk}$.

Now we choose a sequence of dyadic square grids $\{\tG_{w,\tA(x)}\}_{w=1}^\infty$  on each $T_{\tA(x)}\subseteq\mathbb R^{\td-\rk}$, each of which is a grid of squares of length $2^{-w}$, that is, $\tG_{w,\tA(x)}$ is a partition of $\mathbb R^{\td-\rk}$ with elements:
$$\prod_{i=1}^{\td-\rk}[\,a_i+\frac{k_i}{2^w},a_i+\frac{k_i+1}{2^w\,}),\ k_i\in\mathbb Z.$$
By choosing the translation $a=(a_1,\cdots,a_{\td-\rk})$ in the Euclidean space, each $\tG_{w,\tA(x)}$ can be chosen such that: 
$$\forall n\in\mathbb N\cup\{0\},\ \mu_n\left((\pi^u_{n,\tA(x)})^{-1}(\partial\tG_{w,\tA(x)})\right)=0.$$
This can be done since the choice of $\tG_{w,\tA(x)}$ is uncountable. Denote $$\tQ_{w,\tA(x)}:=\bigvee_{i=1}^w \tG_{i,\tA(x)}.$$

\begin{figure}[!htbp]
    \centering
    \includegraphics[width=1\linewidth]{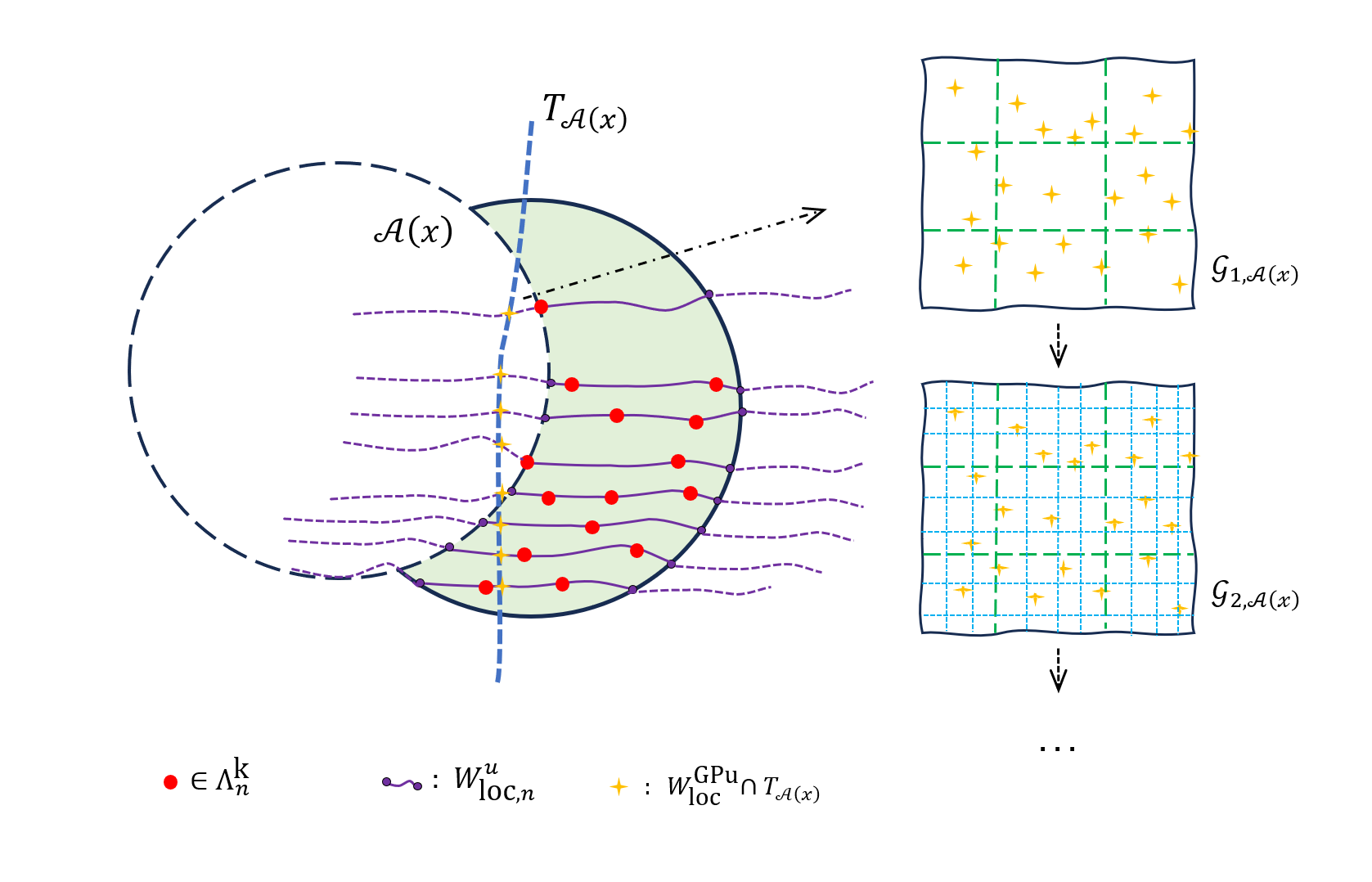}
    \caption{Illustration of the coarsening process ($\dim T_{\tA(x)}=2$)}
    \label{fig:placeholder2}
\end{figure}

Then we use $\tQ_{w,\tA(x)}$ to define $\tA^u_{n,w}$. For each $n\in\mathbb N$, similar to the definition of $\tA^u_n$, we consider three cases:
\begin{itemize}
    \item[Case 1.] If there exists $z\in\Lambda^\rk_n$ such that $x\in W^u_{\mathrm{loc},n}(z,\tA(x))$, then $$p:=\pi^u_{n,\tA(x)}(x)\in T_{\tA(x)}$$ is well-defined and $\tA^u_{n,w}(x)$ is defined as
    $$\tA^u_{n,w}(x):=(\pi^u_{n,\tA(x)})^{-1} (\tQ_{w,\tA(x)}(p))$$
    \item[Case 2.] If $x\notin\Xi^u_n$, then $\tA^u_{n,w}(x)=\tA^u_{n}(x)=\tA(x)\setminus\Xi^u_n.$ \smallskip
    \item[Case 3.] For other $x$, $\tA^u_{n,w}(x)=\tA^u_n(x)=\{x\}$, which could happen only if $x\in\partial \tA$.
\end{itemize}

By definition, $\tA^u_{n,w}$ is a finite partition outside $\partial\tA$ and $\tA_{n,w}^u\succ\mathbf{\Xi^u_n}$. Since the diameter of the elements in $\tQ_{w,\tA(x)}\to0$ when $w\to\infty$,  one has $\tA^u_{n,w}\nearrow \tA^u_n$. Now,  consider the boundary of $\tA^u_{n,w}|_{\Xi^u_n}$ inside the compact subspace $\Xi^u_n$. By previous definitions and the continuity of local weighted Pesin unstable manifolds \Cref{aprro}, it follows that
\[
\partial_{\Xi^u_n}(\tA^u_{n,w}|_{\Xi^u_n})\subseteq\partial\tA\cup\left(\bigcup_{\tA(x)\in\tA}(\pi^u_{n,\tA(x)})^{-1}(\partial\tG_{w,\tA(x)})\right).
\]
Thus $\mu_n(\partial_{\Xi^u_n}(\tA^u_{n,w}|_{\Xi^u_n}))=0$. Similarly, we can define $\tA^u_{0,w}$ by the same $\tQ_{w,\tA(x)}$ and prove the same statements. Hence, the proof is complete.
\end{proof}
\end{lemma}

\begin{remark}
For any $n\in\mathbb N\cup\{0\}$ and $w\in\mathbb N$, if the element $\tA^u_{n,w}(x)$ falls into Case~1 in the definition of $\tA^u_{n,w}$, then it is associated with an element of $\tQ_{w,\tA(x)}$, namely $\tQ_{w,\tA(x)}(p)\subseteq T_{\tA(x)}$. For any distinct $n_1,n_2\in\mathbb N\cup\{0\}$, if $\tA^u_{n_1,w}(y)$ and $\tA^u_{n_2,w}(z)$ both fall into Case~1 and are associated with the same element of $\tQ_{w,\tA(x)}$, these elements are called related.
\end{remark}

\begin{remark}
Since the atoms of $\tA^u_{n,w}$ in Case~2 are outside $\Xi^u_n$, and the atoms in Case~3 are inside $\partial\tA$, when we consider the measure-theoretic properties of $\tA^u_{n,w}$ with respect to $\mu_n|_{\Xi^u_n}$, the atoms in Case~2 and Case~3 can be disregarded.
\end{remark}

By \Cref{lem 3.4} and \Cref{lem 3.5},
\begin{align*}
H_{\mu_n|_{\Xi^u_n}}(g_n^{-1}\tA\mid\xi^u_n)&=\inf_{w\in\mathbb N} H_{\mu_n|_{\Xi^u_n}}(g_n^{-1}\tA\mid\bigvee_{i=0}^wg^i_n\tA\vee\tA^u_n)\\[2mm]
&=\inf_{w\in\mathbb N} H_{\mu_n|_{\Xi^u_n}}(g_n^{-1}\tA\mid\bigvee_{i=0}^wg^i_n\tA\vee\tA^u_{n,w}).
\end{align*}
Denote $\nu:=\lim_{n\to\infty}(\mu_n|_{\Xi^u_n})$, recall that it has been proved that
\[
\nu((\Xi^u_0)^c)=0\quad\text{and}\quad\nu\ls\mu|_{\Xi^u_0}.
\]

By the construction in \Cref{lem 3.5} and the continuity of the local weighted Pesin unstable manifolds \Cref{aprro}, noting that both $\tA$ and $\tA^u_{0,w}$ have  a zero-measure boundary inside the compact subspace $\Xi^u_0$, one can check that
\[
\lim_{n\to\infty}H_{\mu_n|_{\Xi^u_n}}(g_n^{-1}\tA\mid\bigvee_{i=0}^wg^i_n\tA\vee\tA^u_{n,w})=H_\nu(g^{-1}\tA\mid\bigvee_{i=0}^wg^i\tA\vee\tA^u_{0,w}).
\]
The reason is that, in view of the previous discussions and the continuity properties with respect to convergent measures, for each fixed $w\in\mathbb N$, the $\mu_n|_{\Xi^u_n}$ measure of each element of the finite partition $\bigvee_{i=-1}^wg^i_n\tA\vee\tA^u_{n,w}$ converges to the $\nu$ measure of the related element of the finite partition $\bigvee_{i=-1}^wg^i\tA\vee\tA^u_{0,w}$. Taking $\inf_{w\in\mathbb N}$ on both sides, since the infimum of a sequence of continuous functions is upper semi-continuous, we conclude that
\begin{align*}
\limsup_{n\to\infty}H_{\mu_n|_{\Xi^u_n}}(g^{-1}_n\tA\mid\xi^u_n)&\ls H_{\nu}(g^{-1}\tA\mid\xi^u_0)\\[2mm]
&\ls H_{\mu|_{\Xi^u_0}}(g^{-1}\tA\mid\xi^u_0)\\[2mm]
&\ls h^\rk_\mu(g).
\end{align*}
Thus, we complete the upper semi-continuity estimate of the main term.

\subsection{Uniform estimate of the error term}\label{sec 3.4}
This section is devoted to the uniform control of the error term, which quantifies the entropy contribution from the region external to the weighted Pesin blocks. Since the main term accounts for the complexity within the uniform lamination, the remaining analysis focuses on the set where uniform geometric control is absent.
By deriving estimates based on recurrence times (\Cref{lem 3.11}) and the measure deficiency of the weighted Pesin blocks (\Cref{quant uniform size prop}), we prove that the error term is effectively bounded by the defect in the continuity of the Lyapunov exponent sums. 

First we estimate the term $$\mu_n((\Xi^u_n)^c)\cdot(\log\#\tA+m\td\log\Upsilon_1).$$
Recall that $\tA$ is induced by a covering of open balls with diameters  $$\approx\frac{\epsilon_0}{10^6\cdot\sup_n\{\|g_n\|_{C^1}\}}$$ and $M$ is a $\td$-dimensional compact manifold, one can choose $\tA$ such that there exists a uniform constant $C_0>0$ which does not depend on $m$ and satisfies 
$$\#\tA\ls C_0\cdot\sup_n\{\|g_n\|_{C^1}\}^\td\ls C_0\cdot\Upsilon_1^{m\td}.$$
Thus
$$\mu_n((\Xi^u_n)^c)\cdot(\log\#\tA+m\td\log\Upsilon_1).\ls\mu_n((\Lambda^\rk_n)^c)\cdot(2m\td\cdot\log\Upsilon_1+\log C_0).$$

Then we estimate the term \[
H_{\mu_n|_{\Xi^u_{n,\infty}}}(g^{-1}_n\tA^u_n\mid g^{-1}_n\tA\vee\xi^u_n\vee g^{-1}_n\mathbf {\Xi^u_n}).
\]

Before going into details, we roughly explain the ideas of the following proof. As discussed in the end of \Cref{sec 3.2}, we need to control the number of the winding components of $[g_n\xi^u_n](x)$ within the set $\tA(x)$. An important fact is that the local weighted Pesin unstable manifolds have a uniformly large size on the weighted Pesin blocks. Thus, to control the number of winding components, we only need to estimate the unstable diameter of each $[g_n\xi^u_n](x)$. This quantity is related to the ``hyperbolicity'' of the point $g_n^{-1}x$, more specifically, the recurrence time of $g_n^{-1}x$ to $\Xi^u_n$ by $g_n^{-1}$ (see \Cref{lem 3.9}). Fortunately, the integral of this recurrence function can be bounded (see \Cref{lem 3.11}), and thus this term.

Now, we start the proof. Note that when $x\notin\Xi^u_n$, we have $[\tA\vee\mathbf {\Xi^u_n}](x)=\tA^u_n(x)$, and thus
\[
[\tA\vee g_n\xi^u_n\vee\mathbf{\Xi^u_n}](x)\subseteq\tA^u_n(x).
\]
Hence, when $x\notin g^{-1}_n\Xi^u_n$,\[
[g^{-1}_n\tA^u_n](x)\supseteq [g^{-1}_n\tA\vee\xi^u_n\vee g^{-1}_n\mathbf{\Xi^u_n}](x).
\]
Together with the fact that $\mathbf {\Xi^u_n}\prec\xi^u_n$, it follows that

\begin{equation}\label{eq split}
\begin{split}
&H_{\mu_n|_{\Xi^u_{n,\infty}}}(g^{-1}_n\tA^u_n\mid g^{-1}_n\tA\vee\xi^u_n\vee g^{-1}_n\mathbf{\Xi^u_n}) \\[2mm]
=&\int_{g^{-1}_n\Xi^u_n}-\log\mu_{n,\ [g^{-1}_n\tA\vee\xi^u_n\vee g^{-1}_n\mathbf{\Xi^u_n}](x)}([g^{-1}_n\tA^u_n](x))\,\td\mu_n(x)\\[2mm]
\ls&\int_{\Xi^u_n}-\log\mu_{n,\ [g^{-1}_n\tA\vee\xi^u_n\vee g^{-1}_n\mathbf{\Xi^u_n}](x)}([g^{-1}_n\tA^u_n](x))\,\td\mu_n(x)\\[2mm]
&+\int_{g^{-1}_n\Xi^u_n\setminus\Xi^u_n}-\log\mu_{n,\ [g^{-1}_n\tA\vee\xi^u_n\vee g^{-1}_n\mathbf{\Xi^u_n}](x)}([g^{-1}_n\tA^u_n](x))\,\td\mu_n(x)\\[2mm]
:=&\mathrm{(I)} +\mathrm{(II)}.
\end{split}
\end{equation}
We estimate these two terms (I) and (II) one by one.

We first estimate the first term (I) in \Cref{eq split}. Since $\mathbf {\Xi^u_n}\prec\xi^u_n$, for any $x\in\Xi^u_n$,
\[
[g^{-1}_n\tA\vee\xi^u_n\vee g^{-1}_n\mathbf{\Xi^u_n}](x)\subseteq\Xi^u_n.
\]
In particular, any atom of $g^{-1}_n\tA\vee\xi^u_n\vee g^{-1}_n\mathbf{\Xi^u_n}$ that contains a point of $\Xi^u_n$ is contained in $\Xi^u_n$. Then, we establish the following inclusion relation:
\begin{lemma}\label{lem 3.8}
For any $x\in\Xi^u_n$, there exists $z\in M$ such that 
\[[g^{-1}_n\tA^u_n](z) \csupset [g^{-1}_n\tA\vee\xi^u_n\vee g^{-1}_n\mathbf{\Xi^u_n}](x).\]
\begin{proof}
Consider the partition $\tA^u_n$ restricted to the set
\[
g_n([g^{-1}_n\tA\vee\xi^u_n\vee g^{-1}_n\mathbf{\Xi^u_n}](x))=[\tA\vee g_n\xi^u_n\vee \mathbf{\Xi^u_n}](g_nx).
\]
Since $\tA^u_n$ is a partition, we only need to prove that the partition $\tA^u_n$ has only one element that intersects the set $[\tA\vee g_n\xi^u_n\vee \mathbf{\Xi^u_n}](g_nx)$ disregarding $\partial\tA$.
\begin{itemize}
    \item[1.] If $[\tA\vee g_n\xi^u_n\vee \mathbf{\Xi^u_n}](g_nx)\cap(\Xi^u_n)^c\neq\emptyset$, then $$[\tA\vee g_n\xi^u_n\vee \mathbf{\Xi^u_n}](g_nx)\subseteq(\Xi^u_n)^c.$$ Thus, by the definition of $\tA^u_n$ we have $$[\tA\vee g_n\xi^u_n\vee \mathbf{\Xi^u_n}](g_nx)\subseteq (\Xi^u_n)^c\cap\tA(g_nx)=\tA^u_n(g_nx).$$
    Thus, there is only one element of $\tA^u_n$. \smallskip 
    \item[2.] If $[\tA\vee g_n\xi^u_n\vee \mathbf{\Xi^u_n}](g_nx)\subseteq\Xi^u_n$, recall that the construction of $\tA^u_n$ guarantees a minimum $u$-distance that exceeds $\frac{\epsilon_0}{10}$ between any distinct local unstable manifolds inside any element of $\tA$. By this uniform geometric separation, we can ensure that the set $[\tA\vee g_n\xi^u_n\vee \mathbf{\Xi^u_n}](g_nx)$ intersects only a single element of $\tA^u_n$ by showing that
    \[
    \diam^u([\tA\vee g_n\xi^u_n\vee \mathbf{\Xi^u_n}](g_nx))<\frac{\epsilon_0}{10}.
    \]
    Since $x\in\Xi^u_n$, $\tA^u_n(x)$ falls into Case~1  or Case~3 in the definition of $\tA^u_n$. When it falls into Case~3, we have $\tA^u_n(x)=\{x\}$ and thus $[g_n\xi^u_n](g_nx)=\{g_nx\}$. Then the claim follows. When it falls into Case~1, $\tA^u_n(x)\subseteq W^u_{\mathrm{loc},n}(x,\tA(x))$.  Recall that \[
    \diam^u\left(W^u_{\mathrm{loc},n}(x,\tA(x))\right)\ls\epsilon_1,
    \]
    we have\[
    \diam^u\left(g_n W^u_{\mathrm{loc},n}(x,\tA(x))\right)\ls\sup_{n}\{||g_n\|_{C^1}\}\cdot\epsilon_1<\frac{\epsilon_0}{10}.
    \]
    Together with the fact that \[
     [g_n\xi^u_n](g_nx) \subseteq g_n W^u_{\mathrm{loc},n}(x,\tA(x)),
    \]
    the proof of the claim is complete.
\end{itemize}
\end{proof}
\end{lemma}
Thus, for the first term (I), it follows that 
\[
\int_{\Xi^u_n}-\log\mu_{n,\ [g^{-1}_n\tA\vee\xi^u_n\vee g^{-1}_n\mathbf{\Xi^u_n}](x)}([g^{-1}_n\tA^u_n](x))\,\td\mu_n(x)=0.
\]
 
Now we estimate the second term (II). When $x\in g^{-1}_n\Xi^u_n\setminus\Xi^u_n$, since $\mathbf{\Xi^u_n}\prec\xi^u_n$,\[
[g^{-1}_n\tA\vee\xi^u_n\vee g^{-1}_n\mathbf{\Xi^u_n}](x)\subseteq g^{-1}_n\Xi^u_n\setminus\Xi^u_n.
\]
In particular, any element of $g^{-1}_n\tA\vee\xi^u_n\vee g^{-1}_n\mathbf{\Xi^u_n}$ that contains a point in $g^{-1}_n\Xi^u_n\setminus\Xi^u_n$ is contained entirely in $g^{-1}_n\Xi^u_n\setminus\Xi^u_n$.
For $\mu_n-$a.e. $x\in {\Xi^u_{n,\infty}}$, we denote $t(x)$ as the following recurrence time function induced by $g_n^{-1}$:
$$t(x):=\min\{t\in\mathbb N:g_n^{-t}x\in\Xi^u_n\}.$$
Then we can prove the following lemma, which controls the size of $\xi^u_n(x)$ by the  recurrence time function.
\begin{lemma}\label{lem 3.9}
    For $\mu_n-$a.e. $x\in {\Xi^u_{n,\infty}}$,  there exists $\{D_i\}_{i=1}^{\sup_n\{\|g_n\|_{C^1}\}^{t(x)\cdot \rk}}$ such that
    \begin{itemize}
        \item $D_i\subseteq W^{\mathrm{GPu}}(x),$\smallskip
        \item $\xi^u_n(x)\subseteq\bigcup_iD_i,$\smallskip
        \item $\diam^u(D_i)\ls\epsilon_1.$
    \end{itemize}
\begin{proof}
    We can assume that the orbit of $x$ induced by $g_n$ never intersects $\partial\tA$. Since $g^{-t(x)}_nx\in\Xi^u_n\setminus\partial\tA$, $$\xi^u_n(x)\subseteq g^{t(x)}_n(W^u_{\mathrm{loc},n}(g^{-t(x)}_nx)).$$
    By definition, there exists a differential map $\sigma:[0,1]^\rk\to W^{\mathrm{GPu}}(g^{-t(x)}_nx)$ such that $\|\td\sigma\|<\frac{\epsilon_1}{100}$ and $$W^u_{\mathrm{loc},n}(g^{-t(x)}_nx,\tA(g^{-t(x)}_nx))\subseteq \mathrm{Im}(\sigma).$$
    Moreover, there exists a family of affine contractions $$\{\sigma_i:[0,1]^\rk\to[0,1]^\rk\}_{i=1}^{\sup_n\{\|g_n\|_{C^1}\}^{t(x)\cdot \rk}}$$ each with contraction rate $\frac{1}{\sup_n\{\|g_n\|_{C^1}\}^{t(x)}}$ such that $$ [0,1]^\rk\subseteq\bigcup_{i=1}^{\sup_n\{\|g_n\|_{C^1}\}^{t(x)\cdot \rk}}\mathrm{Im}(\sigma_i).$$
    Note that $$\|\td( g_n^{t(x)}\circ\sigma\circ\sigma_i)\|\ls \sup_n\{\|g_n\|_{C^1}\}^{t(x)}\cdot\frac{\epsilon_1}{100}\cdot\frac{1}{\sup_n\{\|g_n\|_{C^1}\}^{t(x)}}=\frac{\epsilon_1}{100},$$
    one can check that $D_i:=\mathrm{Im}(g_n^{t(x)}\circ\sigma\circ\sigma_i)$ satisfy the requirements of the lemma. Thus, the proof is complete.
    
\end{proof}
\end{lemma}

By the same proof in \Cref{lem 3.8}, we can prove the following lemma:
\begin{lemma}\label{lem 3.10}
For $\mu_n-$a.e. $x\in g^{-1}_n\Xi^u_n\setminus\Xi^u_n$ and any $D_i$ given in \Cref{lem 3.9}, there exists $z\in M$ such that 
\[[g^{-1}_n\tA^u_n](z) \csupset [g^{-1}_n\tA\vee\xi^u_n\vee g^{-1}_n\mathbf{\Xi^u_n}](x)\cap D_i.\]
\end{lemma}
\begin{remark}\label{remark 6}
    \Cref{lem 3.9} and \Cref{lem 3.10} imply that for $\mu_n-$a.e. $x\in g^{-1}_n\Xi^u_n\setminus\Xi^u_n$,\[
    \#\{ g^{-1}_n\tA^u_n \mid_{[\xi^u_n\vee g^{-1}_n\mathbf{\Xi^u_n}](x)}\}\ls \sup_n\{\|g_n\|_{C^1}\}^{t(x)\cdot \rk}.
    \]
\end{remark}

We propose a measure-theoretic fact about $t(x)$:
\begin{lemma}\label{lem 3.11}
    $$\int_{g^{-1}_n\Xi^u_n\setminus\Xi^u_n}t(x)\,\td\mu_n(x)\ls \mu_n((\Xi^u_n)^c).$$
    \begin{proof}
        Denote $$T_{p,1}:=\{x\in g^{-1}_n\Xi^u_n\setminus\Xi^u_n: t(x)=p\}$$
        and $$T_{p,q}:=g^{-q+1}_n(T_{p,1}),\quad\forall 1\ls q\ls p.$$

        One can directly check that $T_{p,q}\subseteq (\Xi^u_n)^c$ and that they are pairwise disjoint. Thus, by a direct calculation,
        \begin{align*}
            \int_{g^{-1}_n\Xi^u_n\setminus\Xi^u_n}t(x)\,\td\mu_n(x)&=\sum_{p=1}^\infty p\cdot \mu_n(T_{p,1})\\[2mm]
            &=\sum_{p=1}^\infty\sum_{q=1}^p \mu_n(T_{p,q})\\[2mm]
            &\ls \mu_n((\Xi^u_n)^c).
        \end{align*}
    \end{proof}
\end{lemma}
Now we can estimate the second term (II) in \Cref{eq split} by \Cref{remark 6} and \Cref{lem 3.11} as follows:
\begin{align*}
    &\int_{g^{-1}_n\Xi^u_n\setminus\Xi^u_n}-\log\mu_{n,\ [g^{-1}_n\tA\vee\xi^u_n\vee g^{-1}_n\mathbf{\Xi^u_n}](x)}([g^{-1}_n\tA^u_n](x))\,\td\mu_n(x)\\[2mm]
    \ls& \int_{g^{-1}_n\Xi^u_n\setminus\Xi^u_n}  \log\sup_n\{\|g_n\|_{C^1}\}^{t(x)\cdot \rk}\,\td\mu_n(x)\\[2mm]
    =&  \log\sup_n\{\|g_n\|_{C^1}\}\cdot \int_{g^{-1}_n\Xi^u_n\setminus\Xi^u_n}{t(x)\cdot \rk}\,\td\mu_n(x)\\[2mm]
    \ls&\log\sup_n\{\|g_n\|_{C^1}\}\cdot \rk\cdot\mu_n((\Xi^u_n)^c)\\[2mm]
    \ls&m\td\cdot\log\Upsilon_1\cdot\mu_n((\Lambda^\rk_n)^c).
\end{align*}

Thus, combining all previous discussions, the error term is estimated by

\begin{align*}
    &\mu_n\left((\Xi^u_n)^c\right)\cdot(\log\#\tA+m\td\log\Upsilon_1) + \log 2 + H_{\mu_n|_{\Xi^u_{n,\infty}}}(g_n^{-1}\tA^u_n\mid g_n^{-1}\tA\vee\xi^u_n\vee g_n^{-1}\mathbf{\Xi^u_n})\\[2mm]
    &\ls\mu_n((\Lambda^\rk_n)^c)\cdot(3m\td\cdot\log\Upsilon_1+\log C_0)+\log 2.
\end{align*}
Together with inequality~(\ref{entropy split}), it follows that 
\[
 h^\rk_{\mu_n}(g_n)\ls H_{\mu_n|_{\Xi^u_n}}(g_n^{-1}\tA\mid\xi^u_n)+\mu_n((\Lambda^\rk_n)^c)\cdot(3m\td\cdot\log\Upsilon_1+\log C_0)+\log 2.
\]
Now we take $\limsup_{n\to\infty}$ on both sides. By the upper semi-continuity estimate of the main term at the end of \Cref{sec 3.3}, it follows that
\[
\limsup_{n\to\infty} h^\rk_{\mu_n}(g_n)\ls h^\rk_{\mu}(g)+\limsup_{n\to\infty}\mu_n((\Lambda^\rk_n)^c)\cdot(3m\td\cdot\log\Upsilon_1+\log C_0)+\log 2.
\]
Recall that $g_n=f_n^m$ and $g=f^m$, we have\[
\limsup_{n\to\infty} h^\rk_{\mu_n}(f_n)\ls h^\rk_{\mu}(f)+\limsup_{n\to\infty}\mu_n((\Lambda^\rk_n)^c)\cdot3\td\cdot\log\Upsilon_1+\frac{\log C_0+\log 2}{m}.
\]
Letting $m\to\infty$,
\[
\limsup_{n\to\infty} h^\rk_{\mu_n}(f_n)\ls h^\rk_{\mu}(f)+\limsup_{n\to\infty}\mu_n((\Lambda^\rk_n)^c)\cdot3\td\cdot\log\Upsilon_1.
\]
Combining this inequality with \Cref{uniform size prop}, we complete the proof of \Cref{thm 3.3}. 

Moreover, combining this inequality with the quantitative version of the uniform size estimate of the weighted Pesin blocks, which is \Cref{quant uniform size prop}, the following quantitative version of \Cref{thm 3.3} follows.
\begin{theorem}\label{thm 3.4}
      Let $f_n\in\mathrm{Diff^{1+\alpha}}(M)$ and $\mu_n$ be an ergodic measure of $f_n$. Assume 
    \begin{itemize}
    \item $f_n\xrightarrow{C^1} f\in \mathrm{Diff}^{1+\alpha}(M)$
    with $\sup_{n\in \mathbb N}\{\|f_n\|_{C^{1+\alpha}}\}< \infty$, 
    \item $\mu_n\xrightarrow{*}\mu$,
    \end{itemize}    
and denote $$C_f:={\frac{10^9\cdot\dim M}{\alpha}\cdot(\log\|f\|_{C^1}+1)^4}.$$ If there exists $\rk\in\mathbb N$ such that $\mathrm{Hyp}^{\mathrm{k}}(\mu)>0$,
then $h^\rk_{\mu_n}(f_n)$ is well-defined when $n$ is large enough and
\[
\limsup_{n\to\infty} h^\rk_{\mu_n}(f_n)\ls h^\rk_{\mu}(f)+\frac{C_f}{\mathrm{Hyp}^{\mathrm{k}}(\mu)^4}\cdot\limsup_{n\to\infty}(\sum_{i=1}^\rk\lambda_i(\mu,f)-\sum_{i=1}^\rk\lambda_i(\mu_n,f_n)).
\] 
whenever the error term 
$$\frac{C_f}{\mathrm{Hyp}^{\mathrm{k}}(\mu)^4}\cdot\limsup_{n\to\infty} (\sum_{i=1}^\rk\lambda_i(\mu,f)-\sum_{i=1}^\rk\lambda_i(\mu_n,f_n))\ls \td\cdot\log\|f\|_{C^1}.$$

\end{theorem}

\subsection{Passage to invariant measures}\label{sec 3.5}

In this section, we establish Theorem \ref{thm best}, which generalizes Theorem \ref{thm 3.3} by relaxing the assumption of ergodicity for measures $\mu_n$. The primary technique for this extension involves the theory of joinings and the passage lemma from \cite{DMsymbolicextension}.

Recall that for any $f$ invariant measure $m$, we define the $\rk$-dimensional partial entropy $\tilde h^\rk_{m}(f)$ by its ergodic decomposition as follows:
$$\tilde h^\rk_{m}(f):=\int_{\mathrm{Hyp}^{\mathrm{k}}(m_x)>0} h^\rk_{m_x}(f)\,\td m(x).$$

We first introduce the definition of joining as in \cite[Section 4]{DMsymbolicextension}.
\begin{definition}[Joining]
    For two measures $\Omega_1,\ \Omega_2\in\mathcal{M}(\mathcal M(X))$, where $\mathcal{M}(X)$ denotes the space of probability measures of $M$, $J$ is called a joining of $\Omega_1$ and $\Omega_2$ if it is a probability measure on $\mathcal M(X)\times\mathcal M(X)$  satisfying
    $$(\pi_1)_* J=\Omega_1,\ (\pi_2)_* J=\Omega_2,$$
    where $\pi_i$ denotes the projection to its $i$-coordinate.
\end{definition}

 Next, we introduce the following passage lemma from \cite[Corollary 4.27]{DMsymbolicextension}. The $d_{\mathcal{M}}$ here denotes the metric on the space of Borel probability measures over $M$, which is compatible with the weak$^*$ topology.
 \begin{lemma}\label{lem:joining}
     Let $\Omega_{\mu_n}$ be the ergodic decomposition measure of $\mu_n$ supported by $\mathcal M_e(M,f_n)$, and by choosing a subsequence we can assume $\Omega_{\mu_n}$ converging to a measure $\Omega$ supported by $\mathcal{M}(M,f)$ such that $$\int_{\mathcal M(M,f)} m\,\td\Omega(m)=\mu.$$ Then for any $\epsilon>0$ there exists an $n(\epsilon)\in\mathbb N$, such that for any $n\gs n(\epsilon)$ we have $d_{C^1}(f,f_n)<\epsilon$, and there exists a joining $J^n$ of $\Omega_{\mu_n}$ and $\Omega$ such that $J^n(\Delta_e^\epsilon)>1-\epsilon$, where
     \[
     \Delta_e^\epsilon=\{(\nu,\tau)\in\mathcal M_e(M,f_n)\times\mathcal M(M,f):d_\mathcal{M}(\nu,\tau)<\epsilon\}. 
     \]
 \end{lemma}
Let $\Upsilon_{1+\alpha}=\sup_{n\in \mathbb N}\{\|f_n\|_{C^{1+\alpha}}\}$. Fix $\gamma>0$. Since $\mathrm{Hyp}^{\mathrm{k}}(\mu)>0$, by the definitions one has $$\mathrm{Hyp}^{\mathrm{k}}(\tau)\gs\mathrm{Hyp}^{\mathrm{k}}(\mu)>0,\ \text{ for }\,\Omega-a.e.\ \tau\in\mathcal M(M,f).$$ Applying \Cref{thm 3.4} to $\tau$ and $f$, 
together with the upper semi-continuity property of the sum of Lyapunov exponents, we can obtain the neighborhoods $U_f(\tau)$ and $V_{\tau}$ as shown in the following lemma:
\begin{lemma}
     For any $\gamma>0$ and $\Upsilon_{1+\alpha}>0$ there exist a $C^1$ neighborhood $U_f(\tau)$ of $f$ and a neighborhood under the weak$^*$ topology $V_{\tau}$ of $\tau$ such that for any $f_\nu\in U_f(\tau)$ satisfying $\|f_\nu\|_{C^{1+\alpha}}\ls\Upsilon_{1+\alpha}$, $\nu\in V_{\tau}\cap\mathcal M_e(M,f_\nu)$ 
 and $$C_f:={\frac{10^9\cdot\dim M}{\alpha}\cdot(\log\|f\|_{C^1}+1)^4},$$ 
we have
\begin{itemize}
    \item $\sum_{i=1}^\rk\lambda_i(\tau,f)-\sum_{i=1}^\rk\lambda_i(\nu,f_\nu)+\gamma>0,$
    \item $$\tilde h^\rk_{\nu}(f_\nu)\ls h^\rk_{\tau}(f)+\frac{C_f}{\mathrm{Hyp}^{\mathrm{k}}(\tau)^4}\cdot(\sum_{i=1}^\rk\lambda_i(\tau,f)-\sum_{i=1}^\rk\lambda_i(\nu,f_\nu)+\gamma).$$
\end{itemize} 
\end{lemma}

Choose a sufficiently small $\epsilon_2<\gamma^2$ such that there exists a set $G_1\subseteq\mathcal M(M,f)$ of $\Omega$ measure larger than $1-\gamma$, and for any $\tau\in G_1$, the diameters of $U_f(\tau)$ and $V_{\tau}$ are larger than $\epsilon_2$. Then we apply \Cref{lem:joining} to $\epsilon_2$ and obtain $n(\epsilon_2)$.  Thus, when $n\gs n(\epsilon_2)$, $f_n\in U_f(\tau)$ for any $\tau\in G_1$.   

Let $J^n_\tau$ denote the conditional probability measure of $J^n$ with $\tau$ fixed on the second coordinate. For $\Omega-$a.e. $\tau\in\mathcal M(M,f)$, $J^n_\tau$ is supported on the ergodic measures of $f_n$. Moreover, there exists a set $G_2\subseteq \mathcal M(M,f)$ of $\Omega$ measure at least $1-\sqrt{\epsilon_2}>1-\gamma$, and for any $\tau\in G_2$, $J^n_\tau$ is up to $\sqrt{\epsilon_2}<\gamma$ supported by the $\epsilon_2$-neighborhood of $\tau$ for any $n\gs n(\epsilon_2)$. Thus, $J^n_\tau$ is up to $\gamma$ supported by $V_{\tau}$ when $\tau\in G:= G_1\cap G_2$.

By previous discussions and direct calculations, for any $n\gs n(\epsilon_2)$, one has
\[
J^n\left(\{(\nu,\tau)\in\mathcal M_e(M,f_n)\times\mathcal M(M,f):\tau\in G,\nu\in V_{\tau}\}\right)\gs 1-3\gamma.
\]
Then we can calculate that
\begin{align*}
    &\tilde h^\rk_{\mu_n}(f_n)=\int\tilde h^\rk_{\nu}(f_n)\,\td J^n(\nu,\tau)\\[2mm]
    =&\int_{\tau\in G,\nu\in V_\tau} \tilde h^\rk_{\nu}(f_n)\,\td J^n(\nu,\tau)+\int_{else}\tilde h^\rk_{\nu}(f_n)\,\td J^n(\nu,\tau)\\[2mm]
    \ls&\int_{\tau\in G,\nu\in V_\tau}\tilde h^\rk_{\nu}(f_n)\,\td J^n(\nu,\tau)+ 3\gamma \cdot\rk \log\Upsilon_1\\[2mm]
    \ls&\int_{\tau\in G,\nu\in V_\tau}\left(h^\rk_{\tau}(f)+\frac{C_f}{\mathrm{Hyp}^{\mathrm{k}}(\tau)^4}\cdot(\sum_{i=1}^\rk\lambda_i(\tau,f)-\sum_{i=1}^\rk\lambda_i(\nu,f_n)+\gamma)\right)\,\td J^n(\nu,\tau)+3\gamma \cdot \rk\log\Upsilon_1.\\[2mm]
    \ls&\int_{\tau\in G,\nu\in V_\tau}\left(h^\rk_{\tau}(f)+\frac{C_f}{\mathrm{Hyp}^{\mathrm{k}}(\mu)^4}\cdot(\sum_{i=1}^\rk\lambda_i(\tau,f)-\sum_{i=1}^\rk\lambda_i(\nu,f_n)+\gamma)\right)\,\td J^n(\nu,\tau)+3\gamma \cdot \rk\log\Upsilon_1.
\end{align*}
Note that we have
\[
\int_{\tau\in G,\nu\in V_\tau}h^\rk_{\tau}(f)\,\td J^n(\nu,\tau)\ls\int h^\rk_{\tau}(f)\,\td J^n(\nu,\tau)=h^\rk_{\mu}(f),
\]
and calculate that
\begin{align*}
    &\int_{\tau\in G,\nu\in V_\tau}\frac{C_f}{\mathrm{Hyp}^{\mathrm{k}}(\mu)^4}\cdot(\sum_{i=1}^\rk\lambda_i(\tau,f)-\sum_{i=1}^\rk\lambda_i(\nu,f_n)+\gamma)\,\td J^n(\nu,\tau)\\[2mm]
    \ls&\int\frac{C_f}{\mathrm{Hyp}^{\mathrm{k}}(\mu)^4}\cdot(\sum_{i=1}^\rk\lambda_i(\tau,f)-\sum_{i=1}^\rk\lambda_i(\nu,f_n)+\gamma)\,\td J^n(\nu,\tau)+3\gamma\cdot \frac{C_f\cdot 3\rk\log\Upsilon_1}{\mathrm{Hyp}^{\mathrm{k}}(\mu)^4}\\[2mm]
    =&\frac{C_f}{\mathrm{Hyp}^{\mathrm{k}}(\mu)^4}\cdot\left( \sum_{i=1}^\rk\lambda_i(\mu,f)-\sum_{i=1}^\rk\lambda_i(\mu_n,f_n)+\gamma\cdot(9 \rk\log\Upsilon_1+1)  \right).
\end{align*}
Together, for any $\gamma>0$, we have
\begin{align*}
\tilde h^\rk_{\mu_n}(f_n)\ls& h^\rk_{\mu}(f)+\frac{C_f}{\mathrm{Hyp}^{\mathrm{k}}(\mu)^4}\cdot\left( \sum_{i=1}^\rk\lambda_i(\mu,f)-\sum_{i=1}^\rk\lambda_i(\mu_n,f_n)+\gamma\cdot(9 \rk\log\Upsilon_1+1)\right)\\&+3\gamma\cdot\rk \log\Upsilon_1
\end{align*}
when $n$ is large enough. Finally, letting $\gamma\to 0$ and $n\to\infty$, the proof of \Cref{thm best} is complete. 
\begin{remark}
 By carefully checking the proof, it follows that the $\rk$-dimensional partial entropy function $\tilde h^\rk_{\mu}(f)$ and $\tilde h^\rk_{\mu_n}(f_n)$ in \Cref{thm bestcor} and  \Cref{thm best}  can be replaced by the following upper extended $\rk$-dimensional partial entropy function $\hat h^\rk_{\mu}(f)$ and $\hat h^\rk_{\mu_n}(f_n)$ defined by:
    \[
    \hat h^\rk_{m}(f):=\int_{\mathrm{Hyp}^{\mathrm{k}}(m_x)>0} h^\rk_{m_x}(f)\,\td m(x)+\int_{\mathrm{Hyp}^{\mathrm{k}}(m_x)\ls0} \log \|f\|_{C^1} \,\td m(x).
    \]
    The original proof is still valid and will lead to the upper extended partial entropy version of \Cref{thm bestcor} and \Cref{thm best}, which are stronger than the original versions. We leave the details to the reader.
\end{remark}

In what follows, we present the proof of \Cref{USC of entropy}, which is related to the upper semi-continuity of the metric entropy. 
\begin{proof}[Proof of Corollary \ref{USC of entropy}] 
     The proof is similar to the proof of \Cref{thm best}. Let $\Omega_{\mu_n}$ and $\Omega$ be as above. Let $\Upsilon_{1+\alpha}=\sup_{n\in \mathbb N}\{\|f_n\|_{C^{1+\alpha}}\}$.  Since $$\essinf_{x\sim\mu}\lambda_u(x,f)>0\gs \esssup_{x\sim\mu}\lambda_{u+1}(x,f),$$ by the definitions one has for $\Omega-a.e.\ \tau\in\mathcal M(M,f)$,     $$\essinf_{x\sim\tau}\lambda_u(x,f)\gs\essinf_{x\sim\mu}\lambda_u(x,f)>0\gs \esssup_{x\sim\mu}\lambda_{u+1}(x,f)\gs\esssup_{x\sim\tau}\lambda_{u+1}(x,f).$$ Thus we have $\mathrm{Hyp}^{u}(\mu)=\essinf_{x\sim\mu}\lambda_u(x,f)$ and $\mathrm{Hyp}^{u}(\tau)=\essinf_{x\sim\tau}\lambda_u(x,f).$ Applying \Cref{thm 3.3} to $\tau$ and $f$, 
     we can obtain the neighborhoods $U_f(\tau)$ and $V_{\tau}$ as shown in the following lemma:
\begin{lemma}\label{lem entropy passage}
     For any $0<\gamma_0\ll \mathrm{Hyp}^{u}(\mu)$ and $\Upsilon_{1+\alpha}>0$, there exist $0<\gamma_1<\gamma_0$, a $C^1$ neighborhood $U_f(\tau)$ of $f$ and a neighborhood under the weak$^*$ topology $V_{\tau}$ of $\tau$ such that for any $f_\nu\in U_f(\tau)$ satisfying $\|f_\nu\|_{C^{1+\alpha}}\ls\Upsilon_{1+\alpha}$, $\nu\in V_{\tau}\cap\mathcal M_e(M,f_\nu)$, 
    if $$|\sum_{i=1}^u\lambda_i(\nu,f_\nu)-\sum_{i=1}^u\lambda_i(\tau,f)|\ls \gamma_1,$$
    then
$$h^u_{\nu}(f_\nu)\ls h^u_{\tau}(f)+\gamma_0.$$

\end{lemma}

By further shrinking $U_f(\tau)$ and $V_\tau$, and by the upper semi-continuity of the sum of Lyapunov exponents, without loss of generality we can assume that for $\gamma_1$ small enough, one has:
\begin{center}
    For any $f_\nu\in U_f(\tau)$ satisfying $\|f_\nu\|_{C^{1+\alpha}}\ls\Upsilon_{1+\alpha}$ and $\nu\in V_{\tau}\cap\mathcal M_e(M,f_\nu)$, 
    \begin{itemize}
        \item[a.] $\sum_{i=1}^u\lambda_i(\nu,f_\nu)-\sum_{i=1}^u\lambda_i(\tau,f)<\gamma_1^2$;\smallskip
        \item[b.] $|\sum_{i=1}^\td\lambda_i(\nu,f_\nu)-\sum_{i=1}^\td\lambda_i(\tau,f)|<\gamma_1$;\smallskip
        \item[c.] $\sum_{i=u+2}^\td\lambda_i(\nu,f_\nu)-\sum_{i=u+2}^\td\lambda_i(\tau,f)>-\gamma_1$,
    \end{itemize}
\end{center}

Choose $0<\gamma_2\ll\gamma_1$ small enough and $\epsilon_2<\gamma_2^2$ small enough such that there exists a set $G_1\subseteq\mathcal M(M,f)$ with $\Omega$ measure larger than $1-\gamma_2$ and for any $\tau\in G_1$, the diameters of $U_f(\tau)$ and $V_{\tau}$ are larger than $\epsilon_2$.  Apply \Cref{lem:joining} and we can obtain $n(\epsilon_2)$ as in the lemma. Thus, when $n\gs n(\epsilon_2)$, $f_n\in U_f(\tau)$.   

Let $J^n_\tau$ denote the conditional probability measure of $J^n$ with $\tau$ fixed on the second coordinate. For $\Omega-$a.e. $\tau\in\mathcal M(M,f)$, $J^n_\tau$ is supported on the ergodic measures of $f_n$. Moreover, there exists a set $G_2\subseteq \mathcal M(M,f)$ of $\Omega$ measure at least $1-\sqrt{\epsilon_2}>1-\gamma_2$, such that for any $\tau\in G_2$, $J^n_\tau$ is up to $\sqrt{\epsilon_2}<\gamma_2$ supported by the $\epsilon_2$-neighborhood of $\tau$ for any $n\gs n(\epsilon_2)$. Thus $J^n_\tau$ is up to $\gamma_2$ supported by $V_{\tau}$ when $\tau\in G_1\cap G_2$.
By previous discussion and a simple calculation, one has for any $n\gs n(\epsilon_2)$, 
\[
J^n\left(\{(\nu,\tau)\in\mathcal M_e(M,f_n)\times\mathcal M(M,f):\tau\in G_1\cap G_2,\nu\in V_{\tau}\}\right)\gs 1-3\gamma_2.
\]

Denote $$A^n_0:=\{(\nu,\tau)\in\mathcal M_e(M,f_n)\times\mathcal M(M,f):\tau\in G_1\cap G_2,\nu\in V_{\tau}\}.$$
By the assumption, there exists $n_1>n(\epsilon_2)$ such that for any $n\gs n_1$, $$\sum_{i=1}^{u}\lambda_i(\mu_n,f_n)>\sum_{i=1}^{u}\lambda_i(\mu,f)-\gamma_2.$$
Thus when $\gamma_2$ is small enough such that $\gamma_2\cdot(1+6\td\cdot\log\Upsilon_{1+\alpha})<\gamma_1^2$, one has
\begin{align*}
    &\int_{A^n_0}\left(\sum_{i=1}^{u}\lambda_i(\nu,f_n)-\sum_{i=1}^{u}\lambda_i(\tau,f)\right)\,\td J^n(\nu,\tau)\\[2mm]
    \gs& \int\left(\sum_{i=1}^{u}\lambda_i(\nu,f_n)-\sum_{i=1}^{u}\lambda_i(\tau,f)\right)\,\td J^n(\nu,\tau)-3\gamma_2\cdot 2\td\log\Upsilon_{1+\alpha}\\[2mm]
    \gs&-\gamma_2\cdot(1+6\td\cdot\log\Upsilon_{1+\alpha})\\[2mm]
    \gs& -\gamma_1^2.
\end{align*}
Together with 
$$\sum_{i=1}^u\lambda_i(\nu,f_n)-\sum_{i=1}^u\lambda_i(\tau,f)<\gamma_1^2,\quad\forall\, (\nu,\tau)\in A_0^n,$$
by a direct calculation, it follows that there exists $A^n_1\subseteq A^n_0$, $J^n(A^n_1)\gs 1-10\gamma_1$ such that for any $(\nu,\tau)\in A^n_1$, 
\[
|\sum_{i=1}^u\lambda_i(\nu,f_n)-\sum_{i=1}^u\lambda_i(\tau,f)|<\gamma_1.
\]
Together with property b, it follows that  for any $(\nu,\tau)\in A^n_1$,
\[
|\sum_{i=u+1}^\td\lambda_i(\nu,f_n)-\sum_{i=u+1}^\td\lambda_i(\tau,f)|< 2\gamma_1.
\]
Together with property c, for any $(\nu,\tau)\in A^n_1$,
\[
\lambda_{u+1}(\nu,f_n)\ls\lambda_{u+1}(\nu,f_n)-\lambda_{u+1}(\tau,f)<3\gamma_1.
\]
By the Ledrappier-Young entropy formula \cite[Theorem C]{LEDRAPPIER_YOUNG_B}, one has
$h^u_{\tau}(f)=h_{\tau}(f)$ and $h_{\nu}(f_n)\ls h^u_\nu(f_n)+\td\cdot3\gamma_1$.
Applying \Cref{lem entropy passage}, one can calculate that for any $n>n_1$,
\begin{align*}
    &h_{\mu_n}(f_n)-h_{\mu}(f)\\[2mm]
    =&\int \left(h_{\nu}(f_n)-h_{\tau}(f)\right)\,\td J^n(\nu,\tau)\\[2mm]
    \ls& \int_{A^n_1}\left(h_{\nu}(f_n)-h_{\tau}(f)\right)\,\td J^n(\nu,\tau)+10\gamma_1\td\cdot\log\Upsilon_{1+\alpha}\\[2mm]
    \ls&  \int_{A^n_1}\left(h^u_{\nu}(f_n)-h^u_{\tau}(f)\right)\,\td J^n(\nu,\tau)+10\gamma_1\td\cdot\log\Upsilon_{1+\alpha}+\td\cdot3\gamma_1\\[2mm]
    \ls& \gamma_0+10\gamma_1\td\cdot\log\Upsilon_{1+\alpha}+\td\cdot3\gamma_1\\[2mm]
    \ls& \gamma_0\cdot(10\td\cdot\log\Upsilon_{1+\alpha}+3\td+1).
\end{align*}
Taking the limit as $\gamma_0 \to 0$ yields \Cref{USC of entropy}, which completes the proof. As a byproduct of this proof, we also obtain \Cref{rem of positive LE sum}. Its derivation is similar and straightforward. The details are left to the reader.

\end{proof}

\subsection{Upper semi-continuity at generic ergodic measures}\label{sec 3.6}
In this section, we complete the proof of \Cref{USC at generic erg} using \Cref{USC of entropy} and \Cref{thm bestcor}.

\begin{proof}[Proof of \Cref{USC at generic erg}]
By previous discussions, one can assume that $\mu$ is ergodic without loss of generality.
Note that by the property of Lyapunov exponents, it is well-known that for any $m\in\mathbb N$, the sum of Lyapunov exponents is upper semi-continuous, that is
\begin{align*}
    \sum_{i=1}^m\lambda_i: \overline{\M_e(M,f)}&\rightarrow \mathbb R\\
    \mu&\rightarrow\sum_{i=1}^m\lambda_i(\mu,f)
\end{align*}
is upper semi-continuous. Thus, there exists a generic subset $A\subseteq\mathcal M_e(M,f)$ in $\overline{\M_e(M,f)}$ such that for any $m$, the map $\sum_{i=1}^m\lambda_i$ is continuous at any $\mu\in A$. It only remains to prove that properties 1 and 2 hold for any $\mu\in A$.

We first claim that for any $\mu\in A$, $\mu_n\in\mathcal M(M,f)$ and $\mu_n\to\mu$, and for any $m\in\mathbb N$, we have
\[
\lim_{n\to\infty}\sum_{i=1}^m\lambda_i(\mu_n,f)=\sum_{i=1}^m\lambda_i(\mu,f).
\]
By the definition of $A$, for any $\epsilon_2>0$ there exists an $0<\epsilon<\epsilon_2$ such that for any ergodic measure $\nu$ satisfying $d(\mu,\nu)<\epsilon$, 
\[
|\sum_{i=1}^m\lambda_i(\mu,f)-\sum_{i=1}^m\lambda_i(\nu,f)|<\epsilon_2,\ \ \forall\, 0<m\ls \td.
\]
Note that $\mu$ is ergodic, and thus the only ergodic decomposition of $\mu$ is itself. By \Cref{lem:joining}, for $0<\epsilon<\epsilon_2$ as above, there exists $n(\epsilon)\in\mathbb N$ such that for any $n\gs n(\epsilon)$, the ergodic decomposition of $\mu_n$ denoted as $\Omega_{\mu_n}$ satisfies $\Omega_{\mu_n}(B(\mu,\epsilon))>1-\epsilon$.
Thus,
\begin{align*}
    |\sum_{i=1}^m\lambda_i(\mu,f)-\sum_{i=1}^m\lambda_i(\mu_n,f)|&\ls |\sum_{i=1}^m\lambda_i(\mu,f)-\int_{\mathcal M_e(M,f)}\sum_{i=1}^m\lambda_i(\nu,f)\,\td\Omega_{\mu_n}(\nu)|\\[2mm]
    &\ls (1-\epsilon)\cdot \epsilon_2+\epsilon\cdot 2\td\log\|f\|_{C^1}.
\end{align*}
Taking $\epsilon_2\to0$, the claim follows. Together with \Cref{USC of entropy} and \Cref{thm bestcor}, it follows that for any $\mu\in A$, properties 1 and 2 hold. Thus, the proof is complete.
    
\end{proof}

\section{Applications and examples}\label{sec example}
In this section, we present the proofs for the applications presented in \Cref{sec: 1.2}.
\subsection{Measures with dominated splitting}$\,$\smallskip\label{subsection:dominated splitting}

\begin{proof}[Proof of \Cref{prop domniated}]
    By the continuity property of the dominated splitting, it follows that for any $x_n\in\supp\mu_n$, $x_0\in\supp\mu$ such that $x_n\to x_0$, one has $E^\rk_n(x_n)\to E^\rk(x_0)$. Moreover, according to the Oseledets Multiplicative Ergodic Theorem, the sum of the first $\rk$ Lyapunov exponents for an invariant measure is given by the integral of the $\rk$-dimensional Jacobian along the corresponding bundle. Thus, it follows that
    \[
    \sum_{i=1}^\rk\lambda_i(\mu_n,f_n)=\int \log \|\wedge^\rk(D_yf_n)|_{E^{\rk}_n(y)}\|\,\td\mu_n(y)
    \]
    and
    \[
    \sum_{i=1}^\rk\lambda_i(\mu,f)=\int \log \|\wedge^\rk(D_yf)|_{E^{\rk}(y)}\|\,\td\mu(y).
    \]
    Together with the continuity property of the bundles, it follows that
    \[
    \lim_{n\to\infty}\sum_{i=1}^\rk\lambda_i(\mu_n,f_n)=\sum_{i=1}^\rk\lambda_i(\mu,f).
    \]
     Since $\text{Hyp}^\rk(\mu) > 0$ is guaranteed by the property of $L$-dominated splitting between $E^\rk$ and $F$, all conditions of \Cref{thm bestcor} are satisfied. Then applying \Cref{thm bestcor}, the proof is complete.
\end{proof}

\subsection{SRB measures}$\,$\smallskip
\begin{proof}[Proof of \Cref{thm of SRB}]
    We complete the proof of \Cref{thm of SRB} by \Cref{USC of entropy} and \Cref{rem of positive LE sum}.
  Applying  \Cref{USC of entropy}, one has
 \[
\limsup_{n\to\infty} h_{\mu_n}(f_n)\ls h_{\mu}(f).
 \]
 By \Cref{rem of positive LE sum}, it follows that
 \[
    \int \sum_{i:\lambda_i(x,f_n)>0}\lambda_i(x,f_n) \,\td\mu_n(x)\xrightarrow{n\to\infty}\int \sum_{i:\lambda_i(x,f)>0}\lambda_i(x,f) \,\td\mu(x).
\]
Since $\mu_n$ is an SRB measure of $f_n$, by the definition one has
\[
h_{\mu_n}(f_n)=\int \sum_{i:\lambda_i(x,f_n)>0}\lambda_i(x,f_n) \,\td\mu_n(x).
\]
Letting $n\to\infty$, we obtain
\[
h_{\mu}(f)\gs \int \sum_{i:\lambda_i(x,f)>0}\lambda_i(x,f) \,\td\mu(x).
\]
However, by the inequality of Margulis and Ruelle, see \cite{1978Ruelle-inequality}, it follows that
\[
h_{\mu}(f)\ls \int \sum_{i:\lambda_i(x,f)>0}\lambda_i(x,f) \,\td\mu(x).
\]
Thus 
\[
h_{\mu}(f)= \int \sum_{i:\lambda_i(x,f)>0}\lambda_i(x,f) \,\td\mu(x),
\]
which implies that $\mu$ is an SRB measure of $f$.
\end{proof}
\subsection{Average expanding diffeomorphisms}$\,$\smallskip

Before presenting the proof, we remark that it is direct to check that $\mathrm{Def}(\cdot)$ is a lower semi-continuous function, i.e., if $f_n \to f$ in the $C^1$ topology, then $$\liminf_{n\to \infty} \mathrm{Def}(f_n)\gs \mathrm{Def}(f).$$ Together with \Cref{Rek.UEAP}, it follows that $\mathrm{Def}(f)>0$ is a $C^2$ open property.

We proceed to introduce the Gibbs $u$-states. An invariant measure is a \emph{Gibbs $u$-state} if its conditional measures along the strong-unstable foliation are equivalent to the Lebesgue measure on the strong-unstable leaves. We denote the Gibbs $u$-states by the notation $\mathrm{Gibbs}^u(f)$. By \cite{bonatti2005dynamics}, $\mathrm{Gibbs}^u(f)$ is a non-empty, convex and compact subset of the invariant probability measure space. As a direct corollary of the Ledrappier-Young entropy formula \cite[Theorem C]{LEDRAPPIER_YOUNG_B}, every SRB measure is a Gibbs $u$-state. 
From \cite{COSY}, we introduce the following results on the center Lyapunov exponents of the Gibbs $u$-states for diffeomorphisms with UEAP.

\begin{corollary}[Corollary 4.1 of \cite{COSY}]\label{Cor. hyperbolic Gibbs} 
Let $f\in \mathrm{PH}^2_{\mathrm{bun}}(M)$ be a diffeomorphism with UEAP and satisfy $\mathrm{Def}(f)>0$. Let $\mu \in \mathrm{Gibbs}^u(f)$, then
$$\essinf_{x\sim\mu} \lambda_{\dim E^u+1}(x,f)\gs C_{U}$$
and 
$$\esssup_{x\sim\mu}\lambda_{\dim E^u+2}(x,f)\ls -\mathrm{Def}(f)<0.$$ In particular, every ergodic Gibbs $u$-state of $f$ is a hyperbolic measure, with exponents uniformly away from zero.
\end{corollary}

\begin{theorem}[Theorem C of \cite{COSY}]\label{Thm. continuty of exponents}
Let $f\in \mathrm{PH}^2_{\mathrm{bun}}(M)$ be a diffeomorphism with UEAP and satisfy $\mathrm{Def}(f)>0$.  Let $\{f_n\}_{n\in\mathbb N} \subseteq \mathrm{PH}^2_{\mathrm{bun}}(M)$ be a sequence of diffeomorphisms converging to $f$ in the $C^2$-topology. Moreover, suppose that a sequence $\mu_n\in\mathrm{Gibbs}^u(f_n)$ converges to some measure $\mu$ in the weak$^*$ topology. Then
$$\lim_{n\to \infty} \lambda_{\dim E^u+1}(\mu_n,f_n) = \lambda_{\dim E^u+1}(\mu,f)$$ and $$\lim_{n\to \infty}\lambda_{\dim E^u+2}(\mu_n,f_n) = \lambda_{\dim E^u+2}(\mu,f).$$
\end{theorem}

As an immediate corollary of \Cref{USC of entropy}, we derive the following proposition concerning the upper semi-continuity of the metric entropy for Gibbs u-states; the proof is omitted.

\begin{proposition}\label{Pro.Gibbs uppersemientropy}
Let $f\in \mathrm{PH}^2_{\mathrm{bun}}(M)$ be a diffeomorphism with UEAP and satisfy $\mathrm{Def}(f)>0$. Let $\{f_n\}_{n\in\mathbb N} \subseteq \mathrm{PH}^2_{\mathrm{bun}}(M)$ be a sequence of diffeomorphisms converging to $f$ in the $C^2$-topology. Moreover, suppose that a sequence $\mu_n\in\mathrm{Gibbs}^u(f_n)$ converges to some measure $\mu$ in the weak$^*$ topology. Then
$$\limsup_{n\to \infty} h_{\mu_n}(f_n)\ls h_{\mu}(f).$$
\end{proposition}


\begin{proof}[Proof of \Cref{Pro.Gibbs uppersemi}]
Since any SRB measure of $f_n$ is a Gibbs $u$-state, its ergodic decomposition are all Gibbs $u$-states. Moreover, since the limit of Gibbs $u$-states remains to be a Gibbs $u$-state (cf. \cite{YANG_expanding_entropy}), it follows that $\mu$ is a Gibbs $u$-state of $f$. By Corollary~\ref{Cor. hyperbolic Gibbs}, the center Lyapunov exponents of $\mu$ satisfy $$\essinf_{x\sim\mu} \lambda_{\dim E^u+1}(x,f)\gs C_{U}$$
and 
$$\esssup_{x\sim\mu}\lambda_{\dim E^u+2}(x,f)\ls -\mathrm{Def}(f)<0.$$
By the Pesin entropy formula for SRB measures, 
$$h_{\mu_n}(f_n)=\int \log \mid\det Df_n\mid_{E^u(x,f_n)} \mid d\mu_n(x) +\lambda_{\dim E^u+1}(\mu_n,f_n).$$
Together with Proposition~\ref{Pro.Gibbs uppersemientropy}, it follows that
\begin{align*}
    h_\mu(f)&\gs \limsup h_{\mu_n}(f_n)=\limsup \int \log \mid\det Df_n\mid_{E^u(x,f_n)} \mid d\mu_n(x) +\lambda_{\dim E^u+1}(\mu_n,f_n)\\[2mm]
    &=\int \log \mid\det Df\mid_{E^u(x,f)} \mid d\mu(x) +\lambda_{\dim E^u+1}(\mu,f).
\end{align*}
The last equality comes from the continuity of Lyapunov exponents in Theorem~\ref{Thm. continuty of exponents} and the fact that the unstable bundle $E^u$ varies continuously with respect to the partially hyperbolic diffeomorphisms.

By Corollary~\ref{Cor. hyperbolic Gibbs} and Ruelle inequality, we have 
$$h_\mu(f)\ls \int \log \mid\det Df\mid_{E^u(x,f)} \mid d\mu(x) +\lambda_{\dim E^u+1}(\mu,f).$$

By the previous discussion, we have proved that $$h_\mu(f)=\int \log \mid\det Df\mid_{E^u(x)} \mid d\mu(x) +\lambda_{\dim E^u+1}(\mu,f),$$ thus by \cite{LEDRAPPIER_YOUNG_A}, $\mu$ is an SRB measure of $f$.

From now on, we suppose $\mu$ is the unique SRB measure of $f$. Since $\mathrm{Def}(f)>0$ is a $C^2$ open property, $f_n$ are diffeomorphisms with UEAP and satisfy $\mathrm{Def} (f_n) > 0$ for $n$ large enough. Without loss of generality, we may assume all the diffeomorphisms $f_n$ satisfy these properties. Thus, by Theorem~\ref{Thm. SRB measures}, $f_n$ admits finitely many SRB measures, which are physical measures and the union of whose basins has full volume.

We prove by contradiction, i.e., we suppose there exists a sequence of diffeomorphisms $f_n \to f$ in the $C^2$ topology, and each $f_n$ admits two distinct ergodic SRB measures $\mu_n$ and $\mu_n^\prime$. By our previous discussion on the continuity of SRB measures, we have $\lim \mu_n=\mu$ and $\lim \mu_n^\prime=\mu$. 

Because SRB measures are Gibbs $u$-states, by Corollary~\ref{Cor. hyperbolic Gibbs}, Theorem~\ref{Thm. continuty of exponents} together with \Cref{uniform size prop} and  \Cref{remark unifrom size} (see also the SPR study in \cite[Section 3]{SPR2025}), for any $\rho>0$, there are Pesin blocks $\Lambda_n$ and $\Lambda$ of $f_n$ and $f$ respectively with uniform parameters, such that $\mu_n(\Lambda_n)>1-\rho$, $\mu(\Lambda)>1-\rho$ and $\limsup \Lambda_n=\Lambda$. 

By passing to a subsequence, we may assume that $\Lambda_n$ converging to a compact set $\Lambda_0\subset \Lambda$ in the Hausdorff topology. By Katok's closing lemma  \cite{katok1980}, there exists a hyperbolic periodic point $p$ of $f$ close to $\Lambda_0$, such that there exists a positive $\mu$-measure set $A\subset\Lambda_0$ which is close to $p$, and for any $q\in A$,
\begin{itemize}
    \item the local stable manifold of $p$ has non-trivial transverse intersection with the local unstable manifold of $q$;\smallskip
    \item and the local unstable manifold of $p$ has non-trivial transverse intersection with the local stable manifold of $q$.
\end{itemize}
 Thus, the measure $\mu$ is homoclinicly related with $\delta_{\mathrm{Orb}_f(p)}$ in the sense of \cite{BCS}, see \cite{BCS} for the precise definition. 

For sufficiently large $n$, we denote $p_n$ as the analytic continuation of $p$ for the perturbed diffeomorphism $f_n$.
Since the Pesin stable manifold and unstable manifold vary continuously with respect to $\Lambda_n$, for $n$ large enough, one has
\begin{itemize}
    \item  the local stable manifold of $p_n$ has non-trivial intersection with local unstable manifolds of points in $\Lambda_n$ which is close to $p_n$;\smallskip
    \item and the local unstable manifold of $p_n$ has non-trivial intersection with local stable manifolds of points in $\Lambda_n$ which is close to $p_n$. 
\end{itemize} 
This implies that $\mu_n$ is homoclinically related to $\delta_{\mathrm{Orb}_{f_n}(p_n)}$.

In the same manner, we can show that for the sequence of the SRB measures $\mu_n^\prime$, there is another hyperbolic periodic point $p^\prime$ of $f$ such that $\mu$ is homoclinically related to $\delta_{\mathrm{Orb}_f(p^\prime)}$, and for sufficiently large $n$, $\mu_n^\prime$ is homoclinically related to $\delta_{\mathrm{Orb}_{f_n}(p_n^\prime)}$, where $p_n^\prime$ is the analytic continuation periodic point of $p^\prime$ for the perturbed diffeomorphism $f_n$.

Since homoclinically related is an equivalent relation (cf. \cite[Section 2]{BCS}) and $\delta_{\mathrm{Orb}_f(p)}$ and $\delta_{\mathrm{Orb}_f(p^\prime)}$ both are homoclinically related to $\mu$, it follows that $\delta_{\mathrm{Orb}_f(p)}$ and $\delta_{\mathrm{Orb}_f(p^\prime)}$ are homoclinically related. From the definition, for atomic measures on periodic orbits, homoclinically related is equivalent to the periodic orbits to be homoclinically related (cf. \cite[Section 2]{BCS}). Thus, $\mathrm{Orb}_f(p)$ and $\mathrm{Orb}_f(p^\prime)$ are homoclinically related. Note that the homoclinic intersection is robust under perturbation. Thus, $\delta_{\mathrm{Orb}_{f_n}(p_n)}$ and $\delta_{\mathrm{Orb}_{f_n}(p_n^\prime)}$ are homoclinically related for sufficiently large $n$, which implies that $\mu_n$ and $\mu_n^\prime$ are homoclinically related. Then by \cite[Theorem 1.6]{HHTUCMP2011}, one has  $\mu_n=\mu_n^\prime$, a contradiction. Thus, the proof of \Cref{Pro.Gibbs uppersemi} is complete.

\end{proof}

\subsection{Lorenz and Lorenz-derived attractors}$\,$\smallskip

Before the proof, we present some general technical preparation for singular flows.

Let $\mu$ be an invariant measure of $\phi^X_t$. Then, by the Oseledets theorem, for $\mu$-almost every $x$, there exists a sequence of numbers $\lambda_1(x,\phi^X_t)>\cdots>\lambda_{r(x)}(x,\phi^X_t)$, which are referred to as Lyapunov exponents, together with the corresponding Oseledets splitting at $x$:
$$T_xM =E_1(x)\oplus \cdots \oplus E_{r(x)}(x).$$ 
For regular measures, zero always appears as one Lyapunov exponent, and the corresponding subbundle of zero exponent in the Oseledets splitting contains the flow direction. 

Now, we introduce the concept of the linear Poincar\'e flow $\psi_t^X$ over nonsingular points, which is defined as follows:
\begin{definition}[Linear Poincar\'e Flow]\label{def:linear poincare flow}
Let $\mathcal{N}: = \bigcup_{x\in M\setminus \operatorname{Sing}(X)} N_x$ be the normal bundle of $\phi_t^X$,
where $N_x$ is the orthogonal complement of the flow direction $X(x)$, that is,
$$
N_x: = \bigl\{ v\in T_xM : v \perp X(x) \bigr\}.
$$

Let $\pi_x^N \colon T_xM \to N_x$ denote the orthogonal projection onto $N_x$.
For any $x\in M\setminus \operatorname{Sing}(X)$ and any vector $v\in N_x$,
the linear Poincar\'e flow $\psi^X_t(v)$ is the orthogonal projection of $\Phi^X_t(v)$ onto $N_{\phi^X_t(x)}$,
where $\Phi^X_t$ denotes the tangent flow of $\phi^X_t$. More precisely,
$$
\psi^X_t(v)
:= \pi^N_{\phi^X_t(x)}(\Phi^X_t(v))
= \Phi^X_t(v)
- \frac{\langle \Phi^X_t(v), X(\phi^X_t(x)) \rangle}{\langle X(\phi^X_t(x)), X(\phi^X_t(x)) \rangle} X(\phi^X_t(x)),
$$
where $\langle \cdot, \cdot \rangle$ stands for the inner product on the tangent spaces induced by the Riemannian metric.
\end{definition}

The following is the flow version of the Oseledets theorem over the normal bundles, see \cite[Theorem 2.12]{PYY-Poincare} for instance.

\begin{theorem}\label{Thm.Oseledets flow and diffeo}
For almost every $x$ of a regular measure $\mu$, there exist $\iota = \iota(x) \in \mathbb{N}$ and real
numbers
$$\hat{\lambda}_1(x,\psi^X_t)>\cdots>\hat{\lambda}_\iota(x,\psi^X_t)$$
and a $\psi^X_t$ invariant measurable splitting on the normal bundle:
$$N_x = \hat{E}_1(x) \oplus \cdots \oplus \hat{E}_\iota(x)=\pi^N_x(E_1(x))\oplus \cdots \oplus \pi_x^N(E_\iota(x)),$$
such that for every non-zero vector $v_i \in \hat{E}_i(x)$,
$$\lim_{t\to \rm \infty}\frac{1}{t} \log \|\psi_t^X(v_i)\|=\hat{\lambda}_i(x,\psi_t^X).$$ 

Moreover, the Lyapunov exponents of $\mu$ (counting multiplicity) for $\psi_t^X$ is the subset of the exponents for $f$ obtained by removing one of the zero exponents which comes from the flow
direction.
\end{theorem}

In the following, we present the proofs of \Cref{Thm. Lorenz attractor} and \Cref{Thm. Derived-Lorenz}.

\begin{proof}[Proof of \Cref{Thm. Lorenz attractor}]
To begin with, we introduce the concept of $s$-entropy. The \emph{$\rk$-dimensional $s$-entropy} of an invariant measure $\mu$ with respect to $f$ means the $\rk$-dimensional partial entropy of $\mu$ with respect to $f^{-1}$. In the same manner, we can define the $s$-entropy for flows by its time-one map. More details and discussions can be found, for example, in \cite{LEDRAPPIER_YOUNG_A,LEDRAPPIER_YOUNG_B}.

By the definition of sectional hyperbolicity \Cref{def sectional hyper}, for any measure $\nu$ supported on $\Lambda$ and $\nu$-almost every $x\in\Lambda$, the Lyapunov exponents in $E^{ss}(x)$ are all negative and uniformly away from zero. 

Now suppose $\nu$ is regular, then for $\nu$-almost every $x\in\Lambda$, the Lyapunov exponents in $F_x^{cu}$ are all non-negative: there is only one exponent along the flow direction which is vanishing, and all the rest are positive. The reason is that for any two-dimensional subspace contained in $F^{cu}$ and containing the flow direction, the exponential rate of volume expanding speed is equal to the sum of the two corresponding two exponents in this subspace. Since the exponent corresponding to the flow direction is zero, by the sectional-expanding condition, it follows that the other exponent is positive. Then by the Ledrappier-Young entropy formula \cite{LEDRAPPIER_YOUNG_A}, the metric entropy of a regular measure $\nu$ is equal to its $\dim(E^{ss})$-dimensional $s$-entropy.

Since both the metric entropy and the $\dim(E^{ss})$-dimensional $s$-entropy of atomic measures supported on singularities are vanishing, and the metric entropy function is affine under the ergodic decomposition, we can show that the metric entropy of any invariant measure $\nu$ coincides with its $\dim(E^{ss})$-dimensional $s$-entropy.

By the Oseledets theorem, the integral of the sum of the $\dim(E^{ss})$ least negative Lyapunov exponents equals $$\int \log \bigl|\det Df\big|_{E^{ss}(x)}\bigr|\,\td\nu(x).$$ Since $E^{ss}\oplus F^{cu}$ is a dominated slitting, this integration varies continuously for invariant measures with respect to weak$^*$ topology as discussed in \Cref{subsection:dominated splitting}. Then by \Cref{thm bestcor}, we conclude that the $\dim(E^{ss})$-dimensional $s$-entropy varies upper semi-continuously, and thus the metric entropy varies upper semi-continuously. This completes the proof. 
\end{proof}

\begin{proof}[Proof of \Cref{Thm. Derived-Lorenz}]

We denote $\mathcal{N}_C$ as the normal bundle of $X$ on $C(\sigma,X)$ as defined in \Cref{def:linear poincare flow}. It was proved in \cite{LYYZ-Derived-from-Lorenze} that for the Lorenz-derived flows $X\in\mathcal{V}$, the normal bundle always exhibits a dominated splitting.
\begin{proposition}[Proposition 3.2 of \cite{LYYZ-Derived-from-Lorenze}] The $3$-dimensional normal bundle $\mathcal{N}_C$ admits a dominated splitting of the linear Poincar\'e flow $\psi_t^X$ as $$\mathcal{N}_C=F\oplus E^c \oplus E^{ss}$$ with three one-dimensional subbundles.
\end{proposition}
Moreover, through the proof of \cite[Proposition 3.3]{LYYZ-Derived-from-Lorenze} in \cite[Section 5]{LYYZ-Derived-from-Lorenze}, we can obtain that the continuous sectional-expanding bundle $E$ in \Cref{Pro. LYYZ} is defined by $$E(x)=X(x)\oplus F(x)$$ on nonsingular points in $C(\sigma,X)$, and $E$ extends continuously to the singularities in $C(\sigma,X)$.

By Theorem~\ref{Thm.Oseledets flow and diffeo}, any regular measure $\nu$ supported on $C(\sigma,X)$ has three Lyapunov exponents
$\hat\lambda_1(x,\psi_t^X)>\hat\lambda_2(x,\psi_t^X)>\hat\lambda_3(x,\psi_t^X)$ for $\nu$ almost every $x$ on the normal bundle $\mathcal{N}_C$ corresponding to the dominated splitting $F\oplus E^c\oplus E^{ss}$.
Once again, by the property of dominated splitting, the exponents are uniformly away from each other. In particular, it follows that
\[
\essinf_{x\sim\nu}\{\hat\lambda_1(x,\psi_t^X)-\hat\lambda_2(x,\psi^X_t)\}>0.
\]

Denote $f$ as the time-one map of $\phi^X_t$. By Theorem~\ref{Thm.Oseledets flow and diffeo} again, the $f$-invariant measure $\nu$ has four Lyapunov exponents, which are $$\hat\lambda_1(x,\psi_t^X)>\hat\lambda_2(x,\psi_t^X)>\hat\lambda_3(x,\psi_t^X)\ \text{ and }\ 0.$$ 
Note that by the Oseledets theorem, for any ergodic regular measure $m$, one has
\begin{equation}\label{eq:integral in flows}
    \int \log|\det(Df\mid_{E(x)})|\,\td m(x)=\int (\hat\lambda_1(x,\psi_t^X)+0)\,\td m(x)=\int \hat\lambda_1(x,\psi_t^X)\,\td m(x)
\end{equation}

Since the bundle $E=X\oplus F$ is sectional-expanding, it follows that $\hat\lambda_1(x,\psi_t^X)$ is positive and uniformly bounded away from zero for any ergodic regular measures by \Cref{eq:integral in flows}. Thus by the ergodic decomposition, $\hat\lambda_1(x,\psi_t^X)$ is positive and uniformly bounded away from zero for almost every point of any regular measure.

Thus, $\hat\lambda_1(x,\psi_t^X)$ is the largest Lyapunov exponents for almost every point of any regular measure, and the corresponding Oseledets subbundle $E_1(x)=F(x)$. 
Noting that the bundle $E=X\oplus F$ is invariant and continuous for the diffeomorphism $f$, together with \Cref{eq:integral in flows}, it follows that 
\[
\lambda_1(\mu,\phi_t^X):=\int \hat\lambda_1(x,\psi_t^X)\,\td\mu(x)
\]
varies continuously with respect to regular measures $\mu$, and is uniformly bounded away from zero and $\lambda_2(\mu,\phi_t^X)$. 
Now, one can check that the conditions in \Cref{thm bestcor} holds for $f_n=f$ and $u=1$, and \Cref{Thm. Derived-Lorenz} is a direct consequence of \Cref{thm bestcor}. Thus, the proof is complete.

\end{proof}

\subsection{Standard maps}$\,$\smallskip
\begin{proof}[Proof of \Cref{prop standard maps h_top}]
        By Katok's horseshoe theorem \cite{katok1980}, the topological entropy map is lower semi-continuous. Since the Lyapunov exponent of the measures of maximal entropy $\lambda_1(\mu_g,g)$ is continuous with respect to $g\in\mathcal U$, one can easily check that the assumption of \Cref{USC of entropy} is satisfied for $\mu_g$, $\mu_{g_n}$ and $u=1$. Since the metric entropy of the measures of maximal entropy is equal to the topological entropy, \Cref{USC of entropy} implies the upper semi-continuity of the topological entropy map, and the proof is complete.
    \end{proof}

\subsection{Symbolic codings of diffeomorphisms}$\,$\smallskip

\begin{proof}[Proof of \Cref{prop: USC of coding system}]
We prove the proposition in the stronger form where the convergence
\(\bar\mu_n\to\bar\mu\) is taken with respect to the weak$^*$ topology of admissible potentials.
 Let $\{\bar{\mu}_n\}_{n\in\mathbb{N}}$ be a sequence of $\sigma$-invariant probability measures in $\mathcal{M}(\Sigma_f, \sigma)$ such that $\bar{\mu}_n \xrightarrow{*} \bar{\mu}$ with respect to the space of admissible potentials $C_{\mathrm{adm}}(\Sigma_f, \mathbb{R})$. Let $\mu_n = \pi_*(\bar{\mu}_n)$ and $\mu = \pi_*(\bar{\mu})$ be the respective pushforward measures on $M$. 
 Since the unstable dimension is constant on each component and there are only
finitely many possible values of \(u\), we decompose the measures according to these
values and apply the following argument componentwise.
 By \Cref{thm: Sarig coding} (3), without loss of generality, we can assume that there exists $u\in\mathbb N$ such that $\dim E^u(\pi(\underline{x}))=u$ for $\bar\mu_n$ and $\bar\mu$ almost every $\underline{x}\in\Sigma_f$. 

We first verify the weak$^*$ convergence of the projected measures. For any Hölder continuous function $\varphi \in C^\alpha(M, \mathbb{R})$, its lift $\phi = \varphi \circ \pi$ belongs to $C_{\mathrm{adm}}(\Sigma_f, \mathbb{R})$ by definition. Thus, we have:
\[
\lim_{n\to\infty} \int_M \varphi\,\td\mu_n = \lim_{n\to\infty} \int_{\Sigma_f} \varphi \circ \pi\,\td\bar{\mu}_n = \int_{\Sigma_f} \varphi \circ \pi\,\td\bar{\mu} = \int_M \varphi \,\td\mu.
\]
Since $C^\alpha(M, \mathbb{R})$ is dense in $C^0(M, \mathbb{R})$ under the $C^0$ norm, it follows that $\mu_n \xrightarrow{*} \mu$ in the standard weak$^*$ topology on $M$.

Next, we establish the convergence of the sum of positive Lyapunov exponents. Let $\varphi^u(x) = \log |\det(Df|_{E^u(x)})|$ be the unstable geometric potential. According to \Cref{thm: Sarig coding} (3), for all $\underline{x} \in \Sigma_f$, the splitting $T_{\pi(\underline{x})}M = E^s \oplus E^u$ is such that $E^u$ corresponds to exponents $\gs \frac{\chi}{2}$, and $E^s$ corresponds to exponents $\ls -\frac{\chi}{2}$. Thus, for any $\pi_*( \bar{\nu})$, the sum of positive Lyapunov exponents is given by $\int_M \varphi^u(x) \, d\nu(x)$. Since $\varphi^u \circ \pi$ is an admissible potential in $C_{\mathrm{adm}}(\Sigma_f, \mathbb{R})$, the convergence $\bar{\mu}_n \to \bar{\mu}$ implies:
\[
\lim_{n\to\infty} \int_M \varphi^u \, \td\mu_n = \lim_{n\to\infty} \int_{\Sigma_f} \varphi^u \circ \pi \, \td\bar{\mu}_n = \int_{\Sigma_f} \varphi^u \circ \pi \, \td\bar{\mu} = \int_M \varphi^u \, \td\mu.
\]
This implies the continuity of the sum of positive Lyapunov exponents:
\[
\sum_{i=1}^u \lambda_i(\mu_n, f) \xrightarrow{n\to \infty} \sum_{i=1}^u \lambda_i(\mu, f).
\]
Note that by \Cref{thm: Sarig coding} (3), each projected measure is $\frac{\chi}{2}$-hyperbolic, ensuring $\mathrm{Hyp}^u(\mu)>0$. Thus, by applying \Cref{USC of entropy} to the sequence $\mu_n \to \mu$ with $f_n = f$, we obtain $$\limsup_{n\to\infty} h_{\mu_n}(f) \ls h_\mu(f).$$ Finally, \Cref{thm: Sarig coding} (2) ensures that the entropy is preserved, i.e., $h_{\bar{\mu}_n}(\sigma) = h_{\mu_n}(f)$ and $h_{\bar{\mu}}(\sigma) = h_\mu(f)$. Thus,
\[
\limsup_{n \to \infty} h_{\bar{\mu}_n}(\sigma) = \limsup_{n \to \infty} h_{\mu_n}(f) \ls h_\mu(f) = h_{\bar{\mu}}(\sigma),
\]
which completes the proof.
\end{proof}

\appendix
\section{Proofs in \Cref{sec:pes}}
This appendix provides detailed proofs for the technical results stated in \Cref{sec:pes}. We sequentially prove \Cref{pro 2.1}, \Cref{lem 2.5}, \Cref{lx norm}, \Cref{lem 2.6}, \Cref{cone pre}, \Cref{pro 2.8.2}, \Cref{lem 2.9}, \Cref{gpu pro}, and \Cref{prop 2.15}.

Before going into details, we note that from \Cref{GPB def 2}, by letting $m\to \infty$, we obtain $\chi\ls \log\|f\|_{C^1}$. Indeed, we have
$$\frac{1}{\|f^{m}\|_{C^1}}\ls\|Df^{-m}|_{E^{u}(x)}\|\ls K\cdot  \e^{-m\chi},$$
hence $\chi\ls \frac{\log K}{m}+\log\|f\|_{C^1}$, and taking $m\to\infty$ yields $\chi\ls \log\|f\|_{C^1}$.
     
From \Cref{GPB def 2} we obtain
$$\frac{1}{\|f\|_{C^1}}\ls\|Df^{-1}|_{E^{u}(f^n(x))}\|\ls K\cdot  \e^{|n|\epsilon}\cdot \e^{-\varphi(f^{n-1}(x))},$$
which implies
\begin{equation}\label{eq phi upper bound}
     \varphi(f^{n-1}(x))\ls\log (K \cdot \e^{|n|\epsilon})+\log\|f\|_{C^1}.
\end{equation}

\subsection{Proofs in \Cref{sec 2.1}}\label{appendix a1}
\begin{proof}[Proof of \Cref{pro 2.1}]\label{proof 2.1}
(a). We only need to show that $x$ satisfies the requirements of $\Lambda^\rk(K,\chi,\epsilon,\varphi)$ whenever there exists a sequence $x_i\in \Lambda^\rk(K,\chi,\epsilon,\varphi)$ with $x_i \to x$. 
By passing to a subsequence, without loss of generality we may assume that the corresponding subspaces $E^{cs}(x_{i})$ and $E^u(x_{i})$ converge to linear subspaces $\hat{E}^{cs}(x)$ and $\hat{E}^u(x)$ in $T_xM$. We will show that $\hat{E}^{cs}(x)$ and $\hat{E}^u(x)$ satisfy \Cref{GPB def 1} and \Cref{GPB def 2}. Fix $m>0$. Since
\begin{align*}
    \|Df^m|_{E^{cs}(x_i)}\| &\ls K \cdot \e^{-m\chi}\cdot \e^{\Phi_m(x_i)}\\[2mm]
    \|Df^{-m}|_{E^u(x_i)}\| &\ls K \cdot \e^{-m\chi}\cdot \e^{\Phi_{-m}(x_i)},
\end{align*}
and using the continuity of $D_{(\cdot)}f^{\pm m}$ and $\varphi(\cdot)$, letting $i \to +\infty$ yields
\begin{align*}
    \|Df^m|_{\hat{E}^{cs}(x)}\| &\ls K \cdot \e^{-m\chi}\cdot \e^{\Phi_m(x)}\\[2mm]
    \|Df^{-m}|_{\hat{E}^u(x)}\| &\ls K \cdot \e^{-m\chi}\cdot \e^{\Phi_{-m}(x)}.
\end{align*}
Now fix $n\in\mathbb Z$ and define $\hat{E}^{cs/u}(f^n(x)):= D_xf^n(\hat{E}^{cs/u}(x))$. Then, using the continuity of $f^n(\cdot)$ and the same argument, we obtain
\begin{align*}
    \|Df^m|_{\hat{E}^{cs}(f^n(x))}\| &\ls K \cdot \e^{-m\chi + |n|\epsilon}\cdot \e^{\Phi_m(f^n(x))}\\[2mm]
    \|Df^{-m}|_{\hat{E}^{u}(f^n(x))}\| &\ls K \cdot \e^{-m\chi + |n|\epsilon}\cdot \e^{\Phi_{-m}(f^n(x))}.
\end{align*}
Thus the estimates hold for all $m\in\mathbb N$ and $n\in\mathbb Z$. Observe that the definition of the weighted Pesin block does not impose any control on the angle between $E^u(\cdot)$ and $E^{cs}(\cdot)$; consequently the angle between $\hat{E}^{cs}(x)$ and $\hat{E}^u(x)$ may be zero. It remains to prove that $T_xM = \hat{E}^{cs}(x) \oplus \hat{E}^u(x)$. We establish this by contradiction, deriving a contradiction from the hyperbolicity of the weighted Pesin block. If $T_xM \not= \hat{E}^{cs}(x) \oplus \hat{E}^u(x)$, since $\dim\hat{E}^{u}(x)+\dim\hat{E}^{cs}(x)=\td$, this forces $\hat{E}^{cs}(x)\cap\hat{E}^u(x)\neq \emptyset$. Then there exists a vector $w\in \hat{E}^{cs}(x)\cap\hat{E}^u(x)$ with $\|w\|_x\neq 0$, and we have
\begin{align*}
    \|w\|_x
    &=\|D_{f^m(x)}f^{-m}(D_xf^mw)\|_x\\[2mm]
    &\ls K \cdot \e^{-m\chi+m\epsilon}\cdot \e^{\Phi_{-m}(f^{m}x)}\cdot\|D_xf^mw\|_{f^m(x)}\\[2mm]
    &\ls K^2 \cdot \e^{-2m\chi+m\epsilon}\cdot\e^{\Phi_{-m}(f^{m}x)+\Phi_m(x)}\cdot\|w\|_x,
\end{align*}
where the first and second inequalities both follow directly from \Cref{GPB def 2} and \Cref{GPB def 1}, and we note that $\Phi_{-m}(f^{m}x)+\Phi_{m}(x)=0$. Hence $\chi \ls m^{-1}\log K + \epsilon$ for all $m\in\mathbb N$. Letting $m\to+\infty$ gives $\chi \ls \epsilon$, which contradicts \Cref{var}. Thus we have proved that $T_xM = \hat{E}^{cs}(x) \oplus \hat{E}^u(x)$.

(b). The uniqueness of the decomposition $T_xM = E^{cs}(x) \oplus E^u(x)$ for any $x\in \Lambda^\rk(K,\chi,\epsilon,\varphi)$ is proved by contradiction. Suppose there exists another splitting $T_xM = F^{cs}(x) \oplus F^u(x)$ satisfying \Cref{GPB def 1} and \Cref{GPB def 2}. Without loss of generality, assume $E^u(x) \neq F^u(x)$. Then we can choose $e_x\in F^u(x)$ with $\|e_x\|_x=1$ such that
$$e_x = a\cdot e^{cs} + b\cdot e^u,$$
where $a,b\neq0$, $e^{cs}\in E^{cs}(x)$ and $e^u\in E^u(x)$ with $\|e^{cs}\|_x=\|e^{u}\|_x=1$. Consequently,
\begin{align*}
    \|D_xf^{-m}e_x\|_{f^{-m}(x)} 
    &=\|a D_xf^{-m}e^{cs}+b D_xf^{-m}e^u\|_{f^{-m}(x)}\\[2mm]
    &\gs |a|\,\|D_xf^{-m}e^{cs}\|_{f^{-m}(x)} - |b|\,\|D_xf^{-m}e^u\|_{f^{-m}(x)}\\[2mm]
    &\gs |a|\,\|Df^m|_{E^{cs}(f^{-m}(x))}\|^{-1} - |b|\,\|Df^{-m}|_{E^u(x)}\|\\[2mm]
    &\gs |a| K^{-1} \e^{m(\chi-\epsilon)} \e^{\Phi_{-m}(x)} - |b| K \e^{-m\chi} \e^{\Phi_{-m}(x)}\\[2mm]
    &= \e^{\Phi_{-m}(x)}\bigl(|a| K^{-1} \e^{m(\chi-\epsilon)} - |b| K \e^{-m\chi}\bigr).
\end{align*}
Note that because $\chi-\epsilon>0$, as $m\to+\infty$ we have
\[
|a| K^{-1} \e^{m(\chi-\epsilon)}\to+\infty \quad \text{and}\quad |b| K \e^{-m\chi}\to 0.
\]
It follows that
\[
\e^{-\Phi_{-m}(x)}\|D_xf^{-m}e_x\|_{f^{-m}(x)} \to +\infty,
\]
and hence $\e^{-\Phi_{-m}(x)}\|Df^{-m}|_{F^u(x)}\| \to +\infty$, which contradicts the definition of $F^u(x)$. Therefore the decomposition $T_xM = E^{cs}(x) \oplus E^u(x)$ is unique.

(c). Finally we prove the continuity of the decomposition. Let $x_n\to x$ with $x_n\in \Lambda^\rk(K,\chi,\epsilon,\varphi)$. By compactness of $\Lambda^\rk(K,\chi,\epsilon,\varphi)$ (established in part (a)), we have $x\in \Lambda^\rk(K,\chi,\epsilon,\varphi)$. The uniqueness of the decomposition $T_xM = E^{cs}(x) \oplus E^u(x)$ then implies $E^{cs}(x_n)\to E^{cs}(x)$ and $E^u(x_n)\to E^u(x)$. Thus the maps $x\mapsto E^{cs}(x)$ and $x\mapsto E^u(x)$ are continuous.
\end{proof}

\begin{lemma}\label{lem 2.3}
    Let $f\in \Diff(M)$ and $\epsilon>0$. If $\angle(v_{1},v_{2})<\epsilon\cdot\|f\|_{C^1}^{-2}$, then for any $x\in M$ and $v_{1},v_{2}\in T_{x}^1M$, 
    $$\frac{\|D_{x}f(v_{1})\|_{f(x)}}{\|D_{x}f(v_{2})\|_{f(x)}}<\e^{\epsilon}.$$
\end{lemma}
\begin{proof}
   For any $x$, $v_{1}$ and $v_{2}$ as above, define $\hat{V}:=\mathrm{Span}\{v_{2}\}$ as the linear subspace spanned by $v_{2}$. Then $T_{x}M=\hat{V}\oplus \hat{V}^{\perp}$, where $\hat{V}^{\perp}$ denotes the orthogonal complement of $\hat{V}$. Thus,
   $$v_{1}=\cos\angle(v_{1},v_{2})\cdot v_{2}+\sin\angle(v_{1},v_{2})\cdot v^{*},$$
   where $v^{*}\in \hat{V}^{\perp}$ and $\|v^{*}\|=1$. Moreover,
\begin{align*}
    \frac{\|D_{x}f(v_{1})\|_{f(x)}}{\|D_{x}f(v_{2})\|_{f(x)}}
    &\ls |\cos\angle(v_{1},v_{2})|+|\sin\angle(v_{1},v_{2})|\cdot \frac{\|D_{x}fv^{*}\|_{f(x)}}{\|D_{x}fv_{2}\|_{f(x)}}\\[2mm]
    &\ls |\cos\angle(v_{1},v_{2})|+|\sin\angle(v_{1},v_{2})| \cdot \sup_{x\in M}\|D_{x}f\| \cdot (\inf_{x\in M} \m(D_{x}f))^{-1}\\[2mm]
    &\ls |\cos \angle(v_{1},v_{2})|+ \|f\|_{C^1}^2\cdot |\sin\angle(v_{1},v_{2})|.
\end{align*} 
    When $\angle(v_{1},v_{2}) \to 0$, $\sin\angle(v_{1},v_{2})\to 0$ and $\cos\angle(v_{1},v_{2}) \to 1$. Then set
    $$\delta_{0}:=\frac{\epsilon}{\|f\|_{C^1}^2}<\arcsin\frac{\e^\epsilon-1}{\|f\|_{C^1}^2},$$
    and for $0\ls\angle(v_{1},v_{2})<\delta_{0}$, it follows that
    $$\frac{\|D_{x}f(v_{1})\|_{f(x)}}{\|D_{x}f(v_{2})\|_{f(x)}}\ls|\cos\angle(v_{1},v_{2})|+ \|f\|_{C^1}^2\cdot |\sin\angle(v_{1},v_{2})|<\e^{\epsilon}.$$ 
\end{proof}
\begin{lemma}\label{lem 2.4}
    Let $f\in \Diff(M)$. For any $x\in M$ and linearly independent vectors $v_{1},v_{2}\in T_{x}^1M$, denote $\angle_{m}:= \angle(D_xf^mv_1, D_xf^mv_2)$ for any $m\gs 0$. Then
    $$\e^{-L}\ls \frac{\angle_{m+1}}{\angle_{m}}\ls \e^{L}.$$
\end{lemma}
\begin{proof}
    The notation $|u\times v|$ denotes the area of the parallelogram spanned by $u$ and $v$. For any $m\gs0$,
\begin{align*}
    \frac{|D_{x}f^{m+1}(v_{1})\times D_{x}f^{m+1}(v_{2})|}{|D_{x}f^{m}(v_{1})\times D_{x}f^{m}(v_{2})|} 
    &=  \frac{|D_{f^m(x)}f(D_{x}f^{m}v_{1})\times D_{f^m(x)}f(D_{x}f^{m}v_{2})|}{|D_{x}f^{m}(v_{1})\times D_{x}f^{m}(v_{2})|}\\[2mm]
    &\ls \sup_{E} |\det(D_{f^m(x)}f|_{E})|\\[2mm]
    &\ls \|D_{f^m(x)}f\|^2 \ls \sup_{x\in M}\|D_{x}f\|^2,
\end{align*}
    where $E$ runs over all $2$-dimensional subspaces of $T_{f^mx}M$. Moreover, by definition,
\begin{align*}
    &\frac{\|D_{x}f^{m+1}(v_{1})\|_{f^{m+1}(x)}}{\|D_{x}f^{m}(v_{1})\|_{f^{m}(x)}}\cdot\frac{\|D_{x}f^{m+1}(v_{2})\|_{f^{m+1}(x)}}{\|D_{x}f^{m}(v_{2})\|_{f^{m}(x)}}\cdot\frac{|\sin\angle_{m+1}|}{|\sin\angle_{m}|} \\[2mm]
    &=\frac{|D_{x}f^{m+1}(v_{1})\times D_{x}f^{m+1}(v_{2})|}{|D_{x}f^{m}(v_{1})\times D_{x}f^{m}(v_{2})|}
    \ls \sup_{x\in M}\|D_{x}f\|^2.
\end{align*}
    Thus,
    $$ \frac{|\sin\angle_{m+1}|}{|\sin\angle_{m}|} \ls \sup_{x\in M}\|D_{x}f\|^2\cdot\inf_{x\in M}\m(D_{x}f)^{-2}\ls \|f\|_{C^1}^4.$$
    Similarly, we obtain the lower bound:
    $$ \frac{|\sin\angle_{m+1}|}{|\sin\angle_{m}|} \gs \inf_{x\in M}\m(D_{x}f)^{2}\cdot\sup_{x\in M}\|D_{x}f\|^{-2}\gs \|f\|_{C^1}^{-4}.$$
    Since $\frac{2x}{\pi}\ls\sin x\ls x$ for $0\ls x\ls \frac{\pi}{2}$ and $L=4(\log \|f\|_{C^1}+1)$, it follows that
    $$\e^{-L}\ls \frac{\angle_{m+1}}{\angle_{m}}\ls \e^{L}.$$
\end{proof}

\subsection{Proofs in \Cref{sec new 2.3}}\label{appendix a2}
\begin{proof}[Proof of \Cref{lem 2.5}]
(1). By directly checking the definition, we have $\|v_u\|_x\ls\|v_u\|'_x$ and $\|v_{cs}\|_x\ls\|v_{cs}\|'_x$. To prove the left‑hand side of the inequality, we first note that
\begin{align*}
    \|v\|'_x
    &=\sqrt{(\|v_{cs}\|'_x)^2+(\|v_u\|'_x)^2}\\[2mm]
    &\gs\frac{\|v_{cs}\|'_x+\|v_u\|'_x}{2}\\[2mm]
    &\gs\frac{\|v_{cs}\|_x+\|v_u\|_x}{2}\\[2mm]
    &\gs\frac{\|v\|_x}{2},
\end{align*}
hence we obtain 
$$\frac{1}{2}\|v\|_x\ls\|v\|'_x.$$
To prove the right‑hand side, we proceed with the following calculation. Firstly, by the definition of $\|\cdot\|'_x$ and \Cref{definition of GPB}:
\begin{align*}
    \|v\|'_x
    &=\sqrt{(\|v_{cs}\|'_x)^2+(\|v_u\|'_x)^2}\\[2mm]
    &\ls\|v_{cs}\|'_x+\|v_u\|'_x\\[2mm]
    &\ls \sum_{n=0}^{+\infty}\e^{n(\chi-2\epsilon)} \e^{-\Phi_{n}(x)}\|D_{x}f^{n}v_{cs}\|_{f^n(x)}+\sum_{n=0}^{+\infty}\e^{n(\chi-2\epsilon)} \e^{-\Phi_{-n}(x)}\|D_{x}f^{-n}v_{u}\|_{f^{-n}(x)}\\[2mm]
    &\ls \sum_{n=0}^{+\infty}\e^{n(\chi-2\epsilon)} \e^{-\Phi_{n}(x)}\|Df^{n}|_{E^{cs}(x)}\|\|v_{cs}\|_x+\sum_{n=0}^{+\infty}\e^{n(\chi-2\epsilon)} \e^{-\Phi_{-n}(x)}\|Df^{-n}|_{E^u(x)}\|\|v_{u}\|_x\\[2mm]
    &\ls K\e^{m\epsilon}\sum_{n=0}^{+\infty}\e^{-2n\epsilon}\|v_{cs}\|_x+K\e^{m\epsilon}\sum_{n=0}^{+\infty}\e^{-2n\epsilon}\|v_{u}\|_x\\[2mm]
    &\ls(1-\e^{-2\epsilon})^{-1}\cdot K\e^{m\epsilon}\cdot(\|v_{cs}\|_x+\|v_{u}\|_x).
    \end{align*}
Now it remains to establish the relation between $\|v_{cs}\|_x+\|v_{u}\|_x$ and $\|v\|_x$, for which we need to use \Cref{pro 2.2}.
\begin{align*}
    \|v_{cs}\|_x+\|v_{u}\|_x&\ls\frac{\sqrt{2}}{\sqrt{1-\cos\angle(v_u,v_{cs})}} \cdot\|v\|_x\\[2mm]
    &\ls \sin(\tfrac{1}{2}\epsilon\|f\|_{C^1}^{-2}K^{-\tau}\e^{-m\tau\epsilon})^{-1}\cdot \|v\|_x\\[2mm]
    &\ls \frac{1}{2}C_1\cdot\frac{\|f\|^{2}_{C^1}}{\epsilon} \cdot K\e^{m\epsilon}\cdot\|v\|_x,
\end{align*}
where $C_1>4\pi$ is a constant large enough. Thus 
$$\|v\|'_{x}\ls  \frac{C_1}{2}\cdot C(\epsilon)\cdot\|f\|^{2}_{C^1}\cdot (K\e^{m\epsilon})^{\tau+1}\|v\|_x.$$
Then the proof of (1) is completed.

(2). The inequality is proved by a direct calculation. We first prove the right‑hand side.
\begin{align*}
    \|D_{x}f^{-1}v_{u}\|'_{f^{-1}(x)}
    &=\sum_{n=0}^{+\infty}\e^{n(\chi-2\epsilon)} \e^{-\Phi_{-n}(f^{-1}(x))}\|D_{f^{-1}(x)}f^{-n}(D_{x}f^{-1}v_{u})\|_{f^{-(n+1)}(x)}\\[2mm]
    &=\sum_{n=0}^{+\infty}\e^{n(\chi-2\epsilon)} \e^{-\Phi_{-n}(f^{-1}(x))}\|D_{x}f^{-(n+1)}v_{u}\|_{f^{-(n+1)}(x)}\\[2mm]
    &=\e^{-(\chi-2\epsilon)} \e^{-\varphi(f^{-1}(x))}\sum_{n=0}^{+\infty}\e^{(n+1)(\chi-2\epsilon)} \e^{-\Phi_{-(n+1)}(x)}\|D_{x}f^{-(n+1)}v_{u}\|_{f^{-(n+1)}(x)}\\[2mm]
    &=\e^{-(\chi-2\epsilon)} \e^{-\varphi(f^{-1}(x))}\sum_{n=1}^{+\infty}\e^{n(\chi-2\epsilon)} \e^{-\Phi_{-n}(x)}\|D_{x}f^{-n}v_{u}\|_{f^{-n}(x)}\\[2mm]
    &\ls \e^{-(\varphi(f^{-1}(x))+\chi-2\epsilon)}\sum_{n=0}^{+\infty}\e^{n(\chi-2\epsilon)} \e^{-\Phi_{-n}(x)}\|D_{x}f^{-n}v_{u}\|_{f^{-n}(x)}\\[2mm]
    &= \e^{-(\varphi(f^{-1}(x))+\chi-2\epsilon)}\|v_{u}\|'_x.
\end{align*}
This completes the proof of the right‑hand side. Next we prove the left‑hand side.
\begin{align*}
    \|v_{u}\|'_x
    &=\sum_{n=0}^{+\infty}\e^{n(\chi-2\epsilon)} \e^{-\Phi_{-n}(x)}\|D_{x}f^{-n}v_{u}\|_{f^{-n}(x)}\\[2mm]
    &=\|v_{u}\|_x+\sum_{n=1}^{+\infty}\e^{n(\chi-2\epsilon)} \e^{-\Phi_{-n}(x)}\|D_{x}f^{-n}v_{u}\|_{f^{-n}(x)}\\[2mm]
    &=\|v_{u}\|_x+\e^{\varphi(f^{-1}(x))+\chi-2\epsilon}\sum_{n=1}^{+\infty}\e^{(n-1)(\chi-2\epsilon)} \e^{-\Phi_{-(n-1)}(f^{-1}(x))}\|D_{f^{-1}(x)}f^{-(n-1)}(D_xf^{-1}v_{u})\|_{f^{-n}(x)}\\[2mm]
    &=\|v_{u}\|_x+\e^{\varphi(f^{-1}(x))+\chi-2\epsilon}\sum_{n=0}^{+\infty}\e^{n(\chi-2\epsilon)} \e^{-\Phi_{-n}(f^{-1}(x))}\|D_{f^{-1}(x)}f^{-n}(D_xf^{-1}v_{u})\|_{f^{-(n+1)}(x)}\\[2mm]
    &\ls \|D_{f^{-1}(x)}f\|\|D_{x}f^{-1}v_{u}\|_{f^{-1}(x)}+\e^{\varphi(f^{-1}(x))+\chi-2\epsilon}\|D_{x}f^{-1}v_{u}\|'_{f^{-1}(x)}\\[2mm]
    &\ls \bigl(\|f\|_{C^1}+\e^{\varphi(f^{-1}(x))+\chi-2\epsilon}\bigr)\cdot\|D_{x}f^{-1}v_{u}\|'_{f^{-1}(x)}.
\end{align*}
Thus the proof of (2) is completed.

(3). The proof is similar to (2). However, the specific results have a slight difference; therefore, for the completeness of the proof, we provide the full calculation to the reader. Since
\begin{align*}
    \|D_{x}fv_{cs}\|'_{f(x)}
    &=\sum_{n=0}^{+\infty}\e^{n(\chi-2\epsilon)} \e^{-\Phi_{n}(f(x))}\|D_{f(x)}f^{n}(D_{x}fv_{cs})\|_{f^{n+1}(x)}\\[2mm]
    &=\sum_{n=0}^{+\infty}\e^{n(\chi-2\epsilon)} \e^{-\Phi_{n}(f(x))}\|D_{x}f^{(n+1)}v_{cs}\|_{f^{n+1}(x)}\\[2mm]
    &=\e^{-(\chi-2\epsilon)} \e^{\varphi(x)}\sum_{n=0}^{+\infty}\e^{(n+1)(\chi-2\epsilon)} \e^{-\Phi_{n+1}(x)}\|D_{x}f^{(n+1)}v_{cs}\|_{f^{n+1}(x)}\\[2mm]
    &=\e^{-(\chi-2\epsilon)} \e^{\varphi(x)}\sum_{n=1}^{+\infty}\e^{n(\chi-2\epsilon)} \e^{-\Phi_{n}(x)}\|D_{x}f^{n}v_{cs}\|_{f^{n}(x)}\\[2mm]
    &\ls \e^{\varphi(x)-(\chi-2\epsilon)}\|v_{cs}\|'_x,
\end{align*}
and moreover,
\begin{align*}
    \|v_{cs}\|'_x
    &=\sum_{n=0}^{+\infty}\e^{n(\chi-2\epsilon)} \e^{-\Phi_{n}(x)}\|D_{x}f^{n}v_{cs}\|_{f^{n}(x)}\\[2mm]
    &=\|v_{cs}\|_x+\sum_{n=1}^{+\infty}\e^{n(\chi-2\epsilon)} \e^{-\Phi_{n}(x)}\|D_{x}f^{n}v_{cs}\|_{f^{n}(x)}\\[2mm]
    &=\|v_{cs}\|_x+\e^{-\varphi(x)+\chi-2\epsilon}\sum_{n=1}^{+\infty}\e^{(n-1)(\chi-2\epsilon)} \e^{-\Phi_{n-1}(f(x))}\|D_{f(x)}f^{(n-1)}(D_xfv_{cs})\|_{f^{n}(x)}\\[2mm]
    &=\|v_{cs}\|_x+\e^{-\varphi(x)+\chi-2\epsilon}\sum_{n=0}^{+\infty}\e^{n(\chi-2\epsilon)} \e^{-\Phi_{n}(f(x))}\|D_{f(x)}f^{n}(D_xfv_{cs})\|_{f^{n+1}(x)}\\[2mm]
    &\ls \|D_{f(x)}f^{-1}\|\|D_{x}fv_{cs}\|_{f(x)}+\e^{-\varphi(x)+\chi-2\epsilon}\|D_{x}fv_{cs}\|'_{f(x)}\\[2mm]
    &\ls \bigl(\|f\|_{C^1}+\e^{\chi-2\epsilon}\bigr)\cdot\|D_{x}fv_{cs}\|'_{f(x)}\\[2mm]
    &\ls 2\|f\|_{C^1}\cdot\|D_{x}fv_{cs}\|'_{f(x)}.
\end{align*}
Thus,  the proof (3) is completed.
\end{proof}

\begin{proof}[Proof of \Cref{lx norm}]
    By definition, if $x\in \Lambda^{*}$ and $\kappa(x)=m$, then
\begin{align*}
    \|L_x\|&=\sup_{v\not=0,\;v\in T_xM}\frac{|L_xv|_x}{\|v\|_x}\\[2mm]
    &\ls\sup_{v\not=0,\;v\in T_xM}\frac{|L_xv_u|_x+|L_xv_{cs}|_x}{\|v\|_x}\\[2mm]
    &=\sup_{v\not=0,\;v\in T_xM}\frac{\|v_u\|'_x+\|v_{cs}\|'_x}{\|v\|_x}\\[2mm]
    &\ls\sup_{v\not=0,\;v\in T_xM}2\frac{\|v\|'_x}{\|v\|_x}.
\end{align*}
   Together with \Cref{lem 2.5}, this implies
\begin{align*}
    \|L_x\|\ls C_1\cdot C(\epsilon)\cdot\|f\|^{2}_{C^1} \cdot l(x)^{\tau+1}.
\end{align*}
   Similarly, we can estimate the lower bound for $\m(L_x)$:
\begin{align*}
    \m(L_x)&=\inf_{v\not=0,\;v\in T_xM}\frac{|L_xv|_x}{\|v\|_x}\\[2mm]
    &\gs\inf_{v\not=0,\;v\in T_xM}\frac{|L_xv_u|_x+|L_xv_{cs}|_x}{2\|v\|_x}\\[2mm]
    &=\inf_{v\not=0,\;v\in T_xM}\frac{\|v_u\|'_x+\|v_{cs}\|'_x}{2\|v\|_x}\\[2mm]
    &\gs\frac{1}{4}.
\end{align*}
\end{proof}

\begin{proof}[Proof of \Cref{lem 2.6}]
Recall $\tilde{f}_x(v)$ and $\tg_x(v)$, which are defined by:
$$\tilde{f}_x(v):= \Psi^{-1}_{f(x)}\circ f\circ\Psi_x(v)$$
and
$$\tg_x(v)
  :=\rho(\frac{\|v\|_x}{r(x)})\cdot\tilde{f}_x(v)+(1-\rho(\frac{\|v\|_x}{r(x)}))\cdot D\tilde{f}_x(0)v,$$
  noting that 
  \begin{align*}
      D\tilde{f}_x(0)&=D_{f(x)}\Psi^{-1}_{f(x)}\cdot D_xf\cdot D_{0}\Psi_x\\[2mm]
      &=L_{f(x)}\cdot D_xf\cdot L_x^{-1}.
  \end{align*}
  
By the definition of the bump function, one can easily check that it suffices to estimate the Lipschitz constant of these functions inside \(R_x(2r(x))\). In fact, 
$\mathrm{Lip}(\tg_x-D\tilde{f}_x(0))$ is bounded by $\|D\tg_x(t)-D\tilde{f}_x(0)\|$, so we just need to estimate $\sup_{t\in R_x(2r(x))}\|D\tg_x(t)-D\tilde{f}_x(0)\|$, by Leibniz rule:
\begin{align*}
\sup_{t\in R_x(2r(x))}\|D\tg_x(t)-D\tilde{f}_x(0)\|
&=\sup_{t\in R_x(2r(x))}\|D\bigl(\rho(\frac{\|t\|_x}{r(x)})\cdot(\tilde{f}_x-D\tilde{f}_x(0)(t)\bigr)\|\\[2mm]
&\ls C_4\frac{1}{r(x)}\|(\tilde{f}_x-D\tilde{f}_x(0))|_{R_{x}(2r(x))}\|_{C^0}\\[2mm]
&\quad+\|(D\tilde{f}_x-D\tilde{f}_x(0))|_{R_{x}(2r(x))}\|_{C^0},
\end{align*}
where $C_4>1$ is a constant depending on the norm of bump function $\|\rho\|_{C^1}$. Since $f\in \mathrm{Diff^{1+\alpha}}(M)$ and by \Cref{psi 1}, thus 
\begin{align*}
    \|(D\tilde{f}_x-D\tilde{f}_x(0))|_{R_{x}(2r(x))}\|_{C^0}\ls C^2_3\cdot C(\epsilon)\cdot\|f\|^{2}_{C^1} \cdot l(f(x))^{\tau+1} \cdot\|f\|_{C^{1+\alpha}}\cdot r(x)^{\alpha},
\end{align*}
where $C_3=10C_1$ is a large constant. Similarly, by the Mean Value Theorem
\begin{align*}
    \|(\tilde{f}_x-D\tilde{f}_x(0))|_{R_{x}(2r(x))}\|_{C^0}&\ls\|(D\tilde{f}_x-D\tilde{f}_x(0))|_{R_{x}(2r(x))}\|_{C^0}\cdot 2r(x) \\[2mm]
    &\ls 2C^2_3\cdot C(\epsilon)\cdot\|f\|^{2}_{C^1} \cdot l(f(x))^{\tau+1} \cdot\|f\|_{C^{1+\alpha}}\cdot r(x)^{\alpha+1}
\end{align*}
Then we have
\begin{align*}
\mathrm{Lip}(\tg_x-D\tilde{f}_x(0))
    &\ls 2C_4\cdot C^2_3\frac{1}{r(x)}\cdot C(\epsilon)\cdot\|f\|^{2}_{C^1} \cdot l(f(x))^{\tau+1} \cdot\|f\|_{C^{1+\alpha}}\cdot r(x)^{\alpha}\cdot r(x)\\[2mm]
    &\quad + C^2_3\cdot C(\epsilon)\cdot\|f\|^{2}_{C^1} \cdot l(f(x))^{\tau+1} \cdot\|f\|_{C^{1+\alpha} }\cdot r(x)^{\alpha}\\[2mm]
    &\ls C_5\cdot C(\epsilon)\cdot \e^{(\tau+1)\epsilon}\cdot\|f\|^{3}_{C^{1+\alpha}} \cdot l(x)^{\tau+1} \cdot r(x)^{\alpha}\\[2mm]
    &\ls \frac{\epsilon}{4\|f\|_{C^1}\cdot l(x)},
\end{align*}
where $C_5=4C_4\cdot C^2_3$ is a large constant. A similar estimate holds for \(\mathrm{Lip}(\tg^{-1}_{x}-D\tilde{f}_{x}^{-1}(0))\):
\begin{align*}
    \mathrm{Lip}(\tg^{-1}_{x}-D\tilde{f}_{x}^{-1}(0))
    &\ls\sup_{t\in R_x(2r(x))}\|D\tg^{-1}_{x}(t)-D\tilde{f}_{x}^{-1}(0)\|\\[2mm]
    & =\sup_{t\in R_x(2r(x))}\|D\tg^{-1}_{x}(t)\cdot\bigl(D\tg_{f^{-1}(x)}(\tg^{-1}_{x}t)-D\tilde{f}_{f^{-1}(x)}(0)\bigr)\cdot D\tilde{f}_{x}^{-1}(0)\|\\[2mm]
    &\ls C^4_3\cdot C(\epsilon)^2\cdot\|f\|_{C^{1}}^6\cdot l(x)^{2\tau+2}\sup_{t\in R_x(2r(x))}\|D\tg_{f^{-1}(x)}(\tg^{-1}_{x}t)-D\tilde{f}_{f^{-1}(x)}(0)\|\\[2mm]
    &\ls C^4_3\cdot C(\epsilon)^2\cdot\|f\|_{C^{1}}^6\cdot l(x)^{2\tau+2}\sup_{p\in R_x(2r(f^{-1}(x)))}\|D\tg_{f^{-1}(x)}(p)-D\tilde{f}_{f^{-1}(x)}(0)\|\\[2mm]
    &\ls C_6\cdot C(\epsilon)^3\cdot\|f\|^{9}_{C^{1+\alpha}} \cdot l(x)^{3\tau+3} \cdot r(x)^{\alpha}\\[2mm]
    &\ls \frac{\epsilon}{4\|f\|_{C^1}\cdot l(x)}.
\end{align*}
where $C_6=C_5\cdot C_3^4$ is a large constant. This completes the proof.
\end{proof}

\begin{proof}[Proof of \Cref{cone pre}]
    (i). Let $v\in \mathrm{Q}_{x}^u(w,\frac{1}{3})$. By \Cref{lem 2.6} and the definition of the box norm $|\cdot|$, 
    \begin{align*}
        |(D\tg_x(w)\cdot v)_{cs}|_{f(x)}&\ls|(D\tilde{f}_x(0)\cdot v)_{cs}|_{f(x)}+\mathrm{Lip}(\tg_x-D\tilde{f}_x(0))\cdot|v_u|_x\\[2mm]
        &\ls |D\tilde{f}_x(0)\cdot v_{cs}|_{f(x)}+\frac{\epsilon}{4\|f\|_{C^1}\cdot l(x)}\cdot|v_u|_x\\[2mm]
        &\ls \e^{\varphi(x)-\chi+2\epsilon} \cdot\frac{1}{3}|v_u|_x+\frac{1}{4}\epsilon\cdot|v_u|_x.
    \end{align*}
    Similarly, we have
    \begin{align*}
        |(D\tg_x(w)\cdot v)_u|_{f(x)}\gs \e^{\varphi(x)+\chi-2\epsilon}\cdot |v_u|_x-\frac{1}{4}\epsilon\cdot|v_u|_x.
    \end{align*}
    Then 
    \begin{align*}
      \frac{|(D\tg_x(w)\cdot v)_{cs}|_{f(x)}}{|(D\tg_x(w)\cdot v)_u|_{f(x)}}
      &\ls\frac{\e^{\varphi(x)-\chi+2\epsilon} \cdot\frac{1}{3}|v_u|_x+\frac{1}{4}\epsilon\cdot|v_u|_x}{\e^{\varphi(x)+\chi-2\epsilon}\cdot |v_u|_x-\frac{1}{4}\epsilon\cdot|v_u|_x} \\[2mm]
      &\ls\frac{\frac{1}{3}\e^{-\chi+2\epsilon}+\frac{1}{4}\epsilon}{\e^{\chi-2\epsilon}-\frac{1}{4}\epsilon} \\[2mm]
      &\ls \frac{1}{3},
\end{align*}
where the second inequality is due to  $\varphi\gs 0$ and the last inequality is due to  $ \chi>100\epsilon$.

Now let $v\in \mathrm{Q}_{x}^{cs}(w,\frac{1}{3})$. Note that by \Cref{eq phi upper bound} we have:
$$\varphi(f^{-1}(x))\ls \log\|f\|_{C^1} +\log l(x).$$
It is similar to get that 
    \begin{align*}
      \frac{|(D\tg_{x}^{-1}(w)\cdot v)_u|_{f^{-1}(x)}}{|(D\tg_{x}^{-1}(w)\cdot v)_{cs}|_{f^{-1}(x)}}
      &\ls\frac{|(D\tilde{f}_{x}^{-1}(0)\cdot v)_u|_{f^{-1}(x)}+\mathrm{Lip}(\tg^{-1}_{x}-D\tilde{f}_{x}^{-1}(0))\cdot|v_{cs}
      |_x}{|(D\tilde{f}_{x}^{-1}(0)\cdot v)_{cs}|_{f^{-1}(x)}-\mathrm{Lip}(\tg^{-1}_{x}-D\tilde{f}_{x}^{-1}(0))\cdot|v_{cs}|_x} \\[2mm]
      &\ls\frac{|D\tilde{f}_{x}^{-1}(0)\cdot v_u|_{f^{-1}(x)}+(4\|f\|_{C^1} \cdot l(x))^{-1}\epsilon\cdot|v_{cs}|_x}{|D\tilde{f}_{x}^{-1}(0)\cdot v_{cs}|_{f^{-1}(x)}-(4\|f\|_{C^1} \cdot l(x))^{-1}\epsilon\cdot|v_{cs}|_x} \\[2mm]
      &\ls\frac{\e^{-(\varphi(f^{-1}(x))+\chi-2\epsilon)} \cdot\frac{1}{3}|v_{cs}|_x+(4\|f\|_{C^1} \cdot l(x))^{-1}\epsilon\cdot|v_{cs}|_x}{\e^{-(\varphi(f^{-1}(x))-\chi+2\epsilon)}\cdot |v_{cs}|_x-(4\|f\|_{C^1} \cdot l(x))^{-1}\epsilon\cdot|v_{cs}|_x} \\[2mm]
      &\ls\frac{\frac{1}{3}\e^{-\chi+2\epsilon}+\frac{1}{4}\epsilon}{\e^{\chi-2\epsilon}-\frac{1}{4}\epsilon} \\[2mm]
      &\ls \frac{1}{3}.
    \end{align*}
    This proves (i).

    (ii). By (i), we have $D\tg_x(w)\cdot v_1\in\mathrm{Q}^u_{{f(x)}}(\tg_xw,\frac{1}{3})$. It follows that
    \begin{align*}
        |D\tg_x(w)\cdot v_1|_{f(x)}=|(D\tg_x(w)\cdot v_1)_u|_{f(x)}&\gs |(D\tilde{f}_x(0)\cdot v_1)_u|_{f(x)}-\frac{1}{4}\epsilon\cdot|(v_1)_u|_x \\[2mm]
        &\gs(\e^{\varphi(x)+\chi-2\epsilon}-\frac{1}{4}\epsilon)\cdot|(v_1)_u|_x\\[2mm]
        &\gs \e^{\varphi(x)+\chi-3\epsilon}\cdot|(v_1)_u|_x\\[2mm]
        &=\e^{\varphi(x)+\chi-3\epsilon}\cdot|v_1|_x.
    \end{align*}
    
   The proof for the center-stable direction is similar. 
   By \Cref{eq phi upper bound} we have
   $$\e^{\varphi(f^{-1}(x))}\ls \|f\|_{C^1}\cdot l(x).$$
   Hence one computes
   $$(4\|f\|_{C^1}\cdot l(x))^{-1}\cdot\epsilon\ls \e^{-\varphi(f^{-1}(x))+\chi-3\epsilon}\cdot(\e^{\epsilon}-1).$$
   Then similarly,
   \begin{align*}
       |D\tg_x^{-1}(w)\cdot v_2|_{f^{-1}(x)}
       &=|(D\tg_x^{-1}(w)\cdot v_2)_{cs}|_{f^{-1}(x)}\\[2mm]
       &\gs\bigl(\e^{-\varphi(f^{-1}(x))+\chi-2\epsilon}-\frac{\epsilon}{4\|f\|_{C^1}\cdot l(x)}\bigr)\cdot|(v_2)_{cs}|_{x}\\[2mm]
       &\gs \e^{-\varphi(f^{-1}(x))+\chi-3\epsilon}\cdot|v_2|_x,
   \end{align*}
   which completes the proof of (ii).
\end{proof}

\begin{proof}[Proof of \Cref{pro 2.8.2}]
    The argument is standard, and we sketch the proof here. Based on the above discussion, fix $x\in \Lambda^*$ and consider the sequence of diffeomorphisms $\{\tg_{f^m(x)}\}_{m\in\mathbb{Z}}$ that act on $\{R_{f^m(x)}\}_{m\in\mathbb{Z}}$.

    By the hyperbolicity property (ii) in Lemma \ref{cone pre}, the Graph Transform Principle and the Contraction Mapping Theorem hold as in the Hadamard–Perron theorem (see, e.g., \cite[Theorem 4.1]{HPS} or \cite[Theorem 6.2.8]{Katokbook}). This allows us to construct the fake foliation families $\{\widetilde{\mathcal{F}}^{i}_{f^m(x)}\}_{m\in\mathbb{Z}}$ in the same manner that, for any $x\in \Lambda^* $ and $y\in R_x$,
     \begin{align*}
     \widetilde{\mathcal{F}}^{u}_{x}(y)=\bigcap_{n=0}^{+\infty}\tilde{g}^{n}\bigl( \tilde{Q}_{f^{-n}(x)}^{u}(\tg^{-n}(y), \frac13)\bigr),\quad  \widetilde{\mathcal{F}}^{cs}_{x}(y)=\bigcap_{n=0}^{+\infty}\tilde{g}^{-n}\bigl( \tilde{Q}_{f^{n}(x)}^{cs}(\tg^{n}(y), \frac13)\bigr)
       \end{align*}
     where $\tilde{g}^{n}$ denotes the composition of $\tilde{g}_{(\cdot)}$ along the forward orbit of $x$, $\tilde{g}^{-n}$ is defined similarly and $\tilde{Q}_{(\cdot)}^{u/s}$ represents the cone $Q_{(\cdot)}^{u/s}$ projected onto $R_{(\cdot)}$ via the exponential map. And for any $y\in R_x$,
      \begin{align*}
     T_y\widetilde{\mathcal{F}}^{u}_{x}(y) \subset  Q_{x}^{u}(y, \frac13),\quad T_y\widetilde{\mathcal{F}}^{cs}_{x}(y) \subset Q_{x}^{cs}(y, \frac13).
    \end{align*}
     Then the almost tangency property (1), the invariance property (2) and the Hyperbolicity (3) follow immediately.
\end{proof}

\subsection{Proofs in \Cref{sec 2.3}}\label{appendix a3}
\begin{proof}[Proof of \Cref{lem 2.9}]
    For any $z\in R_{x}(r(x))\cap\widetilde{\mathcal{F}}_x^{u}(0)$, by \Cref{pro 2.8.2} we know
    $$|z-0|_x=|z_u|_x\ls r(x).$$
    On the other hand, the definition of $r(x)$ gives
\begin{align}\label{rx}
    \e^{-\frac{4\tau}{\alpha}\epsilon}\ls\frac{r(f^{-1}(x))}{r(x)}\ls \e^{\frac{4\tau}{\alpha}\epsilon}.
\end{align}
    Recall that $L= 4\cdot(\log \|f\|_{C^1}+1)$ and $\tau= 2L\cdot\chi^{-1}$. Since 
\begin{align*}
       \epsilon\ls\frac{\alpha\cdot\chi^2}{10^3\cdot(\log\|f\|_{C^1}+1)},   
\end{align*}
    together with \Cref{pro 2.8.2} (3) and recalling that $\varphi\gs 0$,
\begin{align*}
    |\tg_x^{-1} z-0|_{f^{-1}(x)}=|\tg_x^{-1}z_u|_{f^{-1}(x)}
    &\ls \e^{-(\chi-3\epsilon)}\cdot r(x)\\[2mm]
    &\ls \e^{-100\frac{L\epsilon}{\alpha \chi}}\cdot r(x)\\[2mm]
    &\ls \e^{-\frac{4\tau}{\alpha}\epsilon}\cdot r(x)\\[2mm]
    &\ls r(f^{-1}(x)).
\end{align*}
    Together with the invariance property of fake foliations as in \Cref{pro 2.8.2}, the proof is complete.
\end{proof}

\begin{proof}[Proof of \Cref{gpu pro}]
    We first prove 
    $$\bigcup_{n=0}^{+\infty}f^n(W^{\mathrm{GPu}}_{\mathrm{loc}}(f^{-n}(x)))\subseteq W^{\mathrm{GPu}}(x).$$
    For any $y\in\bigcup_{n=0}^{+\infty}f^n(W^{\mathrm{GPu}}_{\mathrm{loc}}(f^{-n}(x)))$ there exists some $m\in\mathbb{N}$ such that $y\in f^m(W^{\mathrm{GPu}}_{\mathrm{loc}}(f^{-m}(x)))$, which means that $$\Psi_{f^{-m}(x)}^{-1}y\in R_{f^{-m}(x)}(r(f^{-m}(x)))\cap\widetilde{\mathcal{F}}_{f^{-m}(x)}^{u}(0).$$ By \Cref{lem 2.9}, for any $q\in \mathbb{N}$ we have 
    $$(\Psi_{f^{-(q+m)}(x)}^{-1}\circ f^{-q}\circ\Psi_{f^{-m}(x)})\circ\Psi_{f^{-m}(x)}^{-1}(y)\in \Psi_{f^{-(q+m)}(x)}^{-1}(W^{\mathrm{GPu}}_{\mathrm{loc}}(f^{-(q+m)}(x))).$$
    By \Cref{cone pre}, we can calculate
    $$|\Psi_{f^{-(q+m)}(x)}^{-1}\circ f^{-(q+m)}(y)-\Psi_{f^{-(q+m)}(x)}^{-1}\circ f^{-(q+m)}(x)|_{f^{-(q+m)}(x)}\ls \e^{-q(\chi-3\epsilon)+\Phi_{-q}(f^{-m}(x))}\cdot r(f^{-m}x).$$
    Together with \Cref{psi 1} we have
    \begin{align*}
    \log d(f^{-(q+m)}(y), f^{-(q+m)}(x))
    &\ls -q(\chi-3\epsilon)+\Phi_{-q}(f^{-m}(x))+\log C_3+\log r(f^{-m}x)\\[2mm]
    &\ls -(q+m)(\chi-3\epsilon)+\Phi_{-(q+m)}(x)+\hat{C}\\[2mm]
    &\ls -\frac{1}{2}(q+m)\chi+\Phi_{-(q+m)}(x)+ \hat{C},
    \end{align*}
    where 
    $$\hat{C}:= \log C_3 +m(\chi-3\epsilon)-\Phi_{-m}(x)+\log r(f^{-m}x)$$
    is a constant independent of $q$.
    Letting $q\to\infty$, we obtain:
    \begin{align*}
        \limsup_{q\to +\infty} \frac{1}{q+m}\log\frac{ d(f^{-(q+m)}(x),f^{-(q+m)}(y))}{\e^{\Phi_{-(q+m)}(x)}}\ls-\frac{1}{2}\chi,
    \end{align*}
    which implies that $y\in W^{\mathrm{GPu}}(x)$ by checking the definition.

    To complete the proof, it remains to show that if $y\in W^{\mathrm{GPu}}(x)$, then $y\in f^N(W^{\mathrm{GPu}}_{\mathrm{loc}}(f^{-N}(x)))$ for some $N\in\mathbb{N}$. By the definition of $W^{\mathrm{GPu}}(x)$, if $y\in W^{\mathrm{GPu}}(x)$, then there exists $N_0\in\mathbb{N}$ such that for all $n\gs N_0$,
    $$ \frac{1}{n}\log\frac{ d(f^{-n}(x),f^{-n}(y))}{\e ^{\Phi_{-n}(x)}}\ls-\frac{1}{4}\chi.$$
    Since $\varphi\gs0$, this implies
    $$ d(f^{-n}(x),f^{-n}(y))\ls \e^{-\frac{1}{4}n\chi}.$$
    Together with \Cref{psi -1} and \Cref{var}, we obtain
    \begin{align*}
     |\Psi^{-1}_{f^{-n}(x)}(f^{-n}(x))-\Psi^{-1}_{f^{-n}(x)}(f^{-n}(y))|_{f^{-n}(x)}
     &\ls C_3\cdot C(\epsilon)\cdot\|f\|^{2}_{C^1} \cdot l(f^{-n}(x))^{\tau+1}\cdot \e^{-\frac{1}{4}n\chi}\\[2mm]
     &\ls C_3\cdot C(\epsilon)\cdot\|f\|^{2}_{C^1} \cdot l(x)^{\tau+1}\cdot \e^{2n\tau\epsilon}\cdot \e^{-25n\frac{L\epsilon}{\alpha\chi}}\\[2mm]
     &= C_3\cdot C(\epsilon)\cdot\|f\|^{2}_{C^1} \cdot l(x)^{\tau+1}\cdot \e^{2n\tau\epsilon}\cdot \e^{-12.5n\frac{\tau\epsilon}{\alpha}}\\[2mm]
     &\ls \hat{B}\cdot \e^{-n\frac{5\tau}{\alpha}\epsilon},
    \end{align*}
    where
    $$\hat{B}:= C_3\cdot C(\epsilon)\cdot\|f\|^{2}_{C^1} \cdot l(x)^{\tau+1}$$
   is a constant independent of $n$. Together with \Cref{rx}, there exists $N\in \mathbb{N}$ such that for all $n\gs N$, $\Psi_{f^{-n}x}^{-1}(f^{-n}(y))$ stays in the chart $R_{f^{-n}(x)}(100^{-1}r(f^{-n}(x)))$. By the definition of local weighted Pesin unstable manifolds, to complete the proof of the lemma we only need to show that $\Psi_{f^{-N}(x)}^{-1}(f^{-N}(y))\in \widetilde{\mathcal{F}}_{f^{-N}(x)}^{u}(0)$.
    
    We finish the proof by contradiction.
     Suppose $\Psi_{f^{-N}(x)}^{-1}(f^{-N}(y))\notin \widetilde{\mathcal{F}}_{f^{-N}(x)}^{u}(0)$. Then we can find a unique point 
     $$\vartheta:=\widetilde{\mathcal{F}}_{f^{-N}(x)}^{u}(0)\cap\widetilde{\mathcal{F}}_{f^{-N}(x)}^{cs}(\Psi_{f^{-N}(x)}^{-1}(f^{-N}(y)))$$ 
     such that $\vartheta\neq\Psi_{f^{-N}(x)}^{-1}(f^{-N}(y))$ and $\vartheta\in R_{f^{-N}(x)}(r(f^{-N}(x)))$. 
    Since both $\vartheta$ and $\Psi_{f^{-N}(x)}^{-1}(f^{-N}(y))$ belong to $\widetilde{\mathcal{F}}_{f^{-N}(x)}^{cs}(\Psi_{f^{-N}(x)}^{-1}(f^{-N}(y)))$, which is the graph of a $C^1$ map $\phi:\mathbb{R}^{cs}\to \mathbb{R}^{\rk}$ with $\|D\phi\| \ls \frac {1}{3}$, we have
    \begin{align*}
    |\vartheta-\Psi_{f^{-N}(x)}^{-1}(f^{-N}(y))|_{f^{-N}(x)}
    &=|\Psi_{f^{-N}(x)}^{-1}(f^{-N}(y))-\vartheta|_{f^{-N}(x)}\\[2mm]
    &=|(\Psi_{f^{-N}(x)}^{-1}(f^{-N}(y))-\vartheta)_{cs}|_{f^{-N}(x)}.  
    \end{align*}
    Thus, by the properties of fake foliations and \Cref{cone pre} (ii), we can compute
    \begin{equation}\label{align325}
    \begin{split}
    &|\tg^{-n}(\vartheta)-\tg^{-n}(\Psi_{f^{-N}(x)}^{-1}(f^{-N}(y)))|_{f^{-n-N}(x)}\\[2mm]
    &\gs \e^{\Phi_{-n}{(f^{-N}(x))}+n(\chi-3\epsilon)}\cdot |\vartheta-\Psi_{f^{-N}(x)}^{-1}(f^{-N}(y))|_{f^{-N}(x)}.
    \end{split}
    \end{equation}    
    On the other hand, 
    since 
    \[
    \vartheta\in \widetilde{\mathcal{F}}_{f^{-N}(x)}^{u}(0)\cap R_{f^{-N}(x)}(r(f^{-N}(x)))=\Psi_{f^{-N}(x)}^{-1}\circ W^{\mathrm{GPu}}_{\mathrm{loc}}(f^{-N}(x)),
    \]
    by the properties of local weighted Pesin unstable manifolds just proved,
    $$\frac{|\tg^{-n}(\vartheta)-0|_{f^{-n-N}(x)}}{\e^{\Phi_{-n}{(f^{-N}(x))}}}\xrightarrow0$$
    exponentially.
    Moreover, since $y\in W^{\mathrm{GPu}}(x)$,
    $$\frac{|\tg^{-n} (\Psi_{f^{-N}(x)}^{-1}(f^{-N}(y)))-0|_{f^{-n-N}(x)}}{\e^{\Phi_{-n}{(f^{-N}(x))}}}\xrightarrow[]{n\to\infty} 0$$
    exponentially. 
    Together we obtain
    $$\frac{|\tg^{-n}(\vartheta)-\tg^{-n}(\Psi_{f^{-N}(x)}^{-1}(f^{-N}(y)))|_{f^{-n-N}(x)}}{\e^{\Phi_{-n}{(f^{-N}(x))}}}\xrightarrow[]{n\to\infty}0.$$
    This contradicts \Cref{align325}, thus completing the proof. 
\end{proof}

\subsection{Proofs in \Cref{sec 2.4}}\label{appendix a4}
\begin{proof}[Proof of \Cref{prop 2.15}]
    Denote $M':=\Gamma_\rk\cap B_\varphi$, and for any $x\in M'$ define
    \begin{align*}
        \Pi(x):=\max\{\Pi^{cs}(x),\ \Pi^{u}(x)\},
    \end{align*}
    where $\Pi^{cs}(x)$ and $\Pi^{u}(x)$ are defined by:
    \begin{align*}
        \Pi^{cs}(x)&:=\sup_{k\gs 0}\bigl\{\|Df_n^{Nk}|_{E^{cs}(x)}\|\cdot e^{-\sum_{j=1}^kN\varphi(f_n^{Nj}(x))}\cdot e^{kN\chi}\bigr\};\\[2mm]
        \Pi^{u}(x)&:=\sup_{k\gs 0}\bigl\{\|Df_n^{-Nk}|_{E^{u}(x)}\|\cdot e^{\sum_{j=1}^k N\varphi(f_n^{-Nj}(x))}\cdot e^{kN\chi}\bigr\}.
    \end{align*}
    Then we have
    $$P_N\cap M'\subset\{x:\Pi(x)\ls \Upsilon_1^{N}\}.$$
    Let $\Pi_\epsilon$ denote the $\epsilon$-tempered envelope of $\Pi$, which is defined by
    $$ \Pi_\epsilon(x) := \sup\{\e^{-|m|\epsilon}\cdot \Pi(f_n^{Nm} (x)) : m \in \mathbb{Z}\}.$$
    See \cite[Section 2.4]{SPR2025} for more details. For every $h\in \mathbb{Z}$ and $m\in \mathbb{N}$,
    \begin{align*}
    \|Df^{Nm}_n|_{E^{cs}(f^{Nh}_n(x))}\| &\ls \Pi_\epsilon(x) \cdot \e^{-mN\chi  + |h|N\epsilon}\cdot \e^{\sum_{j=1}^mN\varphi(f^{Nj}_n(f^{Nh}_nx)) },\\[2mm]
    \|Df^{-Nm}_n|_{E^{u}(f^{Nh}_n(x))}\| &\ls \Pi_\epsilon(x) \cdot \e^{-mN\chi  + |h|N\epsilon}\cdot \e^{\sum_{j=1}^{m}-N\varphi(f^{-Nj}_n(f^{Nh}_nx)) }.
    \end{align*}
    Note that for $\Pi^{cs}$ we have
    $$\Upsilon_1^{-N}\cdot e^{N\chi} \cdot\Pi^{cs}(x)\ls \max\{1, \Upsilon_1^N \cdot e^{N\chi}\cdot\Pi^{cs}(f^N_n(x))\},$$
    and a similar estimate holds for $\Pi^u$. Consequently,
    $$e^{-N\chi-N\log\Upsilon_1}\ls \frac{\Pi(f^N_n(x))}{\Pi(x)}\ls e^{N\chi+N\log\Upsilon_1}.$$ 
    Finally, we introduce the following lemma (see \cite[Lemma 2.20]{SPR2025}), whose  statement is presented here:
\begin{lemma}
Let $T$ be an invertible ergodic probability preserving map on a probability space $(\Omega, \mathcal{F}, \nu)$. Suppose $\Pi : \Omega \to (0, \infty)$ is measurable.   Let $\epsilon > 0$, and define the $\epsilon$-tempered envelope $\Pi_\epsilon$ by
$$
\Pi_\epsilon : \Omega \to (0, \infty], \quad \Pi_\epsilon(x) := \sup\{e^{-|n|\epsilon} \Pi(T^n (x)) : n \in \mathbb{Z}\}.
$$
$\mathrm{(1)}$. If $\log \frac{\Pi \circ T}{\Pi} \in L^1(\nu)$, then $\Pi_\epsilon(x)$ is finite at almost every point $x$.\\[2mm]
$\mathrm{(2)}$. If  $\e^{-c_0} \ls \frac{\Pi(T(x))}{\Pi(x)} \ls \e^{c_0}$ for some constant $c_0 > 0$, then for any $t > 0$,
$$\nu\{x : \Pi_\epsilon(x) > t\} \ls 4 \max\left(\frac{c_0}{\epsilon}, 1\right) \nu\{x : \Pi(x) > t\}.$$
\end{lemma}
By (2) of this lemma,
inequality \eqref{size inequality} follows. This completes the proof. 
\end{proof}

\section{Counterexample for \Cref{uniform size prop} in classical regime}\label{appendix a5}

In this section, we present counterexamples for \Cref{uniform size prop} in the classical regime ($\varphi\equiv \text{constant}$).
 We construct a partially hyperbolic Anosov diffeomorphism $f$ with a $2$-dimensional unstable direction, and a sequence of ergodic measures $\mu_n$ converging to an invariant measure $\mu$, such that for any $K,\epsilon,\chi>0$ and $\varphi\equiv \text{constant}$, at least one of the following holds:
\begin{itemize}
    \item[a.] $\mu_n(\Lambda^1(K,\chi,\epsilon,\varphi))=0$ for any $n$ large enough;\smallskip
    \item[b.] $\mu(\Lambda^1(K,\chi,\epsilon,\varphi))=0$,
\end{itemize}
where $\Lambda^1(K,\chi,\epsilon,\varphi)$ is the classical Pesin block of $f$ (along the strong unstable direction).
This phenomenon implies that it is insufficient to use classical Pesin blocks to study the continuity properties of the partial entropy; thus, it is necessary to employ weighted Pesin blocks as defined in \Cref{definition of GPB}.

Let $f\in\Diff(M)$ be a partially hyperbolic Anosov diffeomorphism with a $2$-dimensional unstable direction, that is, there exists a dominated splitting $$TM=E^u_1\oplus E^u_2\oplus E^s$$ such that $\dim E^u_1=\dim E^u_2=1$, $E^u_1$ and $E^u_2$ are uniformly expanding and $E^s$ is uniformly contracting. Moreover, let $p$ and $q$ be two fixed points satisfying the following properties:
\begin{itemize}
    \item[(i)] $\lambda_1(p),\lambda_2(p)\in[99,100]$;\smallskip
    \item[(ii)] $\lambda_1(q),\lambda_2(q)\in [3,4]$;\smallskip
    \item[(iii)] $p$ and $q$ are homoclinically related, i.e., there exist two transverse intersections $a\in W^u(q)\cap W^s(p)$ and $b\in W^u(p)\cap W^s(q)$.
\end{itemize}

Now, fix any decreasing sequence $\{\alpha_n\}_{n\in \mathbb{N}^*}$ with $\alpha_n\to 0$ as $n\to +\infty$. Without loss of generality, we may assume that $\alpha_1$ is sufficiently small. Then, by the Anosov Closing Lemma \cite[Theorem 6.4.15]{Katokbook}, for each $n\in\mathbb{N}$ there exists $\beta_n:=\beta(\alpha_n)\ls\alpha_n$ such that any periodic $\beta_n$-pseudo-orbit can be $\alpha_n$-shadowed by a unique periodic orbit. See \cite{Katokbook} for the precise definitions.

\begin{figure}[htbp!]
    \centering
    \includegraphics[width=0.7\linewidth]{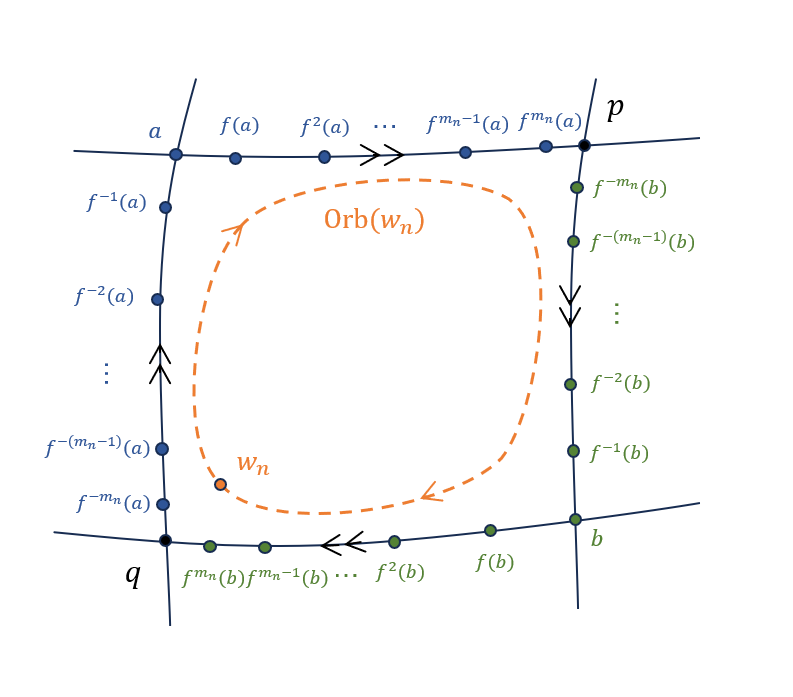}
    \caption{Illustration of the pseudo-orbit construction}
    \label{fig:pseudo_orbit}
\end{figure}

Next, we construct a pseudo-orbit (see \Cref{fig:pseudo_orbit}). For any $n>0$, by the definition of $W^u$ and $W^s$, there exists a positive integer $m_n\gs n$ such that 
\[
d(f^{m_n+1}(a),p)<\tfrac{1}{2}\beta_n,\quad d(f^{-m_n}(a),q)<\tfrac{1}{2}\beta_n,
\]
and 
\[
d(f^{m_n+1}(b),q)<\tfrac{1}{2}\beta_n,\quad d(f^{-m_n}(b),p)<\tfrac{1}{2}\beta_n.
\]
Consequently, the following sequence
\[
L_n:=\bigl(f^{-m_n}(a),\dots,a,\dots,f^{m_n}(a),\,f^{-m_n}(b),\dots,b,\dots,f^{m_n}(b)\bigr)
\]
forms a periodic $\beta_n$-pseudo-orbit. Therefore, there exists a unique $(4m_n+2)$-periodic orbit $\orb(w_n)$ that $\alpha_n$-shadows $L_n$.

Let $\mu_n$ be the periodic ergodic measure supported on $\orb(w_n)$. Then we claim the following:
\begin{itemize}
    \item[(1)] $\mu_n\to \mu:= \frac{1}{2}\delta_p+\frac{1}{2}\delta_q$ as $n\to+\infty$;\smallskip
    \item[(2)] $\lambda_1(\mu_n),\,\lambda_2(\mu_n)\in[10,90]$ for all sufficiently large $n$.
\end{itemize}
\begin{proof}
(1) For any $g\in C(M)$ and any $\epsilon>0$, there exists $\Delta_{g,\epsilon}>0$ such that $|g(x)-g(y)|<\epsilon$ whenever $d(x,y)<\Delta_{g,\epsilon}$. Note that
\[
\int_M g(x)\,\td\mu_n = \frac{1}{4m_n+2}\sum_{i=0}^{4m_n+1} g(f^i(w_n)).
\]
By the definitions of $W^u$ and $W^s$, we can choose $N_\epsilon>0$ so that for all $j>N_\epsilon$,
\[
d(f^{j}(a),p)<\tfrac{1}{2}\Delta_{g,\epsilon},\quad d(f^{-j}(a),q)<\tfrac{1}{2}\Delta_{g,\epsilon}
\]
and
\[
d(f^{j}(b),q)<\tfrac{1}{2}\Delta_{g,\epsilon},\quad d(f^{-j}(b),p)<\tfrac{1}{2}\Delta_{g,\epsilon}.
\]
Moreover, since $\{\alpha_n\}$ decreases to $0$ and $\beta_n\ls\alpha_n$, we can choose $N>0$ such that $\beta_n\ls \frac{1}{2}\Delta_{g,\epsilon}$ for all $n>N$. This implies that among the points $f^i(w_n)$ there are at least $2(m_n-N_\epsilon)$ points in $B(p,\Delta_{g,\epsilon})$ and at least $2(m_n-N_\epsilon)$ points in $B(q,\Delta_{g,\epsilon})$. Consequently,
\begin{align*}
\int_M g(x)\,\td\mu_n
&=\frac{1}{4m_n+2}\sum_{i=0}^{4m_n+1} g(f^i(w_n))\\[2mm]
&\ls \frac{2(m_n-N_\epsilon)}{4m_n+2}\bigl(g(p)+g(q)+2\epsilon\bigr)+\frac{4N_\epsilon+2}{4m_n+2}\,\|g\|_\infty.
\end{align*}
Letting $n\to+\infty$ (hence $m_n\to+\infty$) gives
\[
\limsup_{n\to+\infty}\int_M g(x)\,\td\mu_n \ls \int_M g(x)\,\td\Bigl(\frac{1}{2}\delta_p+\frac{1}{2}\delta_q\Bigr)+\epsilon.
\]
A similar estimate yields the lower bound. Since $\epsilon>0$ is arbitrary, we obtain
\[
\lim_{n\to+\infty}\int_M g(x)\,\td\mu_n = \int_M g(x)\,\td\mu.
\]
As $g$ was arbitrary, then $\mu_n\to\mu$ in the weak$^*$ topology.\smallskip

(2) Because
\[
\lambda_1(\mu_n)=\int_M \log\|Df|_{E^u_1(x)}\|\,\td\mu_n\quad \text{and} \quad 
\lambda_2(\mu_n)=\int_M \log\|Df|_{E^u_2(x)}\|\,\td\mu_n,
\]
and the functions $\log\|Df|_{E^u_{1/2}(x)}\|$ are continuous, statement (1) implies
\[
\lambda_1(\mu_n)\to \frac{1}{2}\lambda_1(p)+\frac{1}{2}\lambda_1(q)\quad \text{and}\quad
\lambda_2(\mu_n)\to \frac{1}{2}\lambda_2(p)+\frac{1}{2}\lambda_2(q).
\]
Since $\frac{1}{2}(100+4)=52$ and $\frac{1}{2}(99+3)=51$, for sufficiently large $n$ we have $\lambda_1(\mu_n),\lambda_2(\mu_n)\in[10,90]$. This establishes (2).
\end{proof}

In this construction, the measures $\mu_n$ and $\mu$ cannot share uniform Pesin blocks of positive measure. 
Specifically, if property (a) fails, it implies that $\varphi \in [10, 90]$. 
Consequently, $p, q \notin \Lambda^1(K, \chi, \epsilon, \varphi)$ for any $K, \chi,$ and $\epsilon$, 
which forces $\mu(\Lambda^1(K, \chi, \epsilon, \varphi)) = 0$. 
This concludes the proof for the example.





\bibliographystyle{abbrv}
{\footnotesize\bibliography{library}}

\end{document}